\documentclass[11pt]{amsart}
\usepackage{tabularx,booktabs}
\usepackage{caption}
\usepackage{amsmath}
\usepackage{amsfonts}

\DeclareMathOperator{\ran}{ran}

\usepackage{amscd}
\usepackage{amsthm}
\usepackage{amssymb} \usepackage{latexsym}
\usepackage{eufrak}
\usepackage{euscript}
\usepackage{epsfig}
\usepackage{graphics}
\usepackage{array}
\usepackage{enumerate}

\usepackage{color}
\usepackage{xcolor}

\usepackage{wasysym}
\usepackage{hyperref}
\usepackage{pdfsync}

\usepackage{graphicx}
\usepackage{float}
\usepackage{subfigure}

\usepackage{bbm}

\usepackage{ulem}
\usepackage{cancel}

\newcommand{\notequiv}{\equiv \hskip -3.4mm \slash}

\newcommand{\bel}[1]{\begin{equation}\label{#1}}

\newcommand{\be}{\begin{equation}}

\newcommand{\ba}{\begin{eqnarray}}
\newcommand{\ea}{\end{eqnarray}}

\newcommand{\qe}{\end{equation}}
\newcommand{\R}{\mathbb{R}}
\newcommand{\N}{\mathbb{N}}
\newcommand{\Z}{\mathbb{Z}}

\DeclareMathOperator\supp{supp}

\newcommand{\Hmm}[1]{\leavevmode{\marginpar{\tiny%
$\hbox to 0mm{\hspace*{-0.5mm}$\leftarrow$\hss}%
\vcenter{\vrule depth 0.1mm height 0.1mm width \the\marginparwidth}%
\hbox to
0mm{\hss$\rightarrow$\hspace*{-0.5mm}}$\\\relax\raggedright #1}}}

\newtheorem{theorem}{Theorem}[section]

\newtheorem{lemma}[theorem]{Lemma}
\newtheorem{corollary}[theorem]{Corollary}
\newtheorem{definition}[theorem]{Definition}

\newtheorem{remark}[theorem]{Remark}

\newtheorem{prop}[theorem]{Proposition}

\newtheorem{example}[theorem]{Example}

\newtheorem*{theorem*}{Theorem}

\newenvironment{thmbis}[1]
    {\renewcommand{\thetheorem}{\ref{#1}$'$}%
     \addtocounter{theorem}{-1}%
     \begin{theorem}}
    {\end{theorem}}

\UseRawInputEncoding

\begin{document}

\title[Discrete Schwarz rearrangement in lattice graphs]{Discrete Schwarz rearrangement on lattice graphs}

\author{Hichem Hajaiej}
\address{Hichem Hajaiej: Department of Mathematics, College of Natural Science Cal State University, 5151 State Drive, 90032 Los Angeles, California, USA}
\email{\href{mailto:hhajaie@calstatela.edu}{hhajaie@calstatela.edu}}

\author{Fengwen Han}
\address{Fengwen Han: School of Mathematics and Statistics, Henan University, 475004 Kaifeng, Henan, China}
\email{\href{mailto:fwhan@outlook.com}{fwhan@outlook.com}}

\author{Bobo Hua}
\address{Bobo Hua: School of Mathematical Sciences, LMNS, Fudan University, Shanghai 200433, China; Shanghai Center for Mathematical Sciences, Fudan University, Shanghai 200433, China}
\email{\href{mailto:bobohua@fudan.edu.cn}{bobohua@fudan.edu.cn}}

\begin{abstract}
In this paper, we prove a discrete version of the generalized Riesz inequality on $\Z^d$. As a consequence, we will derive the extended Hardy-Littlewood and P\'olya-Szeg\"o inequalities. We will also establish cases of equality in the latter. Our approach is totally novel and self-contained. In particular, we invented a definition for the discrete rearrangement in higher dimensions.  Moreover, we show that the definition ``suggested'' by Pruss does not work. We solve a long-standing open question raised by Alexander Pruss in \cite[ p 494]{Pruss98}, Duke Math Journal, and discussed with him in several communications in 2009-2010, \cite{Pruss2010Personal}. Our method also provides a line of attack to prove other discrete rearrangement inequalities and opens the door to the establishment of optimizers of many important discrete functional inequalities in $\Z^d,$ $d\geq2$. We will also discuss some applications of our findings. To the best of our knowledge, our results are the first ones in the literature dealing with discrete rearrangement on $\Z^d,$ $d\geq2$.
\end{abstract}
\maketitle

\tableofcontents

\section{Introduction}

Schwarz rearrangement is the process of transforming an admissible function $u$ into another function $u^*$
that is radial, radially decreasing, and equimeasurable to $u$. A rearrangement inequality is when a certain
functional decreases, or increases, or remains unchanged under the rearrangement process.
Rearrangement inequalities in the continuous setting have attracted many research groups since the
early nineteenth. People were first interested in solving optimization problems like the isoperimetric
inequality, the best constant in some functional inequalities, and some eigenvalue problems, see \cite{polya1951isoperimetric}, \cite{talenti1976best}, to cite only
a few applications. To reach this goal, it has been established that rearrangement decreases the value of
the functionals under study, which implies that the optimizers are Schwarz symmetric functions (radial
and radially decreasing). A nice summary of these aspects can be found in E. H. Lieb and M. Loss
book, Analysis, \cite{LL2001b}, where the authors collected some of the most relevant results along these lines.
The contribution of E. H. Lieb in this topic is impressive, especially to the proof of the Brascamp-Lieb-Luttinger theorem with Brascamp and Luttinger~\cite{brascamp1974general}, the sharp Young inequality with Brascamp~\cite{brascamp1976best}, and the sharp Hardy-Littlewood-Sobolev inequality~\cite{lieb1983sharp}. He has been an inspiring source for generations.
The second wave of highly important contributions dealt with the establishment that the set of
optimizers only has Schwarz symmetric functions. This amounts to studying cases of equality in the
rearrangement inequalities, see \cite{LL2001b}, \cite{Haj2005a}, \cite{burchard2006rearrangement}, to cite only a few contributions. The third
wave of valuable contributions addressed the stability and the quantitative version of these
rearrangement inequalities, see \cite{FMP2008c} that was the first paper in the literature on this highly
important subject. Many valuable contributions were inspired by this breakthrough work. All these
waves gave us outstanding and inspiring contributions published in top journals, and deeply
contributed to the resolution of highly important problems in various domains ranging from analysis
and geometry to physics and engineering. For example, thanks to rearrangement inequalities, we know
that all the optimizers of some functionals are Schwarz symmetric, this reduces the study of the
problems to the space of radial functions instead of the entire space. Therefore, we gain compactness
that is crucial to prove that the minimizing sequences are convergent in some Lebesgue space. This is a
key step in showing the existence of optimizers. The symmetry also reduces PDEs to much simpler ODEs.
This was the main idea used by several colleagues to prove the uniqueness of the solution of some PDEs.
There is a very big number of valuable papers studying various aspects of PDEs using rearrangement
inequalities. 

The situation in the discrete setting is completely different; not that this counterpart is not
extremely interesting, but because the literature has remained quite silent in this field. To the best of
our knowledge, only a handful of papers dealt with the rearrangement in the discrete setting. The two
major contributions are \cite{HH10}, and \cite{Pruss98}. However, both authors were only able to deal with the
one-dimensional setting. The main challenge comes from the difficulty to define the discrete
rearrangement in higher dimension. Because of this major obstacle, things were stagnant for years, and
even the conjecture stated in \cite{Pruss98}, which gave some hope to the community, turned out to be false,
as we will show it in this article. Moreover, in this paper, we will provide a universal definition of
Schwarz rearrangement on lattice graphs. We will then show the P\'olya-Szeg\"{o}, Hardy-Littlewood, and
Riesz inequalities in the discrete setting, as well as the extended versions of these discrete
rearrangement inequalities. The two first rearrangement inequalities are a direct consequence of the
third one as is proved in Theorem \ref{thm:discrete_Riesz_inequality}. Moreover, we will establish cases of equality in all of the above
inequalities, and discuss the optimality of our conditions, see Proposition \ref{prop5.8} and Example \ref{example5.14}. In the
last section, we will present some applications of our findings.
For example, proving that the solutions of discretized partial differential equations have some symmetry properties is relevant for the design of numerical schemes. In particular, it implies that it is sufficient to study them on a small domain instead of the entire space, which cuts down the computational cost (see Section $5$ of \cite{HH10} for more details).

A field where discretization has played an important role is Schr\"{o}dinger equations of the form:
\begin{equation}\label{schrodinger.eq}
  \begin{cases}
    i\partial_t \Phi(t,x)+\Delta\Phi(t,x)+f(|x|,|\Phi|)=0,  \\
    \Phi(0,x)=\Phi_0(x),
  \end{cases}
\end{equation}
where $\Phi_0$ is the initial data, $x\in \R^d,$ $t\in \R.$ \eqref{schrodinger.eq} has attracted the attention of many mathematicians during the last years due to its numerous applications in various domains ranging from biology to quantum mechanics. One of the directions that has particularly sparkled the scientists' interest is the establishment of standing wave solutions to \eqref{schrodinger.eq}, i.e, solutions of the form $\Phi(t,x)=e^{-i\lambda t}u(x),$ where $u$ solves the partial differential equation
\begin{equation}\label{p.d.eq}
  \Delta u+ f(|x|,|u|)+\lambda u=0,
\end{equation}
where $\lambda$ is the Lagrange multiplier.

It turned out that the most stable solutions of equation \eqref{schrodinger.eq} are the solutions of the following constrained variational problem:
\begin{equation}\label{constrained}
  I_c=\inf\left\{\frac{1}{2}\int_{\R^d}|\nabla u|^2 dx-\int_{\R^d}F\left(|x|,|u|\right)dx:\ \int_{\R^d}u^2=c^2\right\}
\end{equation}
where $c\in\R$ is prescribed, and $F(r,s)=\int_{0}^{s}f(r,t)dt.$

In \eqref{constrained}, we minimize the energy
$
E(u)=\frac{1}{2}\int_{\R^d}|\nabla u|^2 dx-\int_{\R^d}F\left(|x|,|u|\right)dx
$ under the sphere $\int_{\R^d}u^2(x)=c^2.$
In order to show that \eqref{constrained} has a solution we need some compactness structure. Usually this is not guaranteed by the Sobolev embeddings. However a good choice of the minimizing sequence to \eqref{constrained} can provide us with the required compactness for the chosen sequence.  To prove that the energy $E$ decreases under Schwarz (symmetric decreasing) symmetrization, the following rearrangement inequalities are crucial:

\begin{equation}\label{polya}
  \int_{\R^d}|\nabla u^*|^2\leq  \int_{\R^d}|\nabla u|^2\quad \text{(P\'{o}lya-Szeg\"{o})}
\end{equation}
\begin{equation}\label{Cavalieri}
   \int_{\R^d}|u^*|^2=\int_{\R^d}|u|^2\quad\text{(Cavalieri's Principle)}
\end{equation}
\begin{equation}\label{F*}
  \int_{\R^d}F\left(|x|,u\right)\leq \int_{\R^d}F\left(|x|,u^*\right).
\end{equation}

Inequalities \eqref{polya}, \eqref{Cavalieri}, and \eqref{F*} guarantee the existence of a sequence that is radial and radially decreasing, which provides compactness to the minimization problem \eqref{constrained}.

For $d=1$, the discrete version of Schwarz rearrangement in $\Z$ has already been studied. We write $\R_+:=[0,+\infty)$ in this paper. Let $u:\Z \to \R_+$ be a function vanishing at infinity, so the function values can be sorted in a decreasing order: $a_1 \ge a_2 \ge \cdots$. The discrete Schwarz rearrangement of $u$ is defined as the unique function $u^*:\Z \to \R_+$ such that
\begin{itemize}
  \item [(i)] For every $x\geq 0$: $u^*(x)\geq u^*(-x)\geq u^*(x+1)\geq \cdots$
  \item [(ii)] For every $t>0:$ $\#\left\{x\in\mathbb{Z}:\ u^*(x)>t\right\}=\#\left\{x\in\mathbb{Z}:\ u(x)>t\right\}.$
\end{itemize}
Explicitly,
\begin{equation*}
    u^*(x)=
    \begin{cases}
    a_{1-2x}, \qquad & x\le 0;\\
    a_{2x}, & x>0.
    \end{cases}
\end{equation*}

Let ``$\prec$'' be a total order in $\Z$ defined via $0 \prec 1 \prec -1 \prec 2 \prec -2 \prec \cdots$, $u^*$ is the unique $\prec$-decreasing function satisfying (ii) of the above. In the early 20th century, Hardy, Littlewood and P\'olya explored the one-dimensional discrete rearrangement, and discrete versions of Hardy-Littlewood and Riesz rearrangement inequalities were proved in chapter X of \cite{hardy1952inequalities}.

The discrete version of \eqref{polya}, \eqref{Cavalieri}, and \eqref{F*} in $\Z$ have been proved for suitable $u$ and $F$:
\begin{equation}\label{equ:discrete_Cavalieri_principle_on_Z}
    \sum_{x\in \Z} |u(x)|^2 = \sum _{x\in \Z} |u^*(x)|^2;
\end{equation}
\begin{equation}\label{equ:discrete_Polya-Szego_inequality_on_Z}
    \sum_{x\in \Z}|u^*(x)-u^*(x+1)|^2 \le \sum_{x\in \Z}|u(x)-u(x+1)|^2;
\end{equation}
\begin{equation}\label{equ:generlized_inverse_Riesz_inequality}
  \sum_{x\in \Z}F\left(|x|,u(x)\right)\le \sum_{x\in \Z}F\left(|x|,u^*(x)\right).
\end{equation}
The discrete version of Cavalieri's Principle~\eqref{equ:discrete_Cavalieri_principle_on_Z} is obtained directly by the definition of $u^*$. By using the idea of discrete polarization, the first author~\cite{HH10} proved \eqref{equ:discrete_Cavalieri_principle_on_Z}, \eqref{equ:discrete_Polya-Szego_inequality_on_Z}, and \eqref{equ:generlized_inverse_Riesz_inequality}. Gupta~\cite{gupta2022symmetrization} also studied some properties of $u^*$ and gave another definition of discrete Schwarz rearrangement in $\Z$ in a recent paper. However, \eqref{equ:generlized_inverse_Riesz_inequality} is violated in his setting.

Pruss \cite{Pruss98} studied the Schwarz rearrangement on general graphs. Given a graph $G=(V,E)$, let $\prec$ be an order on $V$ such that $V$ has a minimal element under $\prec$, i.e. the elements in $V$ form a sequence $x_1 \prec x_2 \prec \cdots$. Given a function $u:V\to \R_+$ with function values $a_1 \ge a_2 \ge a_3 \ge \cdots$, the symmetric decreasing rearrangement (Schwarz rearrangement) of $u$ is defined as the unique $\prec$-decreasing function $u^{\#}$ equimeasurable to $u$, i.e. $u^{\#}(x_i)=a_i$. Given a graph, Pruss tried to find an order such the following discrete P\'{o}lya-Szeg\"{o} inequality holds:
\begin{equation}
    ||\nabla u^{\#}||_2 \le ||\nabla u||_2,
\end{equation}
where $||\nabla u||_2^2=\frac{1}{2}\sum_{x\in V}\sum_{y\sim x} |u(y)-u(x)|^2$ is the discrete analog of the Sobolev energy. He proved that if $G$ is a regular tree, there exists such an order on $V$. ($\Z$ is a 2-tree, the order $\prec$ is defined via $0\prec 1 \prec -1 \prec 2 \prec -2 \prec \cdots$, then $u^{\#}$ coincides with $u^*$ as defined previously). The existence of an order for higher dimensional lattices is an open problem raised by Pruss; see \cite[p494]{Pruss98}. We give a negative answer for $\Z^2$ as stated in the following theorem, the proof is shown in section 3.
\begin{theorem}
    \label{thm:Pruss's_definition_fails_for_Z2}
    There does not exist a total order on $\Z^2$ such that
    $$||\nabla u^{\#}||_2 \le ||\nabla u||_2$$
    holds for all nonnegative functions on $\Z^2$.
\end{theorem}

By the above result, we know that the explicit construction on $\Z$ cannot be extended to higher dimensions ($\mathbb{Z}^d,$ $d>1$), where it is also not reasonable to expect that a sort of Schwarz Symmetrization of $u$ to has the same kind of symmetry as in the one-dimensional case.

Despite its great importance and numerous applications, this crucial extension has remained unresolved for years. The first attempt was the discrete Steiner symmetrization on $\Z^2$ introduced in a pioneering work \cite{shlapentokh2010asymptotic}, aiming for the Faber-Krahn type result for the first Dirichlet eigenvalue. However, the discrete Steiner symmetrization cannot apply to the proofs of the inequalities \eqref{polya}-\eqref{F*} in the discrete setting $\mathbb{Z}^d,$ where $d> 1$. After several investigations and discussions with many experts in the field, it became evident to us that new ideas and an innovative approach are required.

In this work, we will construct a Schwarz rearrangement that preserves the $l^p$ norm and satisfies the analog of all known rearrangement inequalities in the continuous setting. We provide a suitable definition of Schwarz rearrangement of a function defined on $\mathbb{Z}^d$, $d\geq2$, as the limit of a sequence of iterated one-dimensional rearrangements. (see Definition \ref{def:definition_of_discrete_Schwarz_rearrangement}). We then show all known rearrangement inequalities. The main theorem of the paper is stated below, and will be restated and proved in Theorem~\ref{thm:discrete_Riesz_inequality}. Here, we write $R_{\Z^d}$ for the Schwarz rearrangement operator on $\Z^d$, and we say $G$ is ``supermodular'' if $G(s+s_0,t+t_0)+G(s,t) \ge G(s+s_0,t)+G(s,t+t_0)$ for all $s,t,s_0,t_0 \ge 0$ (see Definition~\ref{def:supermodular}). The discrete Hardy-Littlewood and P\'olya-Szeg\"{o} inequalities will follow as corollaries as shown in section 5.

\begin{theorem}[Generalized discrete Riesz inequality]
    \label{thm:main}
    Let $u,v\in C_0^+(\Z^d)$ be admissible, $u^*=R_{\Z^d} u$, $v^*=R_{\Z^d} v$, $H:\R_+ \to \R_+$ be a decreasing function, and $G:\R_+ \times \R_+ \to \R_+$ be a supermodular function with $G(0,0)=0$. Then
    \begin{equation*}
        \sum_{x,y \in \Z^d}G(u(x),v(y))H(d(x,y)) \le \sum_{x,y \in \Z^d}G(u^*(x),v^*(y))H(d(x,y)).
    \end{equation*}
    Moreover, if $u=u^*$ with $\ran u=[a_1>a_2>\cdots]$, $G$ is strictly supermodular, $H(n)>H(n+1)$ for some $n\in \N$, $|\sum\limits_{x,y \in \Z^d}G(u^*(x),v^*(y))H(d(x,y))|<\infty$, then the equality holds if and only if $v=v^*$.
\end{theorem}

Let us emphasize that a lot of challenging open questions about discrete symmetrization in higher dimensions can now be solved thanks to our approach. More precisely, we can address the optimizers of some discrete functional inequalities in $\mathbb{Z}^d,$  $d>1.$

We provide the basic setups in section 2, and prove Theorem~\ref{thm:Pruss's_definition_fails_for_Z2} in section 3. In section 4, we provide the definition for the Schwarz rearrangement on $\Z^d$, and show some properties of rearranged functions. In section 5, the main result of this paper, namely the extended Riesz inequality in the discrete setting, and other inequalities will be established. Section 6 is devoted to some applications of our results. Our study opens a locked door for decades.
Thanks to our findings, it is now possible to tackle all the aspects studied in the continuous case,
however things are totally different on graphs!

Let us start our presentation with basic notation and definitions.

\section{Preliminaries}

\subsection{Graphs and function spaces on graphs}
We recall the setting of graphs: Let $G=(V,E)$ be a simple, undirected, locally finite graph, where $V$ is the set of vertices and $E$ is the set of edges. Given two vertices $x,y \in V$, if there exists an edge joining $x$ and $y$, we write $x\sim y$ or $xy\in E$, and call $y$ a \emph{neighbor} of $x$. The \emph{combinatorial distance} is defined via
$$d(x,y)= \inf \{l:x=x_0\sim x_1 \sim \cdots \sim x_l=y\}.$$

We focus on lattice graphs in this paper. Given the group $\Z^d$, let $e_i$ be the $i$-th standard unit vector, $1\le i \le d$, and let $S=\{\pm e_i\}$. There is a natural graph structure on $\Z^d$ defined via
$$x\sim y \text{ if and only if there exists $e\in S$ such that $x=y+e$.} $$
The graph defined above is called the $d$-dimensional lattice graph, denoted by $\Z^d$. It is the discretization of $d$-dimensional Euclidean space.

Let $x=(x^1,x^2,\cdots,x^d)\in \Z^d$, we write
\begin{itemize}
    \item $e_i=(0,\cdots,0,\ \overset{i\text{-th}}{1}\ ,0,\cdots,0)$ be the $i$-th standard unit vector;
    \item $||x||_p^p=\sum_{i=1}^d |x^i|^p$, where $1\le p < \infty$;
    \item $|x|=||x||_2=\sqrt{\sum_{i=1}^d |x^i|^2};$
    \item $||x||_{\infty}=\max\{|x^i|\}$;
    \item $||x||_1=d(0,x)=\sum_{i=1}^d |x^i|$;
    \item $V(x)=\{(\pm x^{\sigma_1}\!\!,\pm x^{\sigma_2}\!\!,\cdots,\pm x^{\sigma_d}\!)\!:\:\sigma\;\text{is a permutation of }(1,2,\cdots,d)\}$;
    \item $\overline{V}(x)=\{\text{The convex hull of $V(x)$ in $\R^d$} \} \cap \Z^d$;
    \item $\partial \overline{V}(x)=\{\text{The boundary of the convex hull of $V(x)$ in $\R^d$} \} \cap \Z^d$;
    \item $V^{\Box}_l:= \overline{V}{(l,l,\cdots,l)}=\{y\in \Z^d:\ ||y||_{\infty} \le l\}$, where $l$ is a positive integer;
    \item $V^{\Diamond}_l:= \overline{V}{(l,0,\cdots,0)}=\{y\in \Z^d:\ ||y||_1 \le l\}=\{y\in \Z^d:\ d(0,y)\le l\}$, where $l$ is a positive integer.
\end{itemize}

\begin{figure}[H]
\centering
\setcounter{subfigure}{0}
\subfigure[$V(2,3)$ on $\Z^2$.]{
\includegraphics[width=0.3\textwidth]{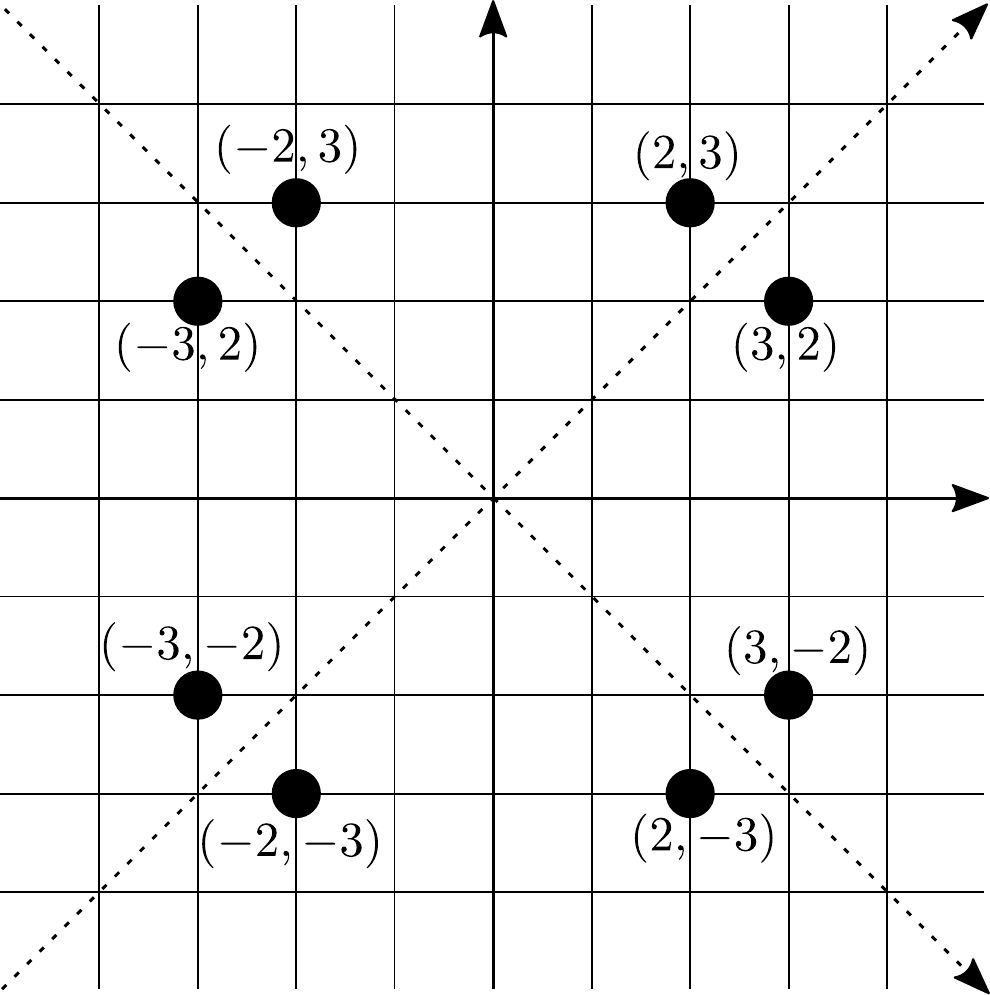}} \ \
\subfigure[$\overline{V}(2,3)$ on $\Z^2$.]{
\includegraphics[width=0.3\textwidth]{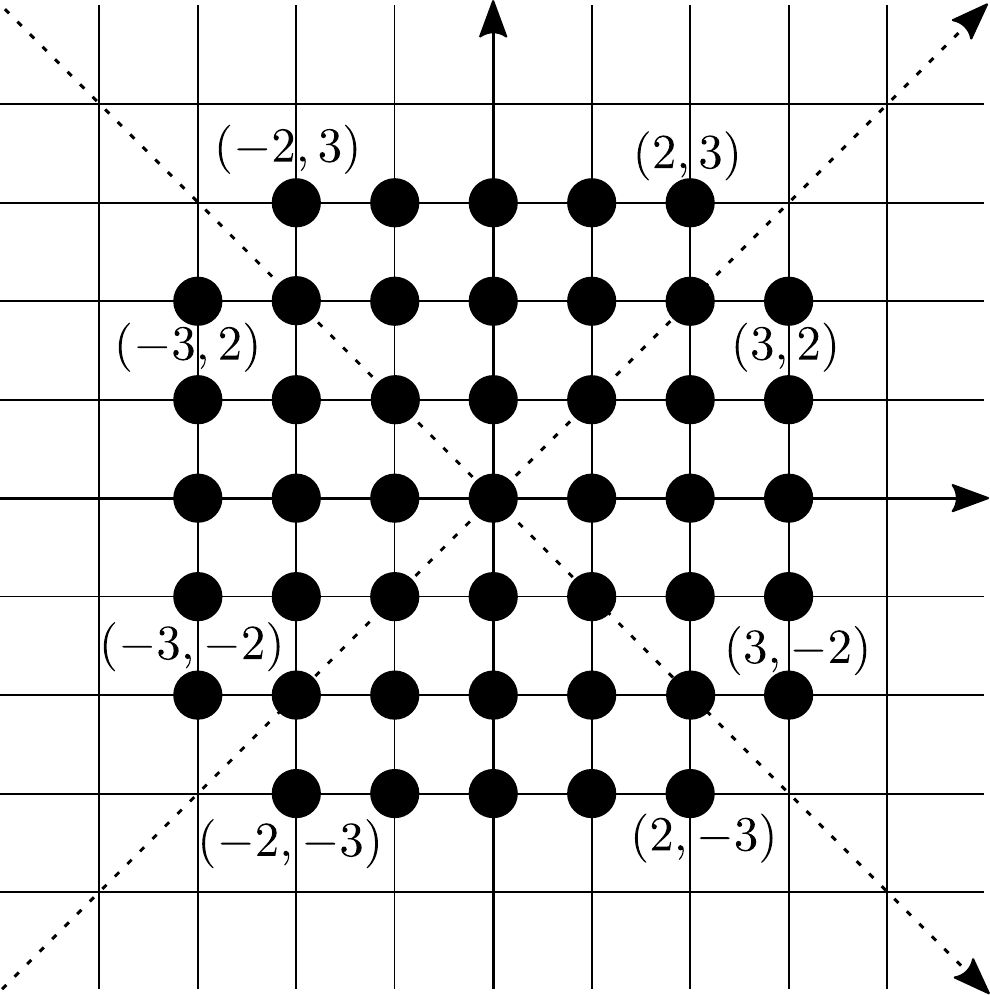}} \ \
\subfigure[$\partial \overline{V}(2,3)$ on $\Z^2$.]{
\includegraphics[width=0.3\textwidth]{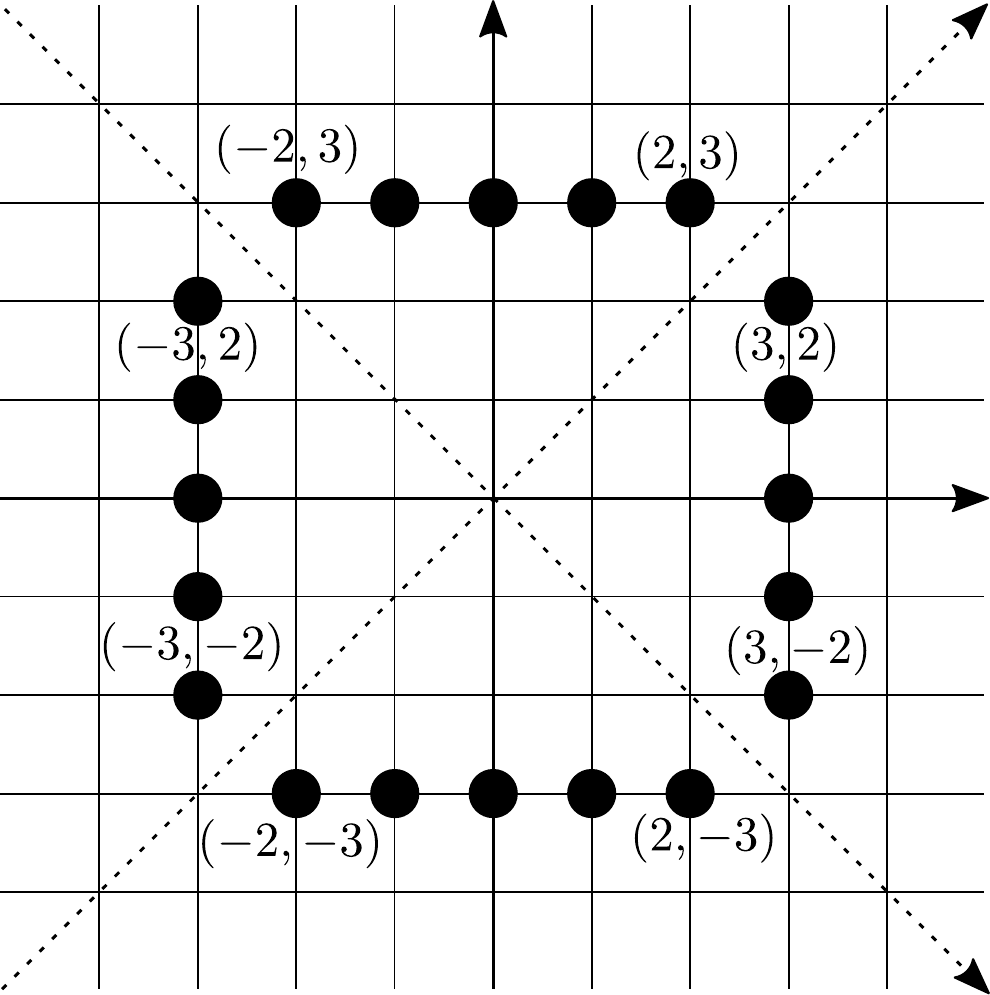}}
\label{fig-V_2,3_}
\end{figure}

For $\Z^2$, we have
\begin{equation}\label{equ-num_box_dia}
    |V^{\Diamond}_l|=2l^2+2l+1, \qquad |V^{\Box}_l|=(2l+1)^2.
\end{equation}

\begin{figure}[H]
\centering
\setcounter{subfigure}{0}
\subfigure[$V^{\Diamond}_3$ on $\Z^2$]{
\includegraphics[width=0.35\textwidth]{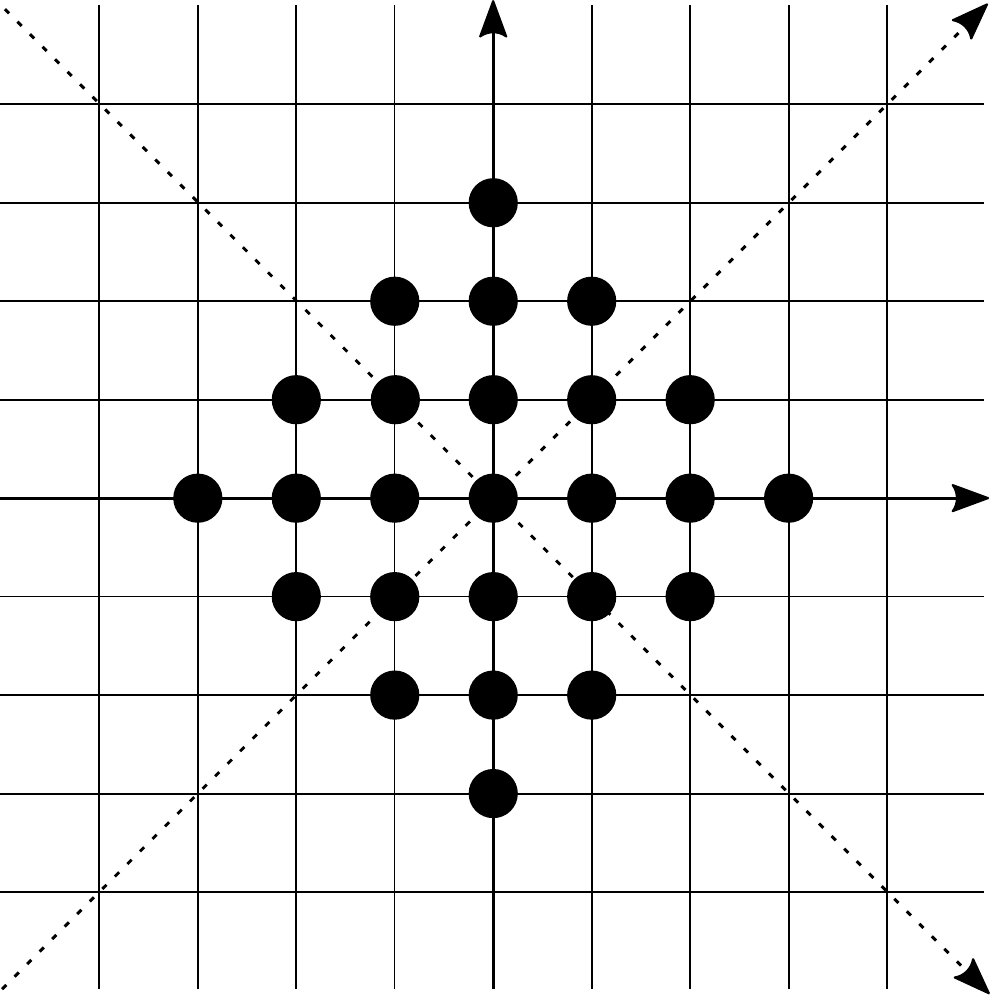}} \quad
\subfigure[$V^{\Box}_3$ on $\Z^2$]{
\includegraphics[width=0.35\textwidth]{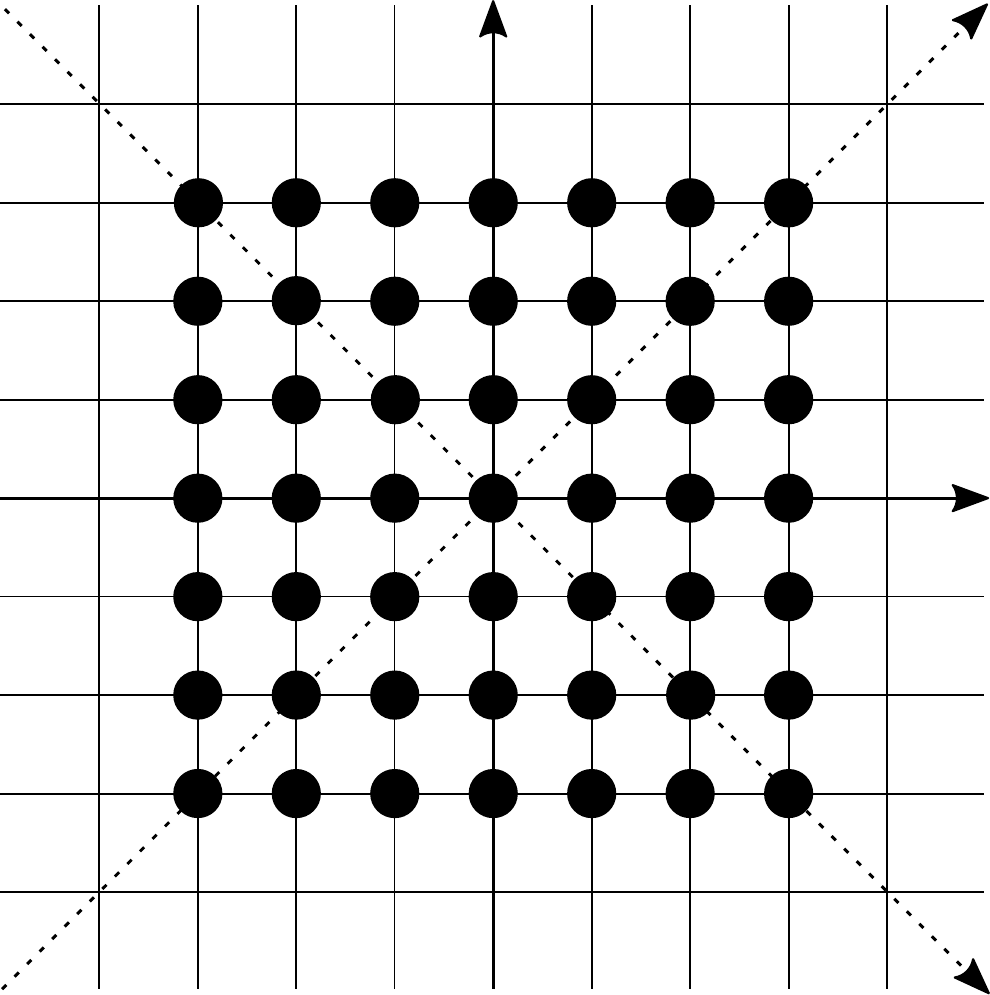}} \quad
\end{figure}

Given a graph $G=(V,E)$ and a subset $\Omega \subset V$, $|\Omega|$ denotes the cardinality of $\Omega$. Let $1\le p < +\infty$ be a fixed number, we have:
\begin{itemize}
    \item $\R^V$ is the set of all real functions on $V$;
    \item $\supp{u}:=\{x \in V: u(x) \ne 0\}$ is the support of $u\in \R^V$;
    \item $C_c(G):=\{u \in \R^V: |\supp{u}| < +\infty\}$ is the space of finitely supported functions on $V$;
    \item $C_0(G):=\{u \in \R^V: |\{x : |u(x)| > t \}|<+\infty,\ \forall t > 0\}$ is the space of functions that vanish at infinity;
    \item $||u||_p:=(\sum_{x\in V} |u(x)|^p)^{1/p}$ is the $l^p$ norm of $u$;
    \item $||u||_{\infty}:=\sup\{|u(x)|:x\in V\}$ is the $l^{\infty}$ norm of $u$;
    \item $l^p(V):=\{u \in \R^V: ||u||_p < +\infty\}$ is the $l^p$ function space on $V$;
    \item $|\nabla_{xy}u|:=|u(y)-u(x)|$, where $xy\in E$;
    \item $||\nabla u||_p^p:=\frac{1}{2}\sum_{x\in V}\sum_{y\sim x}|u(x)-u(y)|^p=\sum_{xy\in E}|\nabla_{xy}u|^p$ is the $p-$Sobolev energy of $u$;
    \item $||u||_{W^{1,p}}:=||u||_p+||\nabla u||_p$ is the $W^{1,p}$ norm of $u$;
    \item $W^{1,p}(G):=\{u\in \R^V:||u||_{W^{1,p}} < +\infty \}$ is the $W^{1,p}$ space on $G$.
\end{itemize}
The graph Laplacian $\Delta : \R^V \to \R^V$ is defined via
$$\Delta u(x)=\sum_{y\sim x} (u(y)-u(x)).$$

It is well-known that $W^{1,p}(\Z^d)=l^p(\Z^d)$. Moreover, $||\cdot||_{W^{1,p}}$ and $||\cdot||_p$ are equivalent norms; see \cite[Lemma 2.1]{chen2020riesz}.

Let us recall that two nonnegative functions $u,v \in \R^V$ are \emph{equimeasurable} if
\begin{equation}\label{equ-equimeasurable}
  |\left\{x\in V:\ u(x)>t\right\}|=|\left\{x\in V:\ v(x)>t\right\}|,\ \forall t>0,
\end{equation}
and we write $u \overset{e.m.}{\sim} v$.

A nonnegative function $u: V \to \R_+$ is said to be \emph{admissible} for Schwarz symmetrization/ Schwarz rearrangement, (or admissible, in short) if $u\in C_0(G)$. We write $C_0^+(G)$ be the set of all admissible functions on $G$. Given $u\in C_0^+(G)$, let
$$a_n=\sup_{\underset{|\Omega|=n}{\Omega \subset V}}\min\{u(x):x\in \Omega\}$$
be the $n$-th largest function value of $u$. We write $\ran(u):=[a_1 \succ a_2 \succ \cdots ]$ as a multiset with a total order ``$\succ$''. If $a_n>a_{n+1}$ for all $n\in \Z^+$, we write $\ran(u)=[a_1 > a_2 > \cdots ]$. Define the $n$-th cut-off of $u$ via
\begin{equation}
    \hat u_n(x)=
    \begin{cases}
    u(x), \qquad &u(x) \succ a_{n+1},\\
    0, & \text{otherwise.}
    \end{cases}
\end{equation}
Then $\hat u_n$ is a finitely supported function with $|\supp{\hat u_n}| \le n$.

\subsection{One dimensional Schwarz rearrangement}

Let $u\in C_0^+(\Z)$ be admissible with $\ran(u)=[a_1 \succ a_2 \succ \dots ]$. Let
\begin{equation*}
    u^*(x)=
    \begin{cases}
    a_{1-2x}, \qquad & x\le 0;\\
    a_{2x}, & x>0.
    \end{cases}
\end{equation*}
The discrete Schwarz rearrangement of $u$ is defined via $R_{\Z}u=u^*$.
The above method also applies to $\Z+1/2=\{\pm 1/2, \pm 3/2, \pm 5/2, \cdots\}$. Let $u\in C_0^+(\Z+1/2)$ be admissible, we denote the map $u \mapsto u^*$ by $R_{\Z+1/2}$, or still use $R_{\Z}$.

\begin{figure}[H]
\centering
\setcounter{subfigure}{0}
\subfigure[$R_{\Z}u$]{
\includegraphics[width=0.8\textwidth]{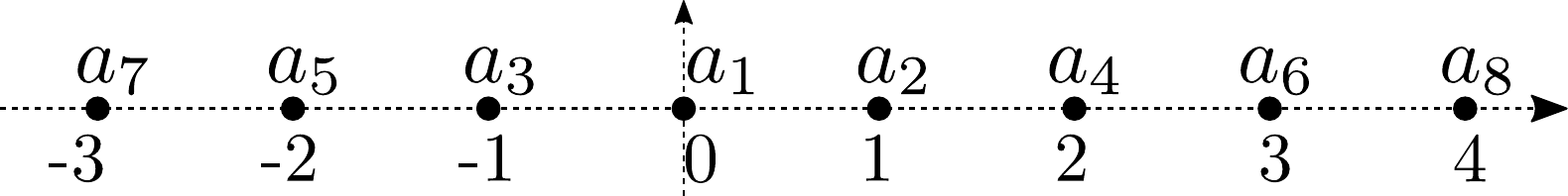}} \quad
\subfigure[$R_{\Z+1/2}u$]{
\includegraphics[width=0.8\textwidth]{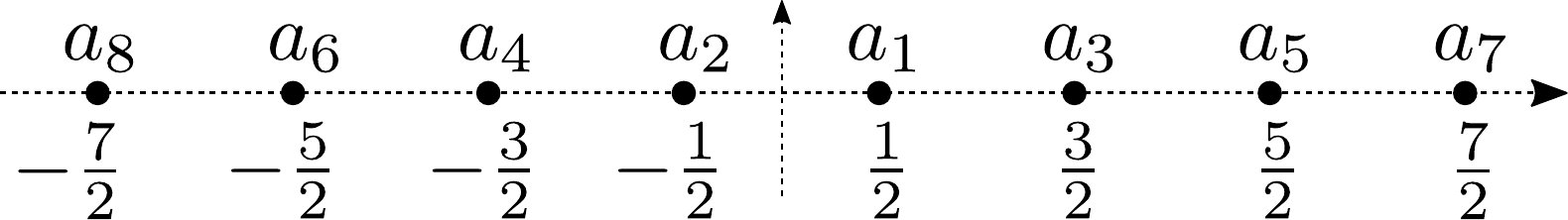}} \quad
\end{figure}

\section{Proof of Theorem \ref{thm:Pruss's_definition_fails_for_Z2}}

The main idea of the proof is to find the optimal equimeasurable function that minimizes the Sobolev energy. We start by introducing some notations:

Given a graph $G=(V,E)$, a bijection $\sigma:V \to V$ is called a \emph{graph automorphism} if $x\sim y \Leftrightarrow \sigma(x) \sim \sigma(y)$ for all $x,y\in V$. Note that a graph automorphism of $\Z^d$ may not be a group automorphism. Given two subsets $\Omega_1,\Omega_2$, we write $\Omega_1 \cong \Omega_2$ if there exits a graph automorphism $\sigma$ such that $\sigma(\Omega_1)=\Omega_2$. Given two functions $u_1,u_2$, we write $u_1 \cong u_2$ if there exits a graph automorphism $\sigma$ such that $u_1(\sigma(x))=u_2(x)$ for all $x\in V$.

\begin{lemma}
    \label{lemma:optimal_geometry_for_functions_with_support_5}
    Let $u\in C_c(\Z^2)$ be a nonnegative function with $|\supp{u}|=5$, then there exists a function $u'$ equimeasurable with $u$ and
    \begin{equation}
        \label{equ:lemma:optimal_geometry_for_functions_with_support_5}
        ||\nabla u'||_2 = \inf \{||\nabla v||_2:\ v \overset{e.m.}{\sim} u\}.
    \end{equation}
    Moreover $\supp{u'} \cong V^{\Diamond}_1$ or $\Omega=\{(0,0),(0,1),(1,0),(0,-1),(1,1)\}$.
\end{lemma}
\begin{figure}[H]
    \centering
    \setcounter{subfigure}{0}
    \subfigure[$V^{\Diamond}_1$]{
        \includegraphics[width=0.35\textwidth]{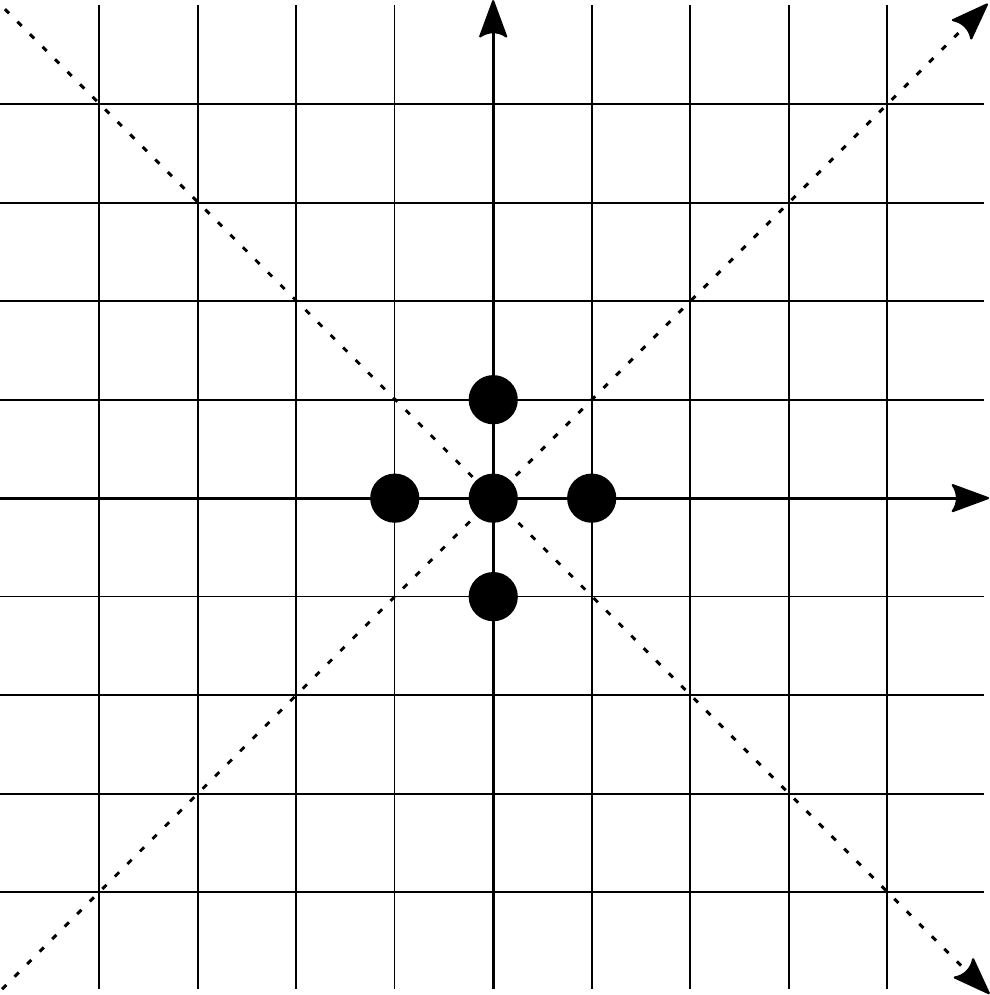}} \quad
    \subfigure[$\Omega$]{
        \includegraphics[width=0.35\textwidth]{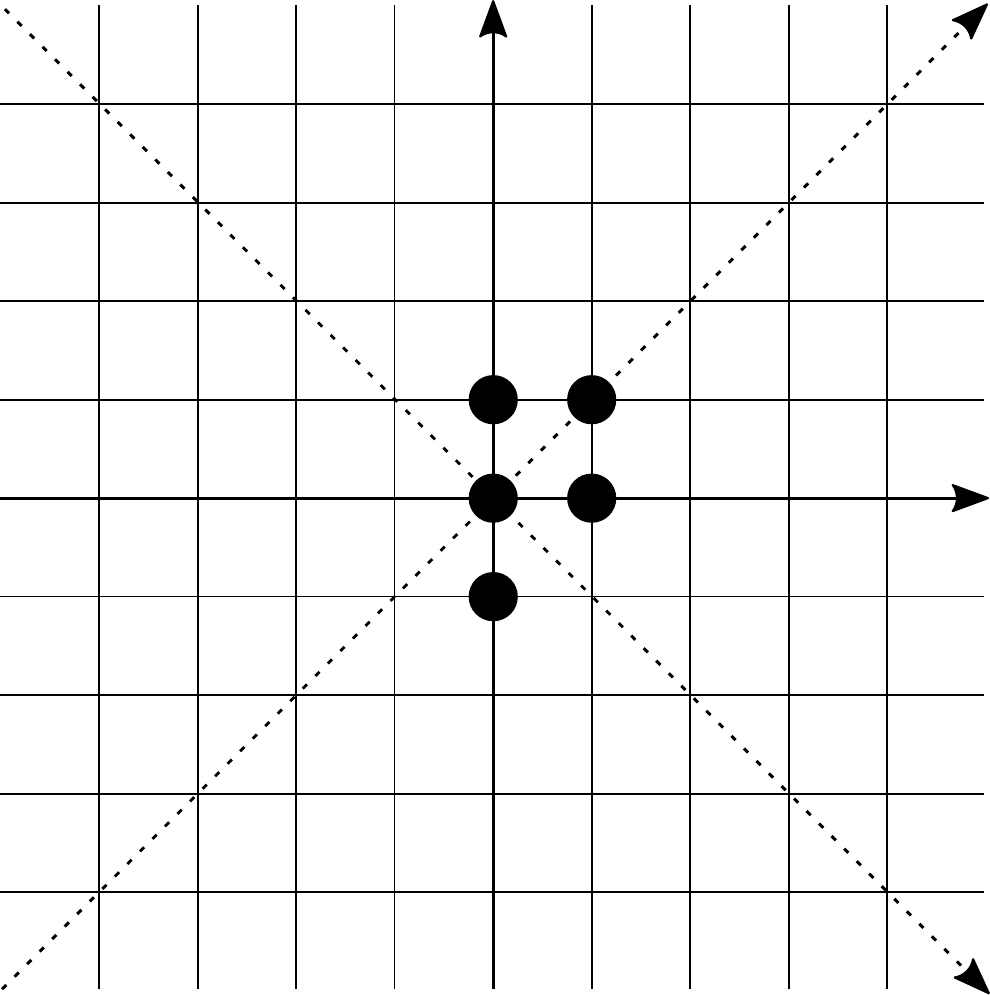}} \quad
    \caption{Geometry of optimal functions with support 5}
\end{figure}

\begin{proof}
    The existence of $u'$ is trivial since $|\supp{u}|$ is finite. We only need to show the second part of the statement.

    \underline{Step 1}: We first prove that $\supp{u'}$ is connected, i.e. $\forall x,y\in \supp{u'}$, there exists a finite path $x = z_1 \sim z_2 \sim \cdots \sim z_l = y$, where $z_i \in \supp{u'}$. Otherwise we shift the connected components to get a connected support and the Sobolev energy decreases strictly.

    \underline{Step 2}: There exists an $x\in \supp{u'}$ such that $|N(x)\cap \ \supp{u'}|\ge 3$, where $N(x):=\{y\in V: y\sim x\}$ is the set of neighbors of $x$. Otherwise, suppose that $|N(x)\cap \ \supp{u'}|\le 2$, i.e. $\supp{u'}=\{x_1,x_2,x_3,x_4,x_5\}$ with $x_i \sim x_j$ if and only if $|i-j|=1$. Let $u''$ be the function with $u''(0,-1)=u'(x_1)=b_1$, $u''(0,0)=u'(x_2)=b_2$, $u''(0,1)=u'(x_3)=b_3$, $u''(1,1)=u'(x_4)=b_4$,$u''(1,0)=u'(x_5)=b_5$, then $||\nabla u''||_2 < ||\nabla u'||$, which contradicts \eqref{equ:lemma:optimal_geometry_for_functions_with_support_5}.
    \begin{figure}[H]
        \centering
        \setcounter{subfigure}{0}
        \subfigure[$u'$]{
            \includegraphics[width=0.35\textwidth]{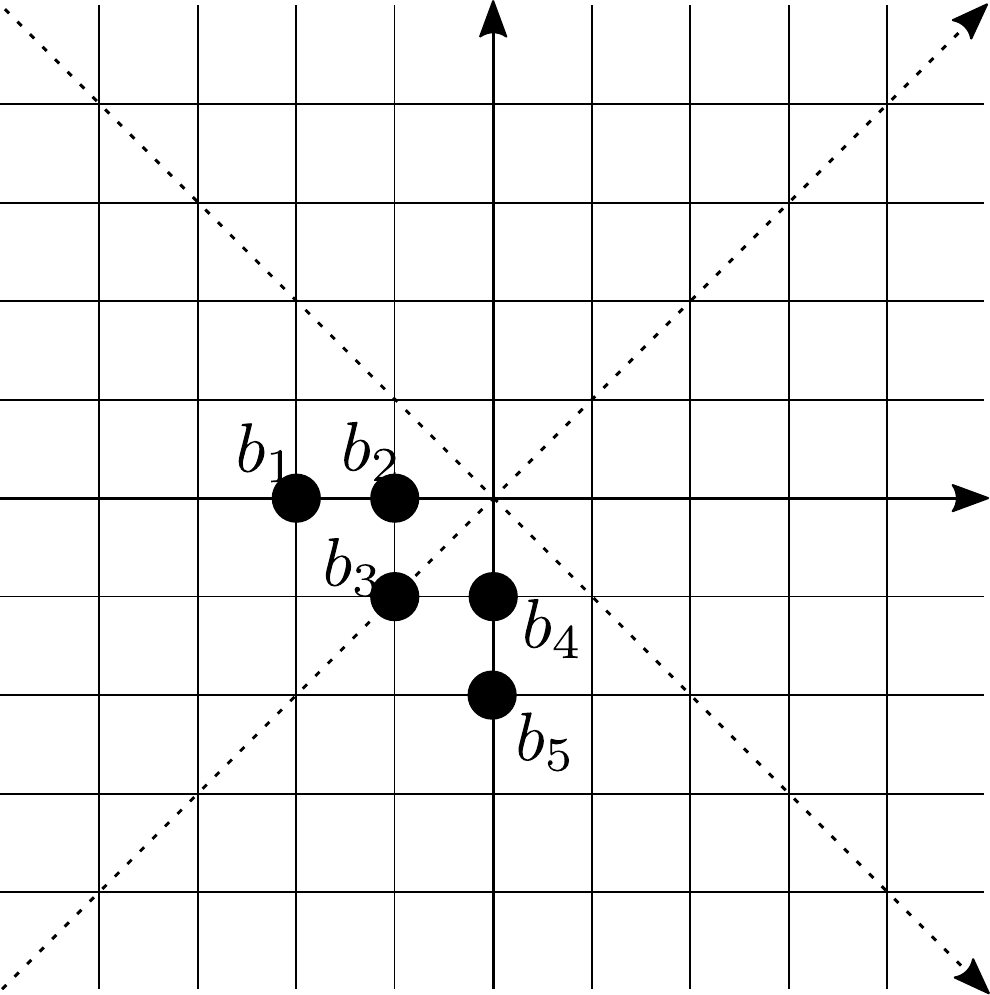}} \quad
        \subfigure[$u''$]{
            \includegraphics[width=0.35\textwidth]{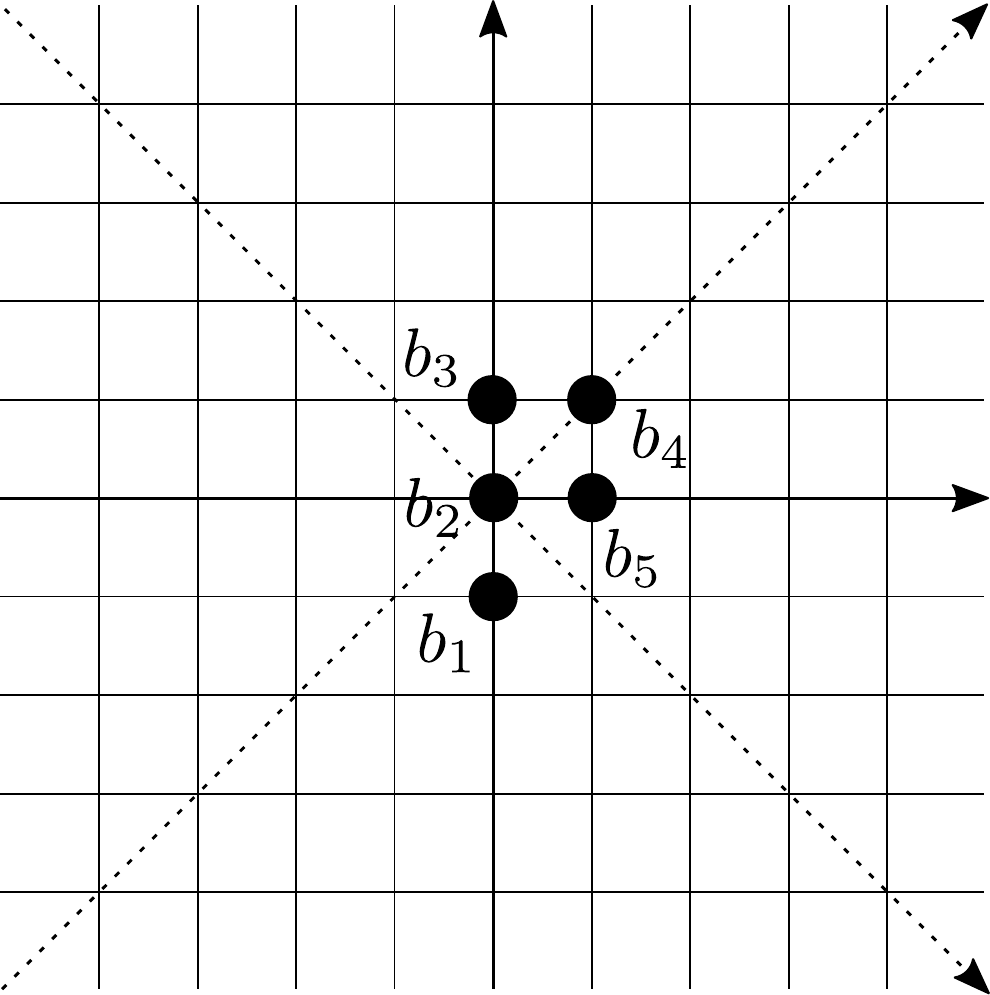}} \quad
    \end{figure}

    \underline{Step 3}: We prove that $\supp{u'} \cong V^{\Diamond}_1$ or $\Omega$. Suppose that there exists an $x\in \supp{u'}$ such that $|N(x)\cap \ \supp{u'}|=4$, then $\supp{u'} \cong V^{\Diamond}_1$. Otherwise W.L.O.G., we assume that $x=(0,0)\in \supp{u'}$, and $N(x)\cap \ \supp{u'}=\{(0,1),(1,0),(0,-1)\}$. Let $\{y\}=\supp{u'} \setminus V^{\Diamond}_1$. Then $y=(1,1)$ or $(-1,1)$, otherwise we move the function value $u'(y)$ to $(1,1)$ or $(1,-1)$ to get a new function $u''$ with $||\nabla u''||_2 < ||\nabla u'||_2$. This proves the result.
    \begin{figure}[H]
        \centering
        \setcounter{subfigure}{0}
        \subfigure[$u'$]{
            \includegraphics[width=0.35\textwidth]{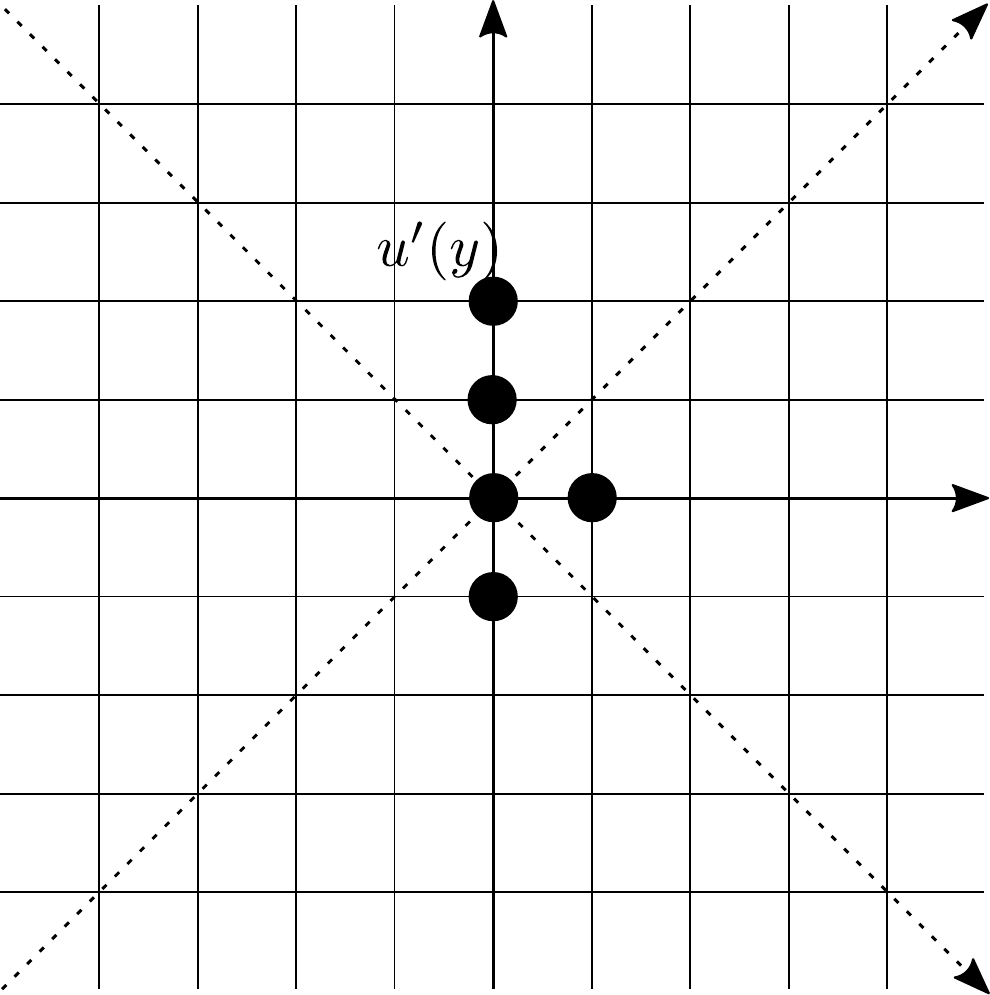}} \quad
        \subfigure[$u''$]{
            \includegraphics[width=0.35\textwidth]{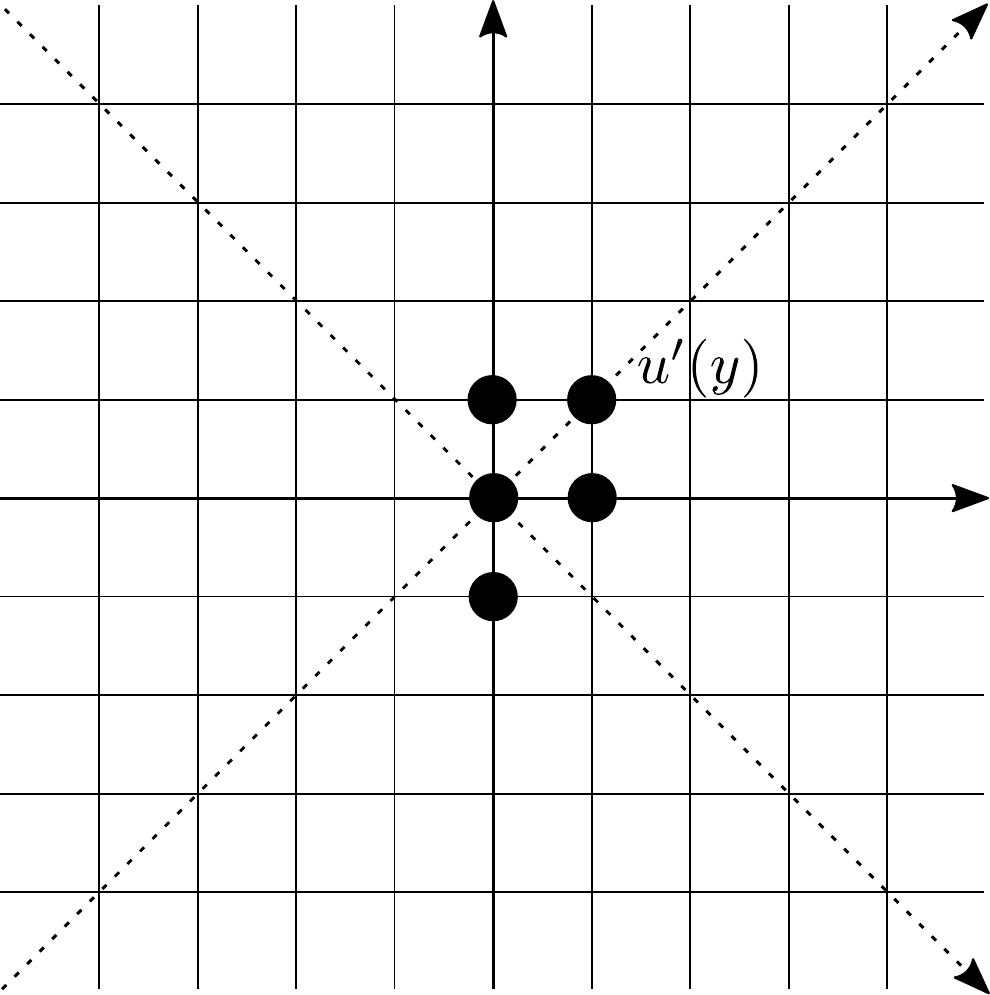}} \quad
        \caption{If $\supp{u'}$ is shown as (a), we shift the function value $u'(y)$ and obtain $u''$ with $||\nabla u''||_2 < ||\nabla u'||_2$.}
    \end{figure}
\end{proof}

Now we are ready to prove Theorem~\ref{thm:Pruss's_definition_fails_for_Z2}.
\begin{proof}[Proof of Theorem~\ref{thm:Pruss's_definition_fails_for_Z2}]
    Let $u_1$ , $u_2$ be functions as shown in the following figures below, function values not labelled are all zeros.
    \begin{figure}[H]
        \centering
        \setcounter{subfigure}{0}
        \subfigure[$u_1$]{\includegraphics[width=0.35\textwidth]{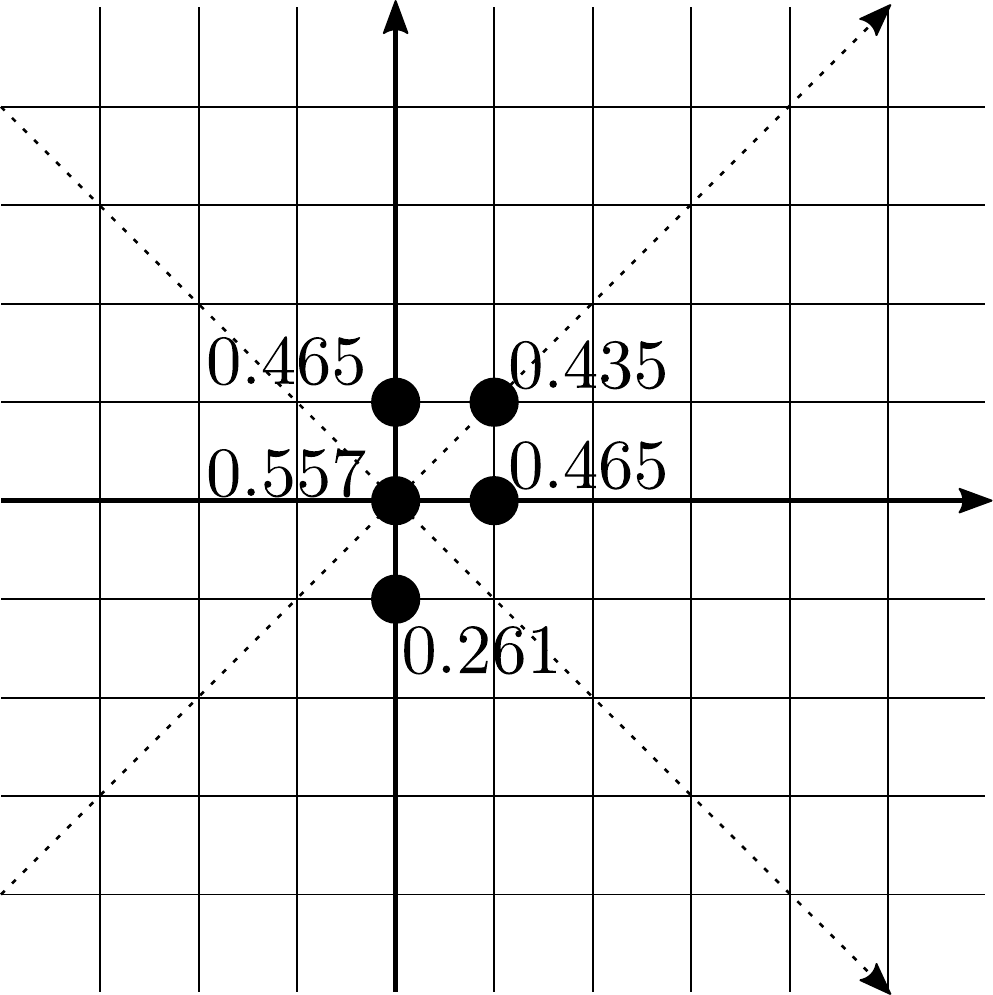}} \quad
        \subfigure[$u_2$]{\includegraphics[width=0.35\textwidth]{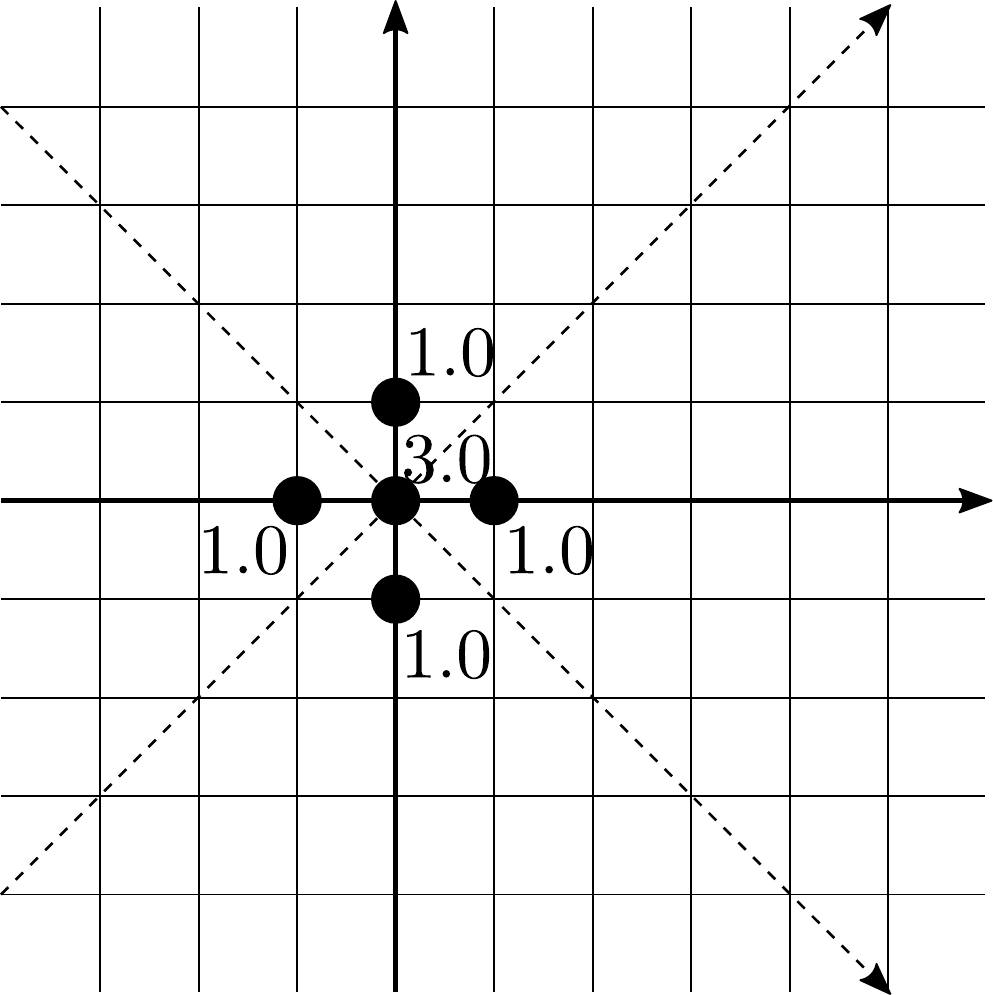}}
    \end{figure}

    By Lemma~\ref{lemma:optimal_geometry_for_functions_with_support_5} and a direct calculation, we have
    \begin{enumerate}[(i)]
        \item $||\nabla u_1||_2 \le ||\nabla u'_1||_2$ for all $u'_1 \overset{e.m.}{\sim} u_1$. The equality holds if and only if $u'_1 \cong u_1$. In fact $\supp{u_1}$ minimizes the first Dirichlet Laplacian eigenvalue of all regions $\Omega \subset \Z^2$ with $|\Omega|=5$, and $u_1$ is the corresponding eigenvector; see \cite{shlapentokh2010asymptotic}.
        \item $||\nabla u_2||_2 \le ||\nabla u'_2||_2$ for all $u'_2 \overset{e.m.}{\sim} u_2$. The equality holds if and only if $u'_2 \cong u_2$.
    \end{enumerate}
     We prove the theorem by contradiction. Suppose that ``$\prec$'' is such an order of $\Z^2$. W.L.O.G., assume that $(0,0)$ is the minimal element w.r.t. ``$\prec$''. Let $V_1=\{(0,\pm 1),(\pm 1, 0)\}$, $V_2=\{(\pm 1, \pm 1)\}$. By (i), there exists $x\in V_1$ and $y\in V_2$ such that $y \prec x$. By (ii), $x\prec y$ for all $x\in V_1$ and $y\in V_2$. Then we get a contradiction and the result follows.
\end{proof}

\section{Definition of Schwarz rearrangement on lattice graphs}
We provide a definition of discrete Schwarz rearrangement on lattice graphs and study the properties of the corresponding rearranged functions in this section.

\subsection{Schwarz rearrangement on $\Z^2$}
By Theorem~\ref{thm:Pruss's_definition_fails_for_Z2}, we know that it is impossible to find a suitable total order on $\Z^2$ as we did for $\Z$. We now state the definition of discrete Schwarz rearrangement.
\subsubsection{\bf{Definition}}
By modifying the discrete Steiner symmetrization developed in \cite{shlapentokh2010asymptotic}, we define four kinds of one-step rearrangement on $\Z^2$, and the Discrete Schwarz rearrangement is defined as the limit of a sequence
of iterated one-step rearrangement. Let
$$\Vec{E}=\{e_1,e_2,\frac{e_1+e_2}{2},\frac{e_1-e_2}{2}\},$$
where $e_1$ and $e_2$ are the standard unit vectors. Given $e\in \Vec{E}$, if $x=(x^1,x^2),y=(y^1,y^2)\in \Z^2$ satisfying $x-y$ is parallel to $e$, we write $x\overset{e}{\sim}y$. Then $\overset{e}{\sim}$ is an equivalence relation and there is an unique corresponding partition $\Z^2=\bigsqcup_{\alpha \in I_{e}} V_e^{\alpha}$. Given $x=(x^1,x^2)$, we choose the superscript $\alpha$ of the equivalence class $V_e^{\alpha} \ni x$ in the following way:
\begin{align*}
    \alpha =
    \begin{cases}
        x^2,\quad & e=e_1;\\
        x^1, & e=e_2;\\
        x^1-x^2, & e=\frac{e_1+e_2}{2};\\
        x^1+x^2, & e=\frac{e_1-e_2}{2}.
    \end{cases}
\end{align*}
Then we have
\begin{equation}
    \Z^2=\bigsqcup_{\alpha\in \Z} V_e^{\alpha}
\end{equation}
Given $e\in \Vec{E}$ and $V_e^{\alpha}$, the map $x \mapsto <e,x>$ is a bijection from $V_e^{\alpha}$ to $\Z$ or $\Z+1/2$, where $<\cdot,\cdot>$ is the inner product of vectors. Then $R_{\Z}$ on $V_e^{\alpha}$ is well-defined. Given $u\in C_0^+(\Z^2)$ be admissible, we define the \emph{one-step rearrangement} of $u$ with respect to $e\in \Vec{E}$ via
\begin{equation}
    (R_e u)|_{V_e^{\alpha}}=R_{\Z}(u|_{V_e^{\alpha}}).
\end{equation}
Then $R_e u$ is admissible and is equimeasurable with $u$. Therefore, $R_e$ is well-defined for admissible functions.

\begin{figure}[H]
    \centering
    \setcounter{subfigure}{0}
    \subfigure[$V_{e_1}^1$ and $V_{e_1}^2$]{\includegraphics[width=0.35\textwidth]{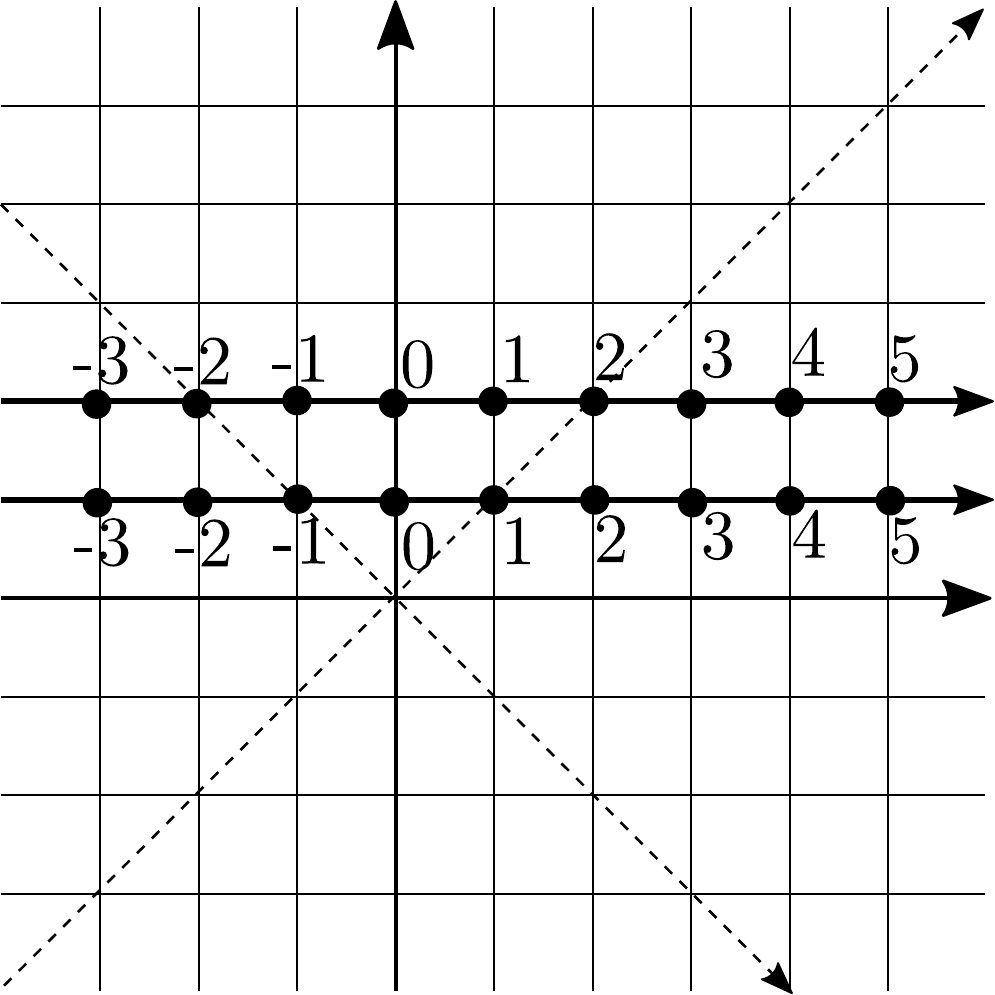}}\quad
    \subfigure[$V_{\frac{e_1+e_2}{2}}^1$ and $V_{\frac{e_1+e_2}{2}}^2$]{\includegraphics[width=0.35\textwidth]{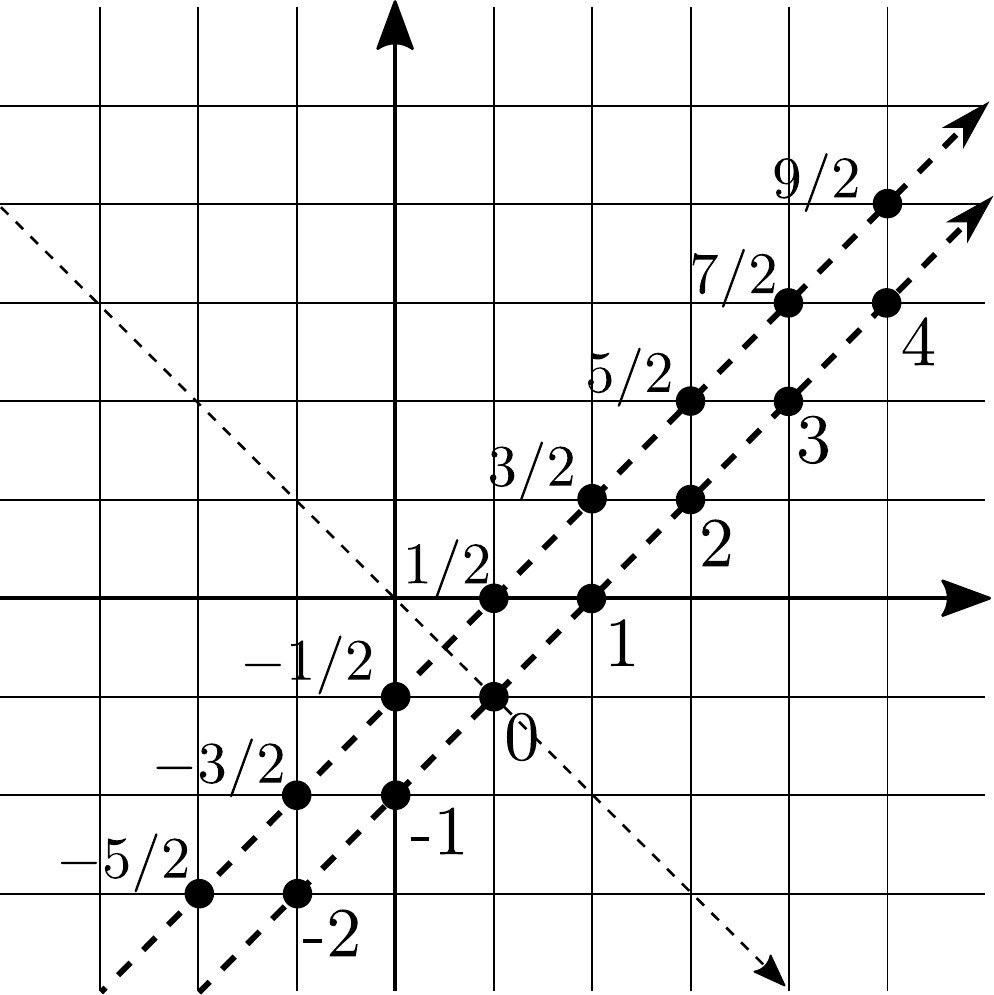}}
    \label{fig-V_2}
\end{figure}

\begin{definition}[Discrete Schwarz rearrangement on $\Z^2$]
    \label{def}
    Let $u \in C_0^+(\Z^2)$ be an admissible function, we write the $k$-th iteration of one-step rearrangements by
    $$R^k u:=\underbrace{\cdots \circ R_{e_2} \circ R_{e_1} \circ R_{\frac{e_1-e_2}{2}} \circ R_{\frac{e_1+e_2}{2}} \circ R_{e_2} \circ R_{e_1}}_k u.$$
    The Schwarz rearrangement of $u$ is defined via
    \begin{equation}
        \label{equ:definition_of_2_dimensional_discrete_Schwarz_rearrangement}
        R_{\Z^2}u(x)=\lim_{k\to \infty} R^k u (x),\ x\in \Z^2.
    \end{equation}
\end{definition}
Before proving that the discrete Schwarz rearrangement is well-defined, we first give an example that shows how a function becomes ``more symmetric'' under one-step rearrangements.

\begin{example}
    Let $\Omega =\{(x,0)\in \Z^2 : x=0,\pm 1, \pm 2, \pm 3.\}$ as shown in the figure (a) below, $u=\mathbbm{1}_{\Omega}$. As shown below, the function values of $R^k u$ gather towards the center by one-step rearrangements.
    \begin{figure}[H]
        \centering
        \setcounter{subfigure}{0}
        \subfigure[represents the graph of $u=R^1 u=R^2 u$]{\includegraphics[width=0.3\textwidth]{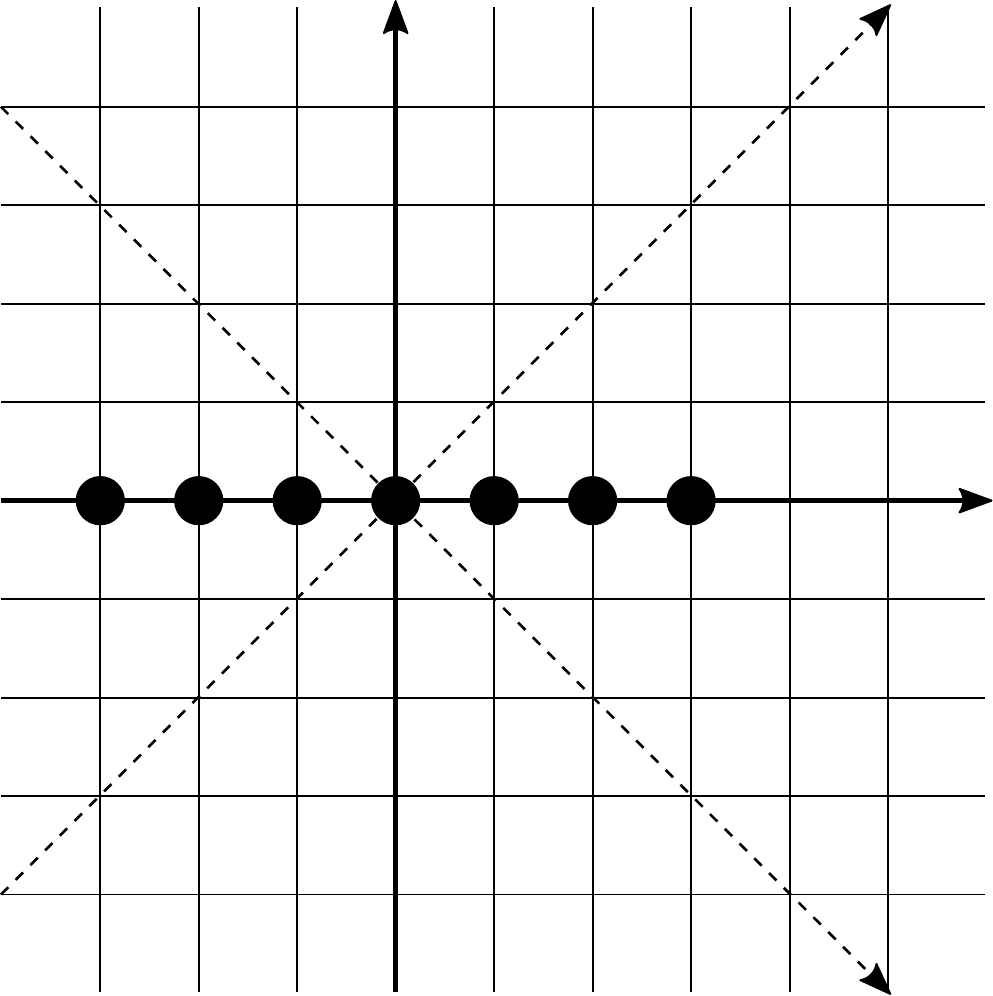}}\
        \subfigure[represents the graph of $R^3 u=R^4 u$]{\includegraphics[width=0.3\textwidth]{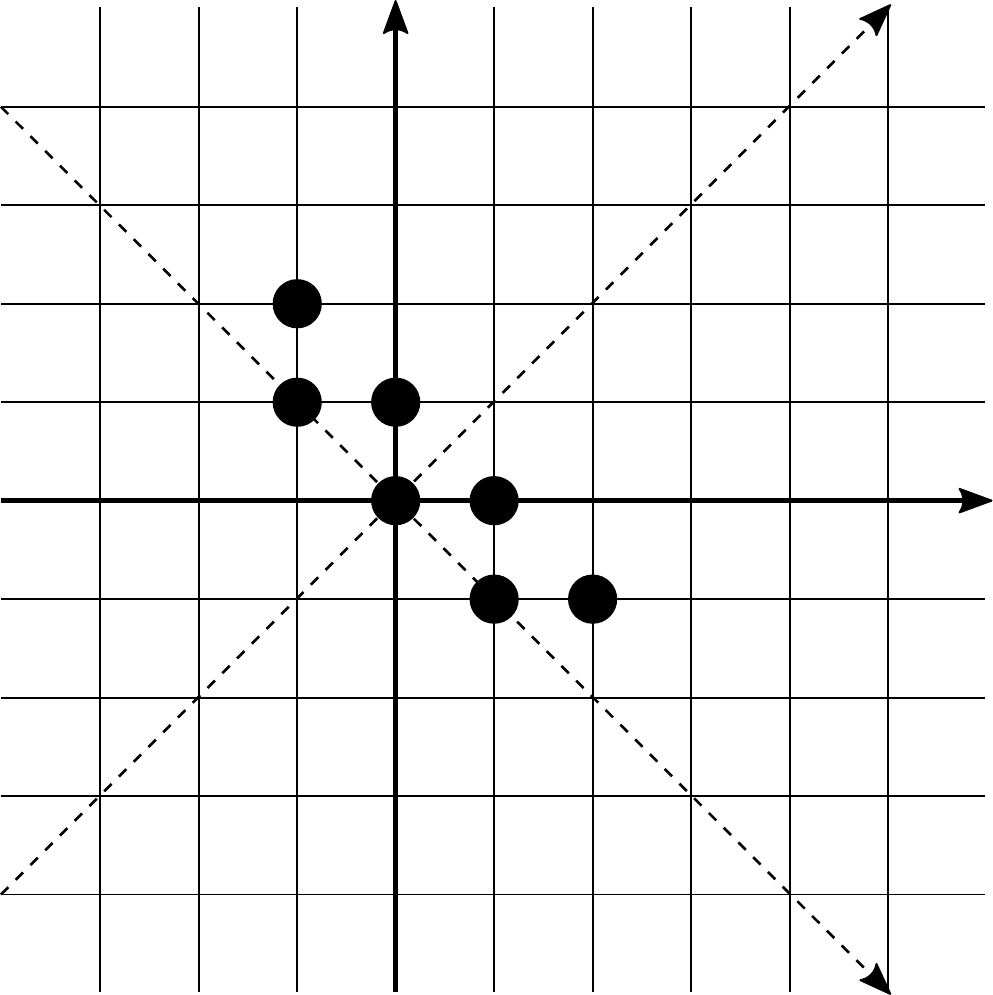}}\
        \subfigure[represents the graph of $R^5 u=R^6 u$]{\includegraphics[width=0.3\textwidth]{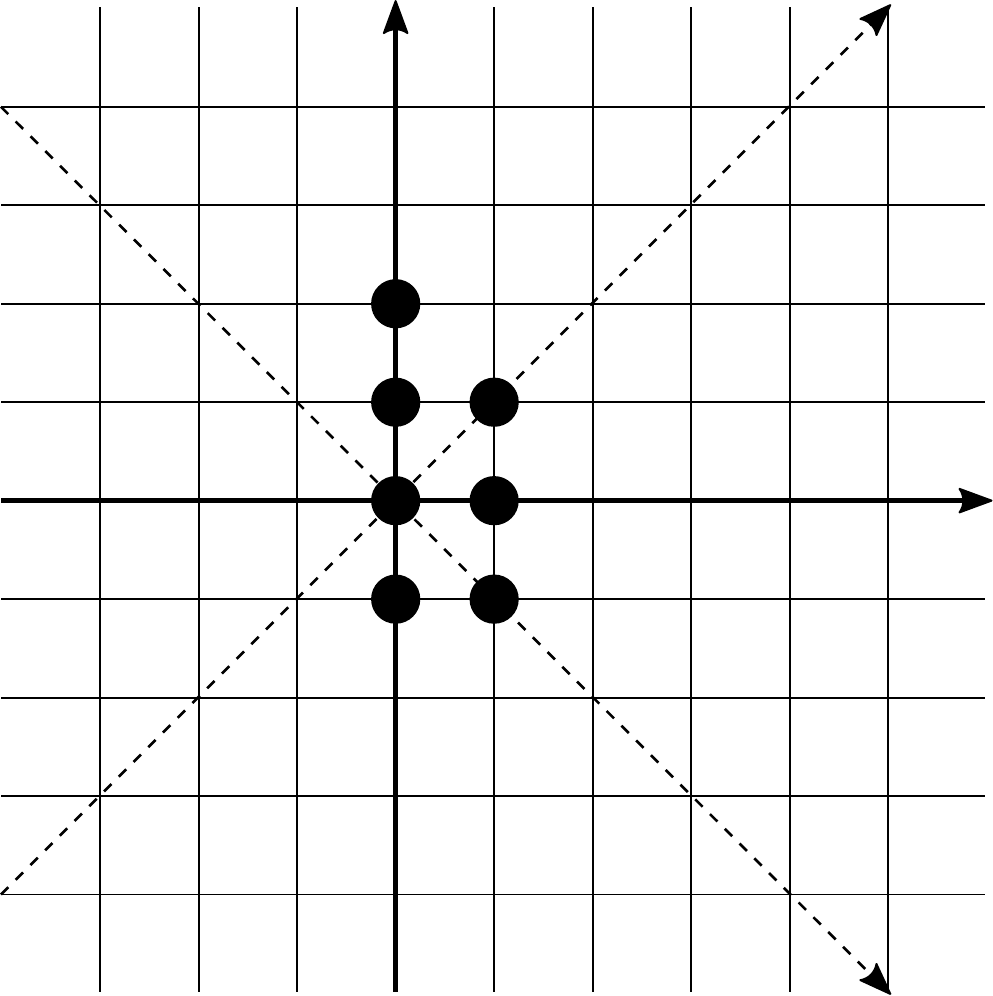}}\
        \subfigure[represents the graph of $R^7 u=R^8 u=R^9 u$]{\includegraphics[width=0.3\textwidth]{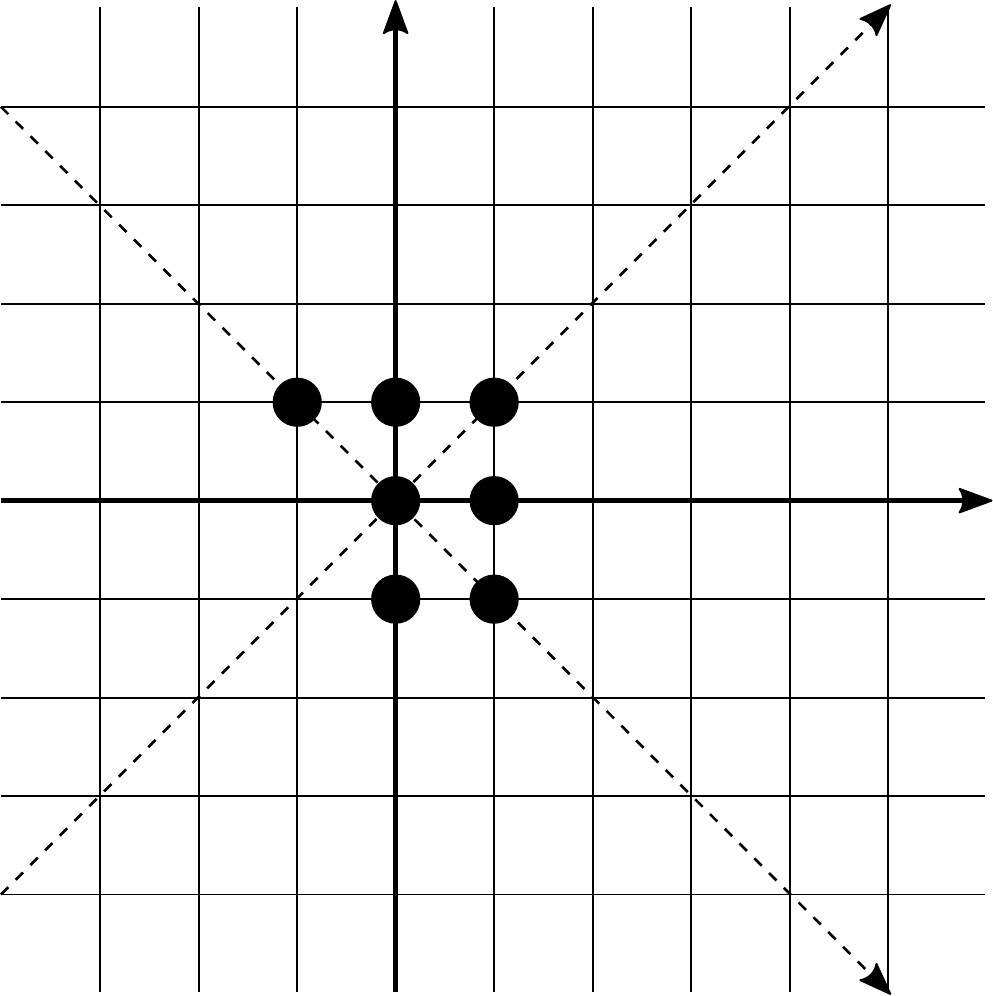}}\
        \subfigure[represents the graph of $R^{10} u=R_{\Z^2}u$]{\includegraphics[width=0.3\textwidth]{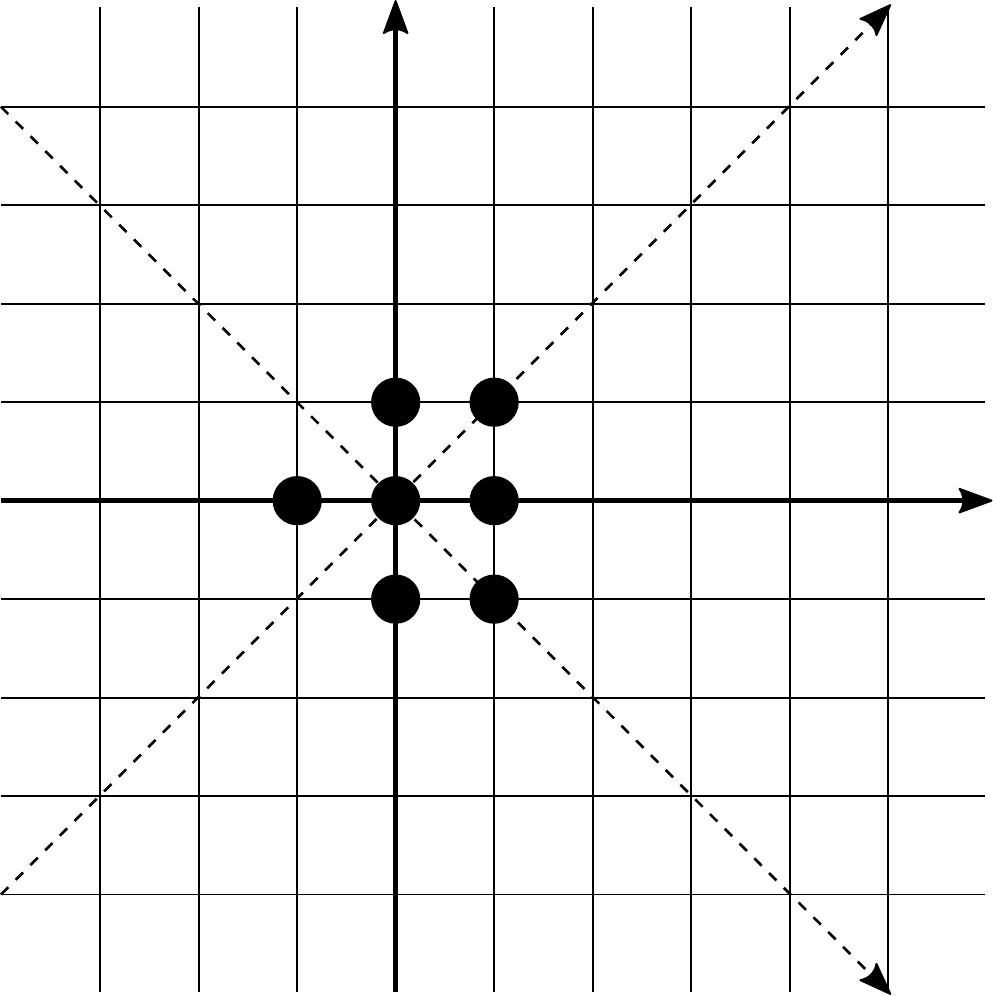}}
    \end{figure}
\end{example}

For the continuous case, given $u\overset{e.m.}{\sim}v$, one has
\begin{equation}
    u^* \equiv v^*.
\end{equation}
Then by P\'olya-Szeg\"o inequality, we know
\begin{equation}
    ||\nabla u^*||_2=\inf\{||\nabla v||_2:\ v \overset{e.m.}{\sim} u\}.
\end{equation}
One can also verify that this holds for one-dimensional discrete Schwarz rearrangement. However, this is not true for the discrete Schwarz rearrangement on $\Z^2$ defined above. In fact, by changing the iteration order of one-step rearrangements, we may get different ``symmetric'' functions with different Sobolev energies.
\begin{example}
    Let $\Omega=\{(1,1),(-1,0),(0,0),(1,0),(0,-1)\}$, $u=\mathbbm{1}_{\Omega}$ as shown in (a) of the following figure. One can verify that $R_{\Z^2}u=R_{e_1}u$ as shown in (b). If we define the iteration of one-step rearrangements in the following order:
    $$R'_{\Z^2}u=\lim_{k\to \infty}R'^{k}u=\cdots \circ R_{e_2} \circ R_{e_1} \circ R_{\frac{e_1-e_2}{2}} \circ R_{\frac{e_1+e_2}{2}}u,$$
    then $R'_{\Z^2} u=R_{\frac{e_1+e_2}{2}}u$ is shown in (c). Moreover $||R'_{\Z^2} u||_2 < ||R_{\Z^2} u||_2$.
    \begin{figure}[H]
        \centering
        \setcounter{subfigure}{0}
        \subfigure[represents the graph of $u$]{\includegraphics[width=0.3\textwidth]{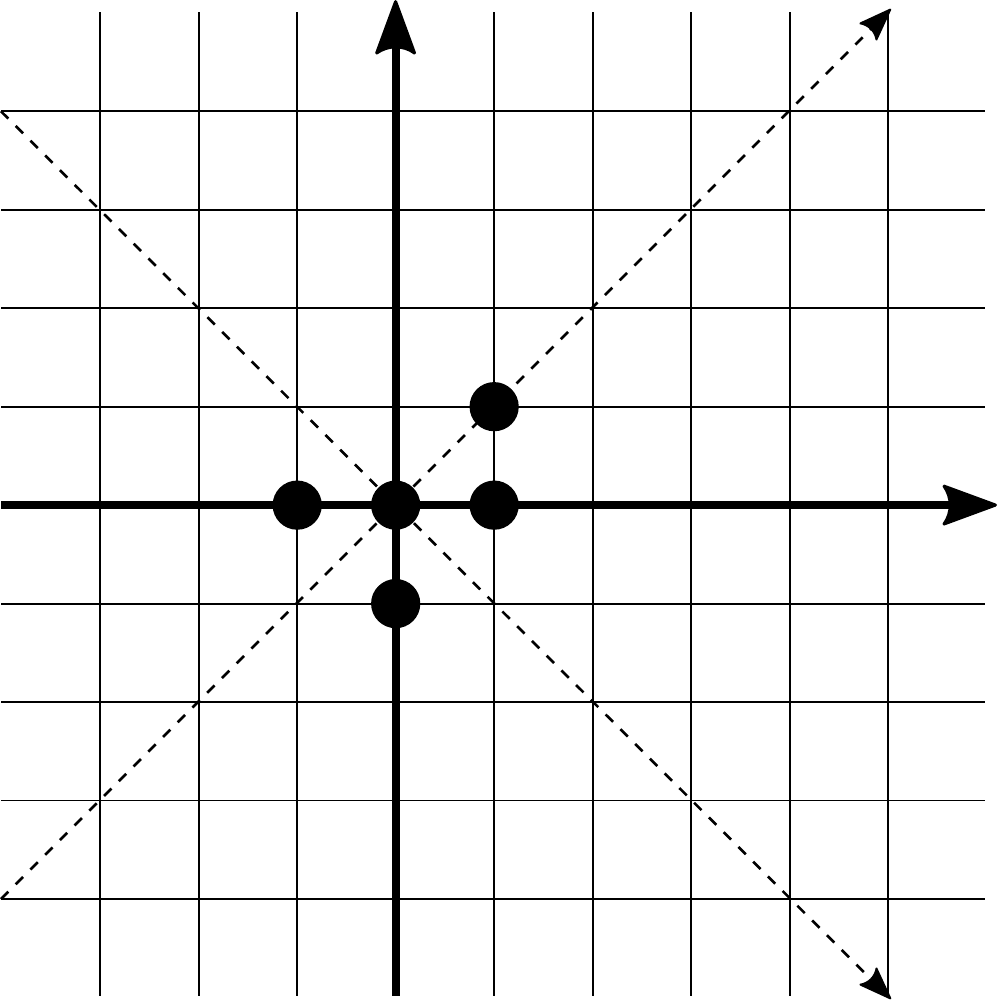}}\
        \subfigure[represents the graph of $R_{\Z^2} u$]{\includegraphics[width=0.3\textwidth]{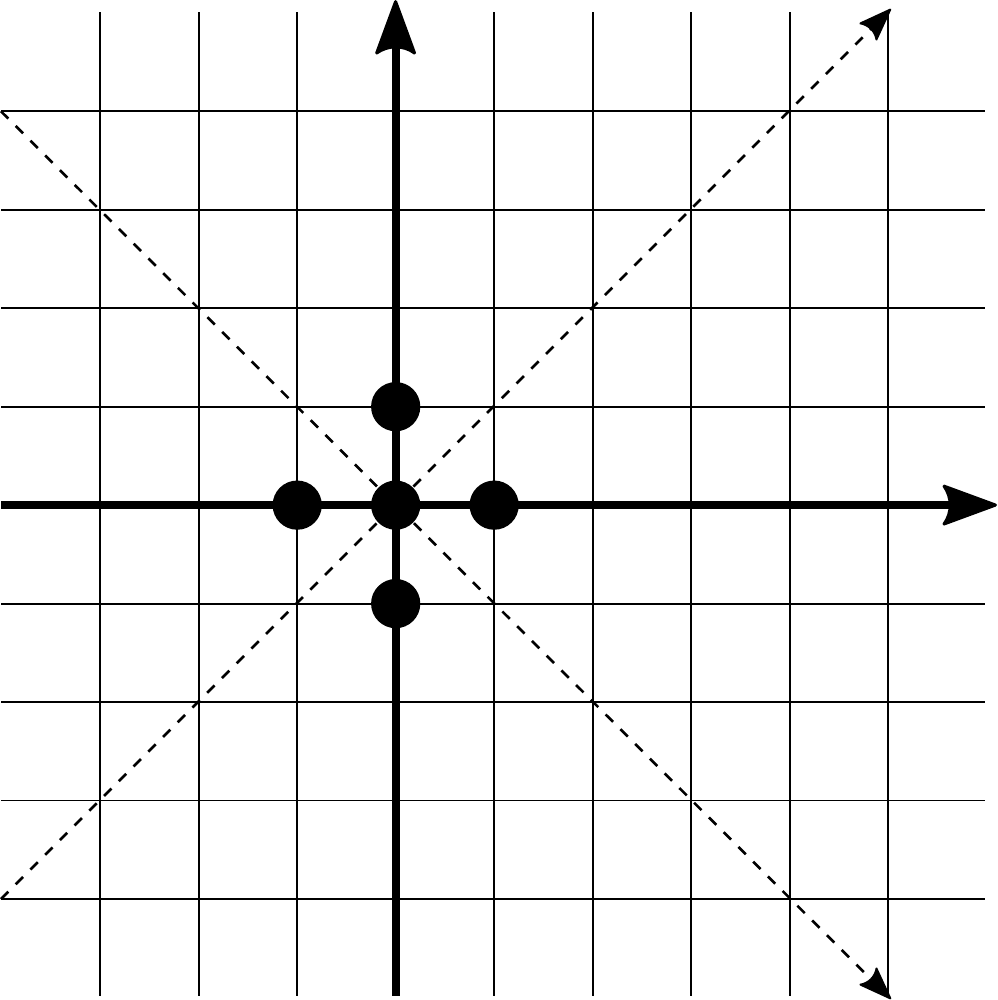}}\
        \subfigure[represents the graph of $R'_{\Z^2} u$]{\includegraphics[width=0.3\textwidth]{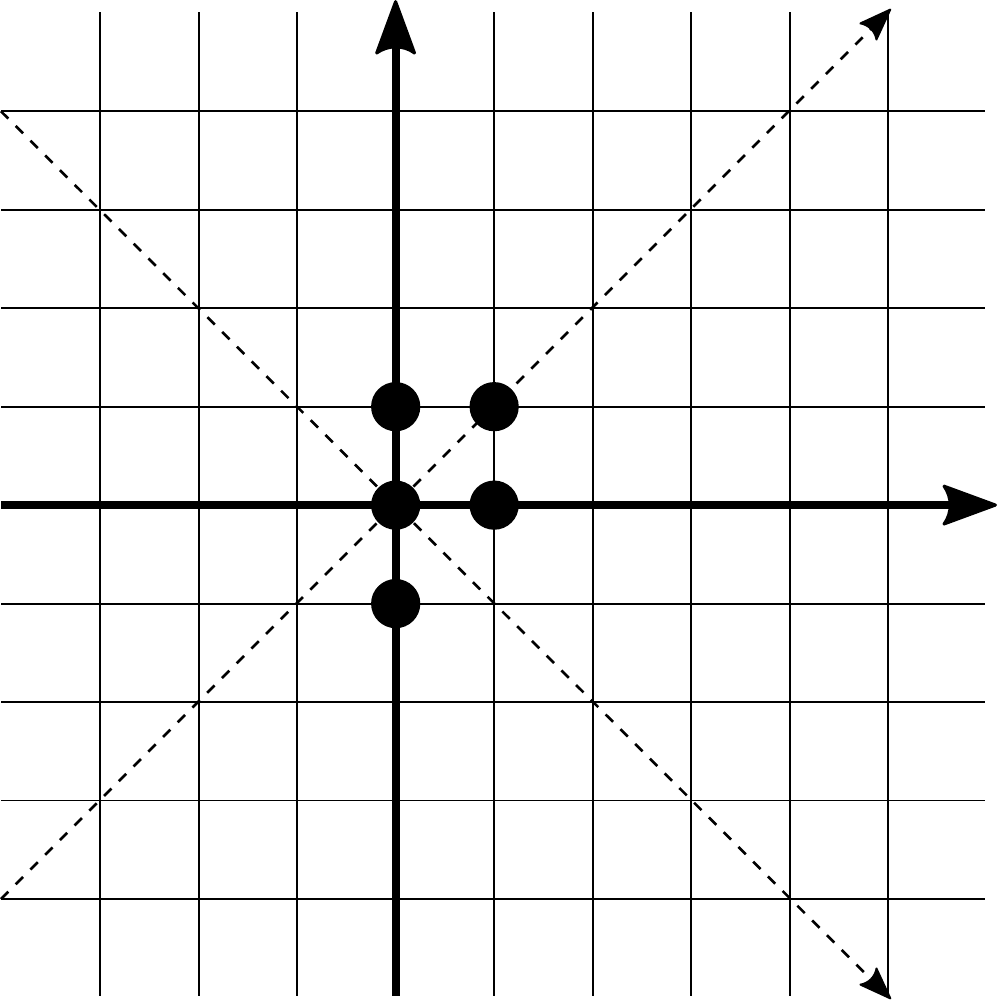}}\
    \end{figure}
\end{example}
Though we may not get the ``best'' equimeasurable symmetric function that minimizes the Sobolev energy by the discrete Schwarz rearrangement, the existence of the ``best'' function is proved; see Lemma~\ref{lemma:existence_of_best_symmetric_function}.

The choice of $\Vec{E}$ is crucial to get a ``satble'' function under the iteration of one-step rearrangements. The following example shows that the limit process introduced in \eqref{equ:definition_of_2_dimensional_discrete_Schwarz_rearrangement} in Definition~\ref{def} may fail if we change the direction of some vector in $\Vec{E}$.
\begin{example}
    Let $\Vec{E}'\!=\!\{\!-\!e_1\!,e_2\!,\frac{e_1\!+\!e_2}{2}\!,\frac{e_1\!-\!e_2}{2}\!\}$, $R_{-e_1}u(x^1\!,x^2)\!=\!R_{e_1}u(-x^1\!,x^2)$, define
    $$\Tilde{R}^k u=\underbrace{\cdots \circ R_{e_2} \circ R_{-e_1} \circ R_{\frac{e_1-e_2}{2}} \circ R_{\frac{e_1+e_2}{2}} \circ R_{e_2} \circ R_{-e_1}}_k u.$$
    Given $u=\mathbbm{1}_{\Omega}$, where $\Omega=\{(0,0),(0,1),(1,0),(0,-1),(-1,0),(1,-1)\}$, $\Tilde{R}^k u$ are shown as follows. The sequence $\Tilde{R}^{k} u (1,-1)$ is not convergent!
    \begin{figure}[H]
        \centering
        \setcounter{subfigure}{0}
        \subfigure[represents the graph of $u=\Tilde{R}^{4l} u$]{\includegraphics[width=0.3\textwidth]{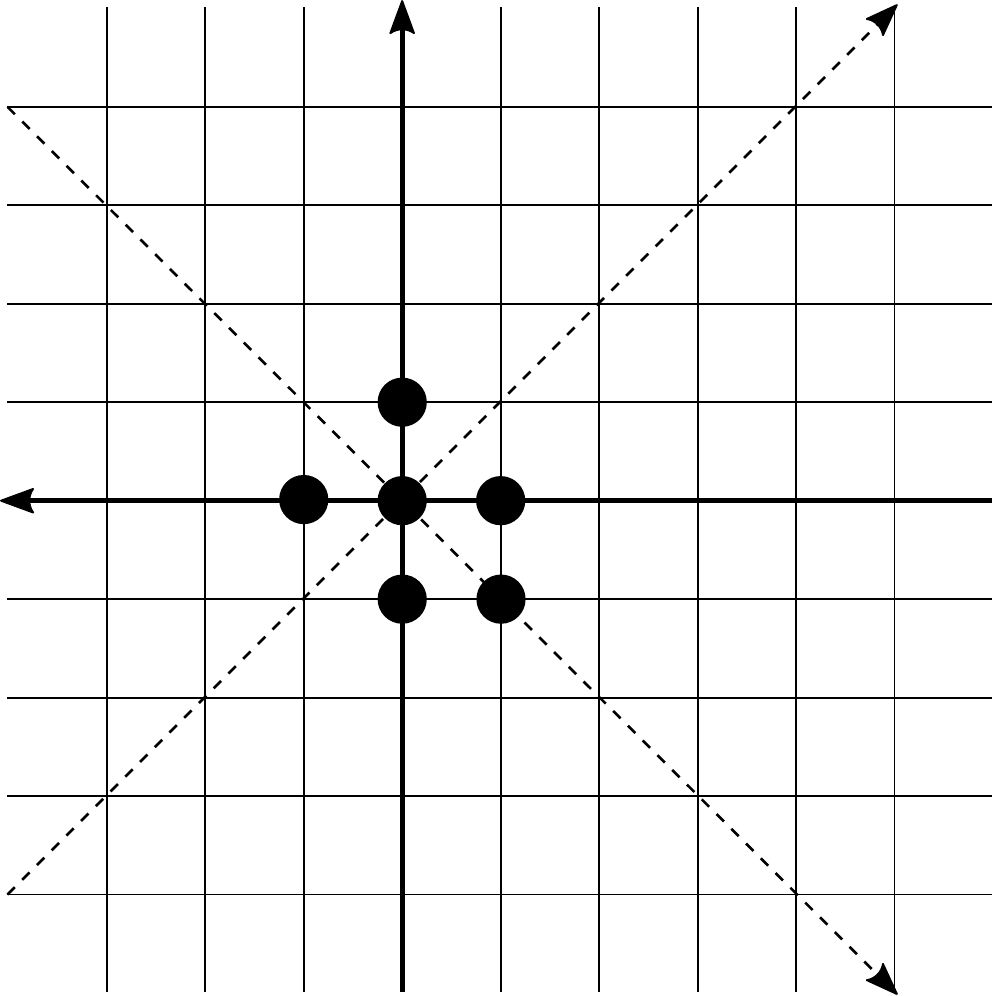}}\
        \subfigure[represents the graph of $\Tilde{R}^{4l+1} u$]{\includegraphics[width=0.3\textwidth]{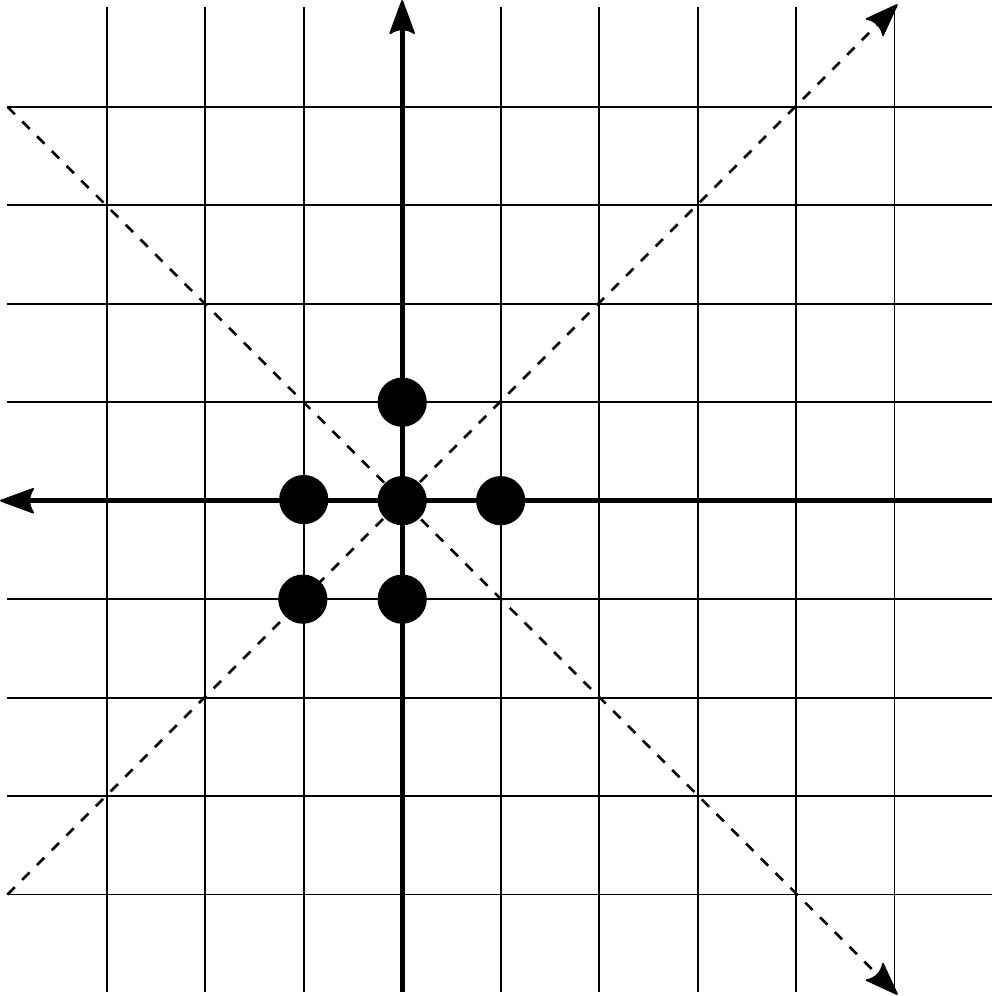}}\
        \subfigure[represents the graph of $\Tilde{R}^{4l+2} u=\Tilde{R}^{4l+3}u$]{\includegraphics[width=0.3\textwidth]{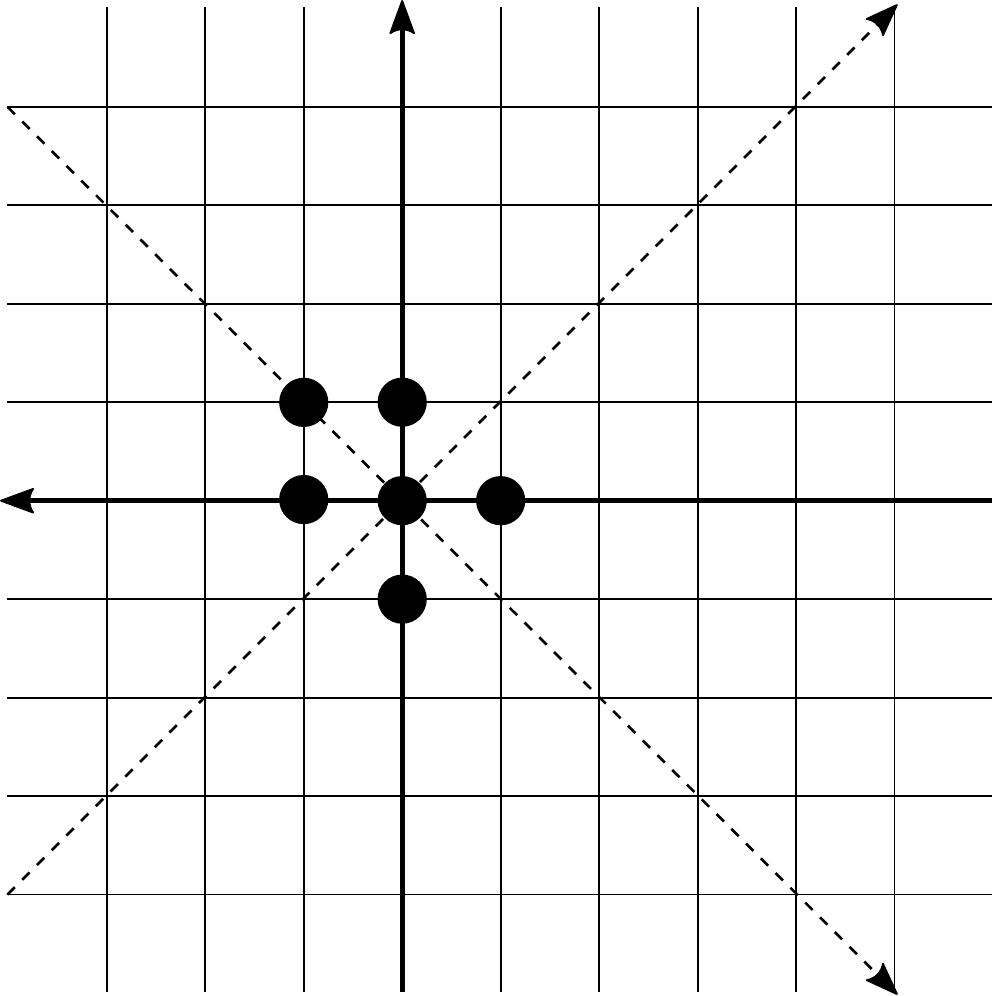}}\
    \end{figure}
\end{example}

\subsubsection{\bf{Well-definedness for finitely supported functions}}

We first prove that $R_{\Z^2}$ is well-defined for finitely supported functions on $\Z^2$. In fact, we have the following result:
\begin{lemma}
    \label{lemma:2_dimensional_discrete_Schwarz_rearrangment_is_well_defined_for_finitely_supported_functions}
    \label{lem-well_defined_Cc}
    Given $u \in C_c(\Z^2)$ be nonnegative, there exists a positive integer $K$ depending on $u$ such that
    $$R_e \circ R^K u = R^K u$$
    for all $e \in \Vec{E}$.
\end{lemma}
\begin{proof}
We prove the result by induction on $|\supp{u}|$. Suppose that the result holds for all functions with $|\supp| \le N-1 $. Let $u \in C_c(\Z^2)$ with $\ran(u)=[a_1 \succ a_2 \succ \cdots \succ a_N \succ 0 \succ \cdots ]$, $a_N>0$. Then $u$'s $(N-1)$-th cut-off $\hat{u}_{N-1}$ satisfies $|\supp{\hat{u}_{N-1}}| = N-1$, and there is a positive integer $K_0$ such that $R_e \circ R^{K_0} \hat{u}_{N-1} =R^{K_0} \hat{u}_{N-1}$ for $e\in \Vec{E}$. Note that the positions of $a_n$ in $R_e u$ and $R_e \hat{u}_{N-1}$ are the same for any $n<N$, we only need to prove that the position of $a_N$ is fixed after finite times of one-step rearrangements. Let $x_k=(x_k^1,x_k^2)\in \Z^2$ be the position of $a_N$ in $R^{K_0+k} u$. One easily verifies the following properties:
\begin{enumerate}
    \item $x_{k+1} \in \overline{V}(x_k)$;
    \item $\overline{V}(x_{k+1}) \subset \overline{V}(x_k)$;
    \item $\overline{V}(x_{k+1})=\overline{V}(x_k)$ if and only if $x_{k+1}\in V(x_k)$.
\end{enumerate}
Since $|\overline{V}(x_k)|$ is finite, there exists $K\in \N$ such that $\overline{V}(x_k)=\overline{V}(x_K)$, and $x_k\in V(x_K)$ for all $k \ge K$.

We only consider the case when $x_K^1 \ne 0$, $x_K^2 \ne 0$, and $|x_K^1| \ne |x_K^2|$, the proofs for other cases are similar. W.L.O.G, suppose that $x_K^1 > x_K^2 >0$. Let
\begin{align*}
    V^1=&\{(x_K^1,x_K^2)\}, \qquad \qquad &V^2&=\{(x_K^2,x_K^1),(x_K^1,-x_K^2)\},\\
    V^3=&\{(x_K^2,-x_K^1),(-x_K^2,x_K^1)\}, &V^4&=\{(-x_K^2,-x_K^1),(-x_K^1,x_K^2)\},\\
    V^5=&\{(-x_K^1,-x_K^2)\}.
\end{align*}

\begin{figure}[H]
    \centering
    \setcounter{subfigure}{0}
    \subfigure[$V(3,2)$]{\includegraphics[width=0.3\textwidth]{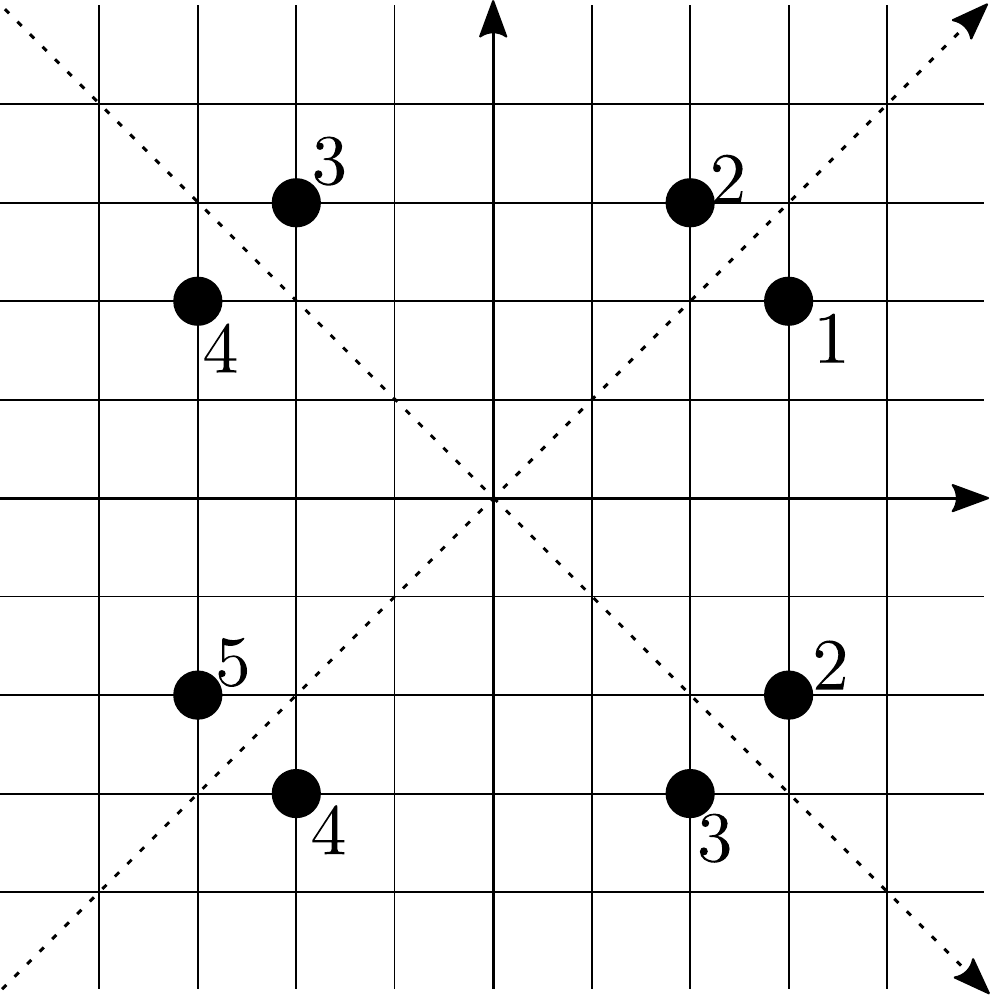}}\
    \subfigure[$V(3,0)$]{\includegraphics[width=0.3\textwidth]{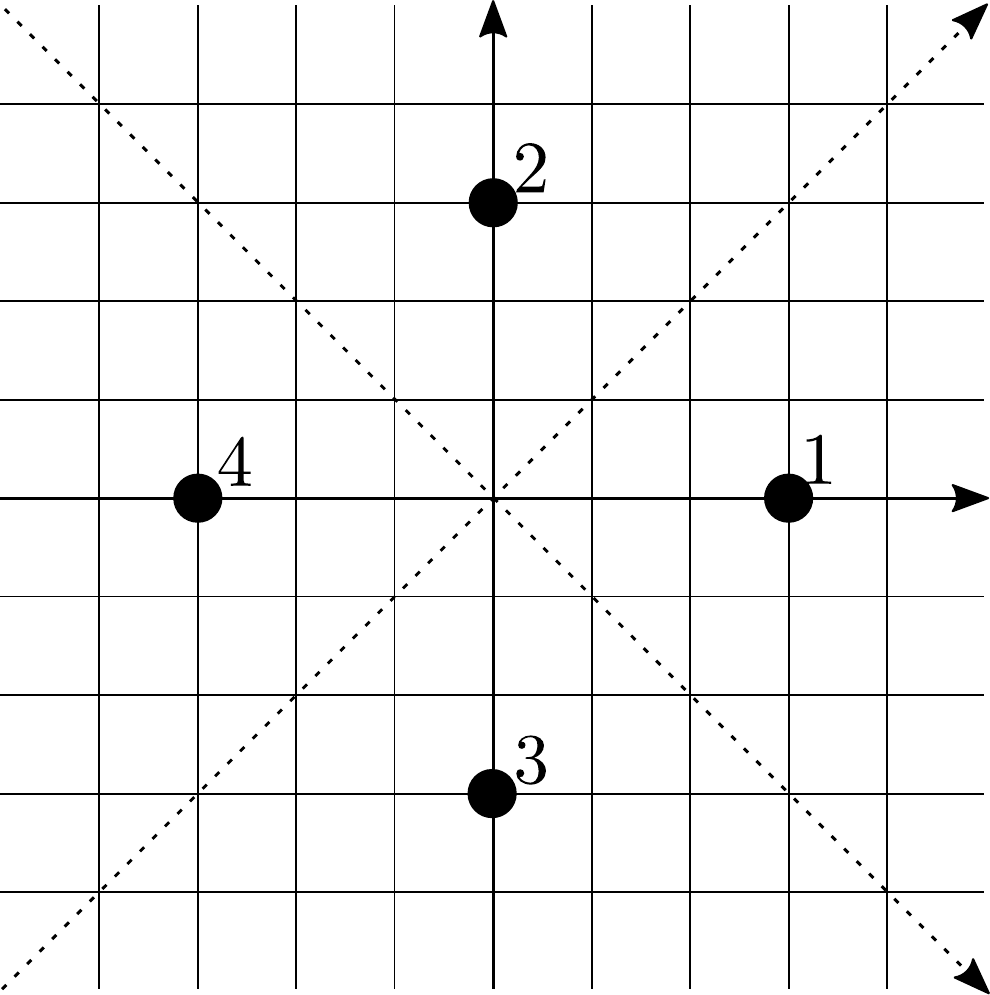}}\
    \subfigure[$V(3,3)$]{\includegraphics[width=0.3\textwidth]{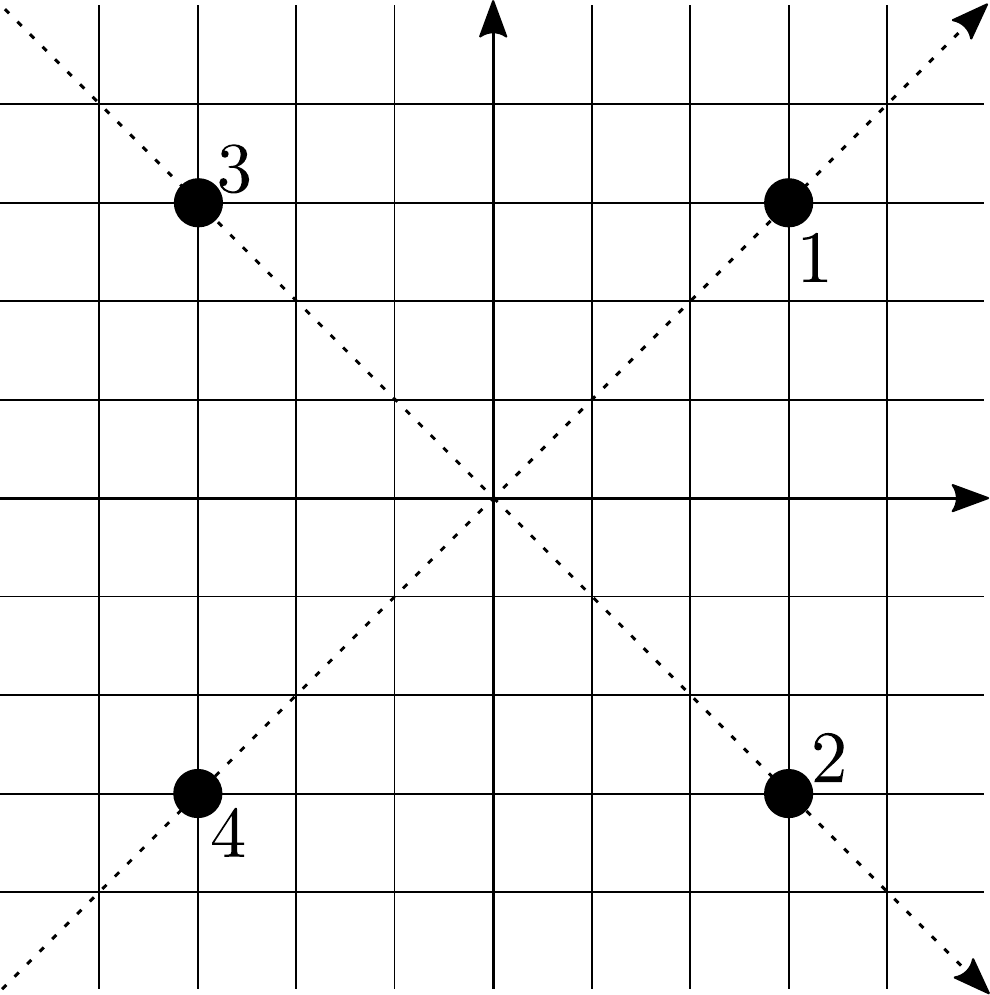}}\
    \caption{Function value $a_N$ at a vertex labeled $j$ can only moves to vertices labeled $j'$, $j'< j$.}
\end{figure}

Suppose that $x_K \in V^j$ for some $1\le j \le 5$, one can verify that:
\begin{enumerate}[(1)]
    \item For all $5\ge j'>j$, $R_e$ will not move $a_N$ to $V^{j'}$.
    \item If $V^{j-1}$ are not captured by $a_1,a_2,\cdots,a_{N-1}$, there exists an $R_e$ that moves $a_N$ to $V^{j-1}$.
\end{enumerate}
Since the positions of $a_1,\cdots, a_{N-1}$ are fixed, the position of $a_N$ is fixed after finitely number of one-step rearrangements, and this proves the lemma.
\end{proof}

\subsubsection{\bf{Geometry of the support of a rearranged finitely supported function}}

By the proof of Lemma~\ref{lem-well_defined_Cc}, one can verify the following properties:
\begin{prop}
    \label{prop-geom_C_c_basic1}
    Let $u\in C_c(\Z^2)$ be nonnegative, and $u^*=R_{\Z^2}u$.
    \begin{enumerate}[(i)]
        \item If $V(x)\in \supp{u^*}$, then $$\overline{V}(x) \in \supp{u^*}.$$
        \item If $\partial \overline{V}(x) \cap \supp{u^*}=\varnothing$, then $$(\overline{V}(x))^c \cap \supp{u^*} =\varnothing.$$
    \end{enumerate}
\end{prop}

Then we get the following result for the support of a rearranged $C_c$ function.
\begin{lemma}\label{thm-geom_Cc}
    Let $u \in C_c(\Z^2)$ be nonnegative with $|\supp{u}|=N$ large enough, $u^*=R_{\Z^2}u$, then
    $$V^{\Diamond}_{L_1} \subset \supp{ u^*} \subset V^{\Box}_{L_2},$$
    where $L_1=\left\lfloor \frac{\sqrt{N}-5}{2} \right\rfloor$ and $L_2=\left\lceil \sqrt{\frac{N}{2}} \right\rceil +2$.
\end{lemma}
\begin{proof}
    Let
    \begin{align*}
        X_1&=\max \{x:\ (x,0) \in \supp{u^*}\},\quad  &X_2=-\min \{x:\ (x,0) \in \supp{u^*}\},\\
        Y_1&=\max \{y:\ (0,y) \in \supp{u^*}\},  &Y_2=-\min \{y:\ (0,y) \in \supp{u^*}\}.
    \end{align*}
    Then
    \begin{equation}
        0 \le X_2 \le Y_2 \le Y_1 \le X_1 \le X_2+1.
    \end{equation}
    By the above inequality and Proposition~\ref{prop-geom_C_c_basic1}, we have
    $$V^{\Diamond}_{X_1-1} \subset V^{\Diamond}_{X_2} \subset \supp{u^*} \subset V^{\Box}_{X_1+1} \subset V^{\Box}_{X_2+2}.$$

    By (\ref{equ-num_box_dia}), we have
    $$2(X_1-1)^2+2(X_1-1)+1 \le N \le (2X_2+5)^2.$$
    Then $X_1 \le \left\lceil \sqrt{\frac{N}{2}} \right\rceil +1$ and $X_2 \ge \left\lfloor \frac{\sqrt{N}-5}{2} \right\rfloor$, which proves the lemma.
\end{proof}

\begin{corollary}
    \label{cor-geom_Cc}
    Let $u\in C_c(\Z^2)$, $N \in \N$ large enough, $u^*=R_{\Z^2}u$, then
    $$\inf_{(x,y)\in V^{\Diamond}_{L_1}}u^* \ge \sup_{(x,y) \notin V^{\Box}_{L_2}} u^*.$$
    where $L_1=\left\lfloor \frac{\sqrt{N}-5}{2} \right\rfloor$, $L_2=\left\lceil \sqrt{\frac{N}{2}} \right\rceil +2$.
\end{corollary}
\begin{proof}
    Let $\ran(u^*)=[a_1 \succ a_2 \succ \cdots]$, we prove the corollary by contradiction. Suppose that there exists $(x_0,y_0) \notin V^{\Box}_{L_2}$ satisfying
    $$u^*(x_0,y_0)=a_M > \inf\limits_{(x,y)\in V^{\Diamond}_{L_1}}u^* \ge 0.$$
    Let $v=\widehat{u^*}_M$ be the $M$-th cut-off of $u^*$, then
    \begin{enumerate}[(i)]
        \item $v=R_{\Z^2}v$.
        \item $|\supp{v}|=M$.
        \item $V^{\Diamond}_{L_1}\not\subset \supp{v}$. By Lemma~\ref{thm-geom_Cc}, $M < N$.
        \item $\supp{v} \not\subset V^{\Box}_{L_2}$. By Lemma~\ref{thm-geom_Cc}, $M > N$.
    \end{enumerate}
    From (iii) and (iv), we get a contradiction, and this proves the result.
\end{proof}

\subsubsection{\bf{Well-definedness for admissible functions}}

$~$

Now we are ready to prove that $R_{\Z^2}$ is well-defined for admissible functions on $\Z^2$.
\begin{lemma}\label{lem-def_fair}
Given $u\in C_0^+(\Z^2)$ be admissible, and $x_0 \in \Z^2$, there exists a positive integer $M$ such that
$$R^M u(x_0)=R_e\circ R^M u(x_0)$$
holds for all $e\in \Vec{E}$.
\end{lemma}
\begin{proof}
Let $\ran(u)=[a_1 \succ a_2 \succ \cdots]$, $L \in \N$ such that $x_0 \in V^{\Diamond}_L$, and $N\in \N$ such that $\left\lfloor \frac{\sqrt{N}-5}{2} \right\rfloor \ge L$. Let $\hat{u}_N$ be the $N$-th cut-off of $u$. Then for all $n\le N$,
$$(R_e u)^{-1}(a_n)=(R_e \hat{u}_N)^{-1}(a_n),$$
i.e. the position of $a_n$ in $R_e u$ is the same with that in $R_e \hat{u}_N$. By Lemma~\ref{lem-well_defined_Cc}, there is a positive integer $M$ such that $R^M \hat{u}_N=R_{\Z^2} \hat{u}_N$. Moreover, $x_0 \in \supp R_{\Z^2} \hat{u}_N$ by Lemma~\ref{thm-geom_Cc}. This proves the lemma.
\end{proof}

\subsection{Schwarz rearrangement on $\Z^d$}
In a similar way to Schwarz rearrangement on $\Z^2$, we define the discrete Schwarz rearrangement on $\Z^d$ for $d\ge 3$. Let $e_i$ be the $i$-th standard unit vector of $\R^d$, and let
$$\Vec{E}=\{e_i,\frac{e_i-e_j}{2},\frac{e_i+e_j}{2}: \ 1\le i< j \le d\}.$$
\begin{prop}
    \label{prop-positive_combination}
    Let $c: \Vec{E} \to \R_+$ be a nonnegative function on $\Vec{E}$, then
    $$\sum_{e\in \Vec{E}} c(e)e=0 \ \Leftrightarrow \  c\equiv 0.$$
\end{prop}
\begin{proof}
    Write $I(e)=i$ if $e=e_i$, $\frac{e_i+e_j}{2}$, or $\frac{e_i-e_j}{2}$, and let $I=\min\{I(e): c(e)\ne 0\}$. If $c\notequiv\:0$, then $I\le d$, and
    $$\sum_{e\in \Vec{E}} c(e)e=\sum_{i=I}^d c'_i e_i$$
    for some constants $c'_i$. Moreover, $c'_I >0$, and the result follows.
\end{proof}

Given $e\in \Vec{E}$, for any $x,y \in \Z^d$, we define $x\overset{e}{\sim} y$ if $(x-y)$ is parallel to $e$ in $\R^d$. Then $\overset{e}{\sim}$ is an equivalence relation, and there is an unique corresponding partition $\Z^d=\bigsqcup_{\alpha \in I_e} V_e^{\alpha}$. The map
\begin{align*}
    V_e^{\alpha} &\to \R,\\
    x &\mapsto <e,x>,
\end{align*}
is a bijection to $\Z$ or $\Z+1/2$. Then $R_{\Z}$ on $V_e^{\alpha}$ is well-defined. Define one-step rearrangement $R_e$ via
\begin{equation*}
    (R_e u)|_{V_e^{\alpha}}=R_{\Z}(u|_{V_e^{\alpha}}), \ u\in C_0^+(\Z^d).
\end{equation*}

\begin{definition}
    \label{def:definition_of_discrete_Schwarz_rearrangement}
    Let $R^k$ be the iteration of $k$ one-step rearrangements, the Schwarz rearrangement of an admissible function $u\in C_0^+(\Z^d)$ is defined via
    \begin{equation}
        R_{\Z^d} u(x):= \lim_{k\to \infty} R^k u(x), \ x\in \Z^d.
    \end{equation}
\end{definition}
Without causing ambiguity, we write $u^*:=R_{\Z^d} u$ for convenient.

\begin{lemma}
    \label{lemma:well-definedness_for_finitely_supported_functions}
    Let $u\in C_c(\Z^d)$ be nonnegative, then there exists $M\in \N$ such that
    $$R_e(R^M u)=R^M u$$
    holds for all $e\in \Vec{E}$.
\end{lemma}
\begin{proof}
    In an analog way to the proof of Lemma~\ref{lem-well_defined_Cc}, we prove by induction on $|\supp{u}|$. Given $x=(x^1,x^2,\cdots,x^d)\in \Z^d$, we write
    $$V(x)=\{(\pm x^{\mu_1},\pm x^{\mu_2},\cdots,\pm x^{\mu_d}):\ \text{$\mu$ is a permutation of $(1,2,\cdots,d)$}\}.$$
    Let $a\in \ran(u)$ be the smallest positive function value, $X_k$ be the position of $a$ in $R^k u$, then there is a $M\in \N$ such that $X_k \in V(X_M)$ for all $k \ge M$.

    Consider a directed graph $G=(V(X_M),E)$, where there is a directed edge from $x'$ to $x''$ if and only if there exists an $e\in \Vec{E}$ such that $<e,(x''-x')>=|e|\cdot|(x''-x')|$, i.e. the directions of $e$ and $(x''-x')$ are the same. Then $X_M \to X_{M+1} \to \cdots$ is a directed walk on the directed graph $G$, we only need to prove $G$ contains no directed cycles.

    Suppose that $Y_1\to Y_2 \to \cdots \to Y_K \to Y_{K+1}=Y_1$ is a directed cycle of $G$. Let $\Vec{v}_k=Y_{k+1}-Y_k$, then:
    \begin{enumerate}[(i)]
        \item $\sum_k \Vec{v}_k =0$;
        \item For each $\Vec{v}_k$, there exists an $e^k\in \Vec{E}$ and $C_k>0$ such that $\Vec{v}_k=C_k e^k$.
    \end{enumerate}
    By (ii) and Proposition~\ref{prop-positive_combination}, $\sum_k \Vec{v}_k \ne 0$, which is contradictory to (i). So $G$ contains no directed cycles, and this proves the lemma.
\end{proof}
Following the arguments of Lemma~\ref{lem-def_fair}, one can prove that $R_{\Z^d}$ is well-defined for $d \ge 3$, and we omit the proof here.

\subsection{Schwarz symmetric functions}
\begin{definition}[Schwarz symmetric function]
    An admissible function $u\in C_0^+(\Z^d)$ is called Schwarz symmetric if $u=R_{\Z^d} u$.
\end{definition}
Let
$$\mathcal{S}(\Z^d):=\{u \in C_0^+(\Z^d): \ R_{\Z^d} u=u\}$$
be the set of all Schwarz symmetric functions on $\Z^d$, and let
$$\mathcal{S}^p(\Z^d) := \mathcal{S}(\Z^d) \cap l^p(\Z^d).$$
Using the same arguments as Lemma~\ref{thm-geom_Cc}, we have the following result.
\begin{theorem}[Geometry of Schwarz symmetric functions]
    \label{thm:geometry_of_Schwarz_symmetric_functions}
    Given $N>0$, there exist two constants $L_1 < L_2$ depending on $N$ and $d$, such that
    \begin{equation}
        V_{L_1}^{\Diamond}\subset \supp{\hat{u}_N} \subset V_{L_2}^{\Box}
    \end{equation}
    holds for all $u\in \mathcal{S}(\Z^d)$, where $\hat{u}_N$ is the $N$-th cut-off of $u$. Moreover, if $|\supp{u}|=\infty$, we can choose $L_1 \to \infty$ as $N\to \infty$.
\end{theorem}

We have the following properties:
\begin{prop}
    The following properties hold.
    \begin{enumerate}[(i)]
        \item $u=R_{\mathbb{Z}^d} u$ $\Leftrightarrow$ $R_e u = u $ for all $e\in \Vec{E}$ $\Leftrightarrow$ $R_{\Z}(u|_{V_e^{\alpha}})=u|_{V_e^{\alpha}}$ for all $V_e^{\alpha}$.
        \item Any admissible function $u$ is equimeasurable with $R_e u$ for $e\in \Vec{E}$, $R^k u$ for any integer $k$, and $R_{\mathbb{Z}^d} u.$
        \item $\mathcal{S}(\Z^d)=R_{\Z^d}(C_0^+(\Z^d))$, $\mathcal{S}^p(\Z^d)=R_{\Z^d}(l^{p,+}(\Z^d))$, where $l^{p,+}(\Z^d)$ is the set of nonnegative $l^p$ functions on $\Z^d$.
        \item Suppose that $u_j \in \mathcal{S}(\Z^d)$ and $u_j \to u$ pointwisely, then $u \in \mathcal{S}(\Z^d)$.
        \item Suppose that $u\in \mathcal {S}(\Z^d)$ with $|\supp{u}|=\infty$, then $u$ is positive.
    \end{enumerate}
\end{prop}

\begin{proof}
    (i): It is obtained directly from the definition of one-step rearrangements and $R_{\Z^d}$.\\
    (ii): $u\overset{e.m.}{\sim} R_e u \overset{e.m.}{\sim} R^k u$ follows from the definitions. We only need to prove that $u\overset{e.m.}{\sim} R_{\Z^d} u$. Let $t>0$, W.O.L.G., we assume that $\ran(u)=[a_1 \succ a_2 \succ \cdots]=\ran(R^k u)$ with $a_N \ge t$ and $a_{N+1} \le t-\epsilon$ for some $\epsilon>0$. Let $\hat{u}_N$ be the $N$-th cut-off of $u$, then $R_{\Z^d} (\hat{u}_N) = \widehat{(R^k u)}_N$ for all $k$ large enough. For any $x\notin \supp{R_{\Z^d} (\hat{u}_N)}$, $R^k u(x) \le t-\epsilon$ for all $k$ large enough, which shows that $R_{\Z^d}(x) \le t-\epsilon$. This implies that $|\{x\in \Z^d: \R_{\Z^d}u(x) \ge t \}|=N$, i.e. $u\overset{e.m.}{\sim} R_{\Z^d} u$.\\
    (iii): We only need to prove that $R_{\Z^d}u \in \mathcal{S}(\Z^d)$ for any admissible function $u$. This follows from (i) and the fact that $\widehat{u^*}_N \in \mathcal{S}(\Z^d)$ for all $N\in\N$. \\
    (iv): We need to prove that $R_{\Z}(u|_{V_e^{\alpha}})=u|_{V_e^{\alpha}}$ for all $V_e^{\alpha}$. Write $u'=u|_{V_e^{\alpha}}$, $u'_j=u_j|_{V_e^{\alpha}}$, which are functions defined on $\Z$ or $\Z+1/2$. Then for all $x,y \in \Z$ or $\Z+1/2$ with $|x| \le |y|$, we have $u'_j(x) \ge u'_j(y)$ and $u'_j(|x|) \ge u'_j(-|x|)$. This implies $u'(x) \ge u'(y)$ and $u'(|x|) \ge u'(-|x|)$, i.e. $R_{\Z}(u|_{V_e^{\alpha}})=u|_{V_e^{\alpha}}$.\\
    (v): We prove it by contradiction. Suppose that there exists $x\in \Z^d$ such that $u(x)=0$. By Theorem~\ref{thm:geometry_of_Schwarz_symmetric_functions}, we can choose $N$ large enough such that $x\in V_{L_1}^{\Diamond}\subset \supp{\hat{u}_N}$ since $|\supp{u}|=\infty$, then $u(x)>0$. The result follows.
\end{proof}

Let $1\le p < q < +\infty$, then $\mathcal{S}^p(\Z^d) \subset \mathcal{S}^q(\Z^d)$ since $l^p(\Z^d) \subset l^q(\Z^d)$. Moreover, we can prove a very useful compact embedding result:
\begin{theorem}
    \label{thm:compact_embedding_for_rearranged_functions}
    Let $\{u_i\}$ be a sequence in $\mathcal{S}^p(\Z^d)$ with $||u_i||_p \le 1$, $p\ge 1$. Then for all $q>p$, there exists a sub-sequence of $\{u_i\}$ convergent in $l^q(\Z^d)$.
\end{theorem}
\begin{proof}
    W.L.O.G., we assume that $u_i$ is convergent pointwisely. By Theorem~\ref{thm:geometry_of_Schwarz_symmetric_functions}, for any $u \in \mathcal{S}^p(\Z^d)$ and $N\in \N$, we have
    $$\inf_{x\in V_{L_1}^{\Diamond}}u(x) \ge \sup_{y \notin V_{L_2}^{\Box}}u(y),$$
    where $L_1 \le L_2$ are constants depending only on $N$ and $d$. Then for all $x \notin V_{L_2}^{\Box}$, we have
    $$|u_i(x)|^p \le \frac{1}{|V_{L_1}^{\Diamond}|}.$$
    Then
    \begin{align*}
        \sum_{x \notin V_{L_2}^{\Box}} |u_i(x)|^q = & \sum_{x \notin V_{L_2}^{\Box}} |u_i(x)|^p  |u_i(x)|^{q-p}\\
        \le & \left(\frac{1}{|V_{L_1}^{\Diamond}|}\right)^{\frac{q-p}{p}} \sum_{x \notin V_{L_2}^{\Box}} |u_i(x)|^p\\
        \le & \left(\frac{1}{|V_{L_1}^{\Diamond}|}\right)^{\frac{q-p}{p}} \sum_{x \in \Z^d} |u_i(x)|^p\\
        \le & \left(\frac{1}{|V_{L_1}^{\Diamond}|}\right)^{\frac{q-p}{p}}.
    \end{align*}
    For any $\varepsilon > 0$, choose $N$ large enough such that $\left(\frac{1}{|V_{L_1}^{\Diamond}|}\right)^{\frac{q-p}{p}} < \varepsilon$, and $K$ large enough such that for all $i,j > K$ the following holds:
    $$\sum_{x \in V_{L_2}^{\Box}}|u_i(x)-u_j(x)|^q \le \varepsilon.$$
    Then
    \begin{align*}
        ||u_i-u_j||_q^q  = & \sum_{x \in V_{L_2}^{\Box}}|u_i(x)-u_j(x)|^q + \sum_{x \notin V_{L_2}^{\Box}}|u_i(x)-u_j(x)|^q\\
        \le & \ \varepsilon + \sum_{x \notin V_{L_2}^{\Box}} \left(|u_i(x)|^q +|u_j(x)|^q\right)\\
        \le & \ 3\varepsilon.
    \end{align*}
    This proves the result.
\end{proof}

Note that the above theorem does not hold for $q=p$. The following sequence serves as a counterexample.
\begin{equation*}
    u_i(x)=
    \begin{cases}
    \frac{1}{|V^{\Diamond}_i|^{1/p}}, \qquad & x \in V^{\Diamond}_i;\\
    0, & \text{others.}
    \end{cases}
\end{equation*}

\section{Discrete rearrangement inequalities}
Using the ideas developed in the previous sections, we show some classic rearrangement inequalities in the discrete setting in this section. Without loss of generality, we will assume that the integrands are continuous. The results for non-continuous functions can be obtained by using slicing techniques developed in \cite{burchard2006rearrangement}.

Note that an admissible function $u$ is equimeasurable with $R_e u$ for $e\in \Vec{E}$, $R^k u$ for any integer $k$, and $R_{\mathbb{Z}^d} u.$
\begin{prop}[Discrete Cavalieri's Principle]
    \label{prop:discrete_Cavalieri's_principle}
    Let $f\colon\R_+\to\R_+,$ $u\in C_0^+(\Z^d)$ be admissible, $u^*=R_{\Z^d} u$. If $f(0)=0,$ then:
    \begin{equation}\label{Cavalieri's}
        \sum_{x}f\left(u(x)\right)=\sum_{x}f\left(u^*(x)\right).
    \end{equation}
\end{prop}
\begin{proof}
  This is a direct consequence of the equimeasurability of $u$ and $u^*$.
\end{proof}

\begin{prop}
    Let $u,v$ be admissible with $u(x) \le v(x)$ for all $x\in \Z^d$, then
    $$R_{\Z^d}\,u(x) \le R_{\Z^d}\,v(x)$$
    for all $x\in \Z^d$.
\end{prop}
\begin{proof}
    By the definition of one-step rearrangement, we have $R_e u(x) \le R_e v(x)$ for all $x\in \Z^d$ and $e\in \Vec{E}$, then $R^k u(x) \le R^k v(x)$ for all $x\in \Z^d$ and $k\in \N$. The result follows since $R^k u \to R_{\Z^d} u$ and $R^k v \to R_{\Z^d} v$ pointwisely.
\end{proof}

\begin{definition}
    \label{def:supermodular}
    A function $G\colon\R_+\times\R_+\to\R$ is supermodular if
    \begin{equation}\label{ineq4.7}
        G(s+s_0,t+t_0)+G(s,t)\geq G(s,t+t_0)+G(s+s_0,t)
    \end{equation}
    for any $s,t,s_0,t_0\in\R_+.$\\
    $G$ is said to be strictly supermodular if \eqref{ineq4.7} holds with a strict inequality when $s_0,t_0>0$.
\end{definition}

\begin{prop}
    The following functions are supermodular.
    \begin{enumerate}[(i)]
        \item $G(s,t)=f(s)g(t), \ (s,t)\in \R_+ \times \R_+$, $f$ and $g$ are both increasing or decreasing. Particularly, $G(s,t)=s t$ is strictly supermodular.
        \item $G(s,t)=-J(|s-t|), \ (s,t)\in \R_+ \times \R_+$, where $J$ is an increasing convex function. Particularly, $G(s,t)=-|s-t|^p,\ p\ge 1$, is supermodular. Moreover, if $p>1$, $G(s,t)=-|s-t|^p$ is strictly supermodular.
    \end{enumerate}
\end{prop}
\begin{proof}
    (i) is obtained directly by the definition of supermodular function. For (ii), let $s,t,s_0,t_0 \ge 0$. W.L.O.G., we assume that $A=s-t\ge 0$. Only need to prove
    \begin{equation}
        J(|A+s_0 -t_0|)+J(|A|) \le J(|A+ s_0|)+J(|A-t_0|).
    \end{equation}
    If $t_0\le A$, it is obtained by the convexity of $J$. If $t_0 \ge A+s_0$, it is obtained by the monotonicity of $J$. Otherwise
    \begin{align*}
        J(|A+s_0 -t_0|)+J(|A|)
        &\le J(A+s_0 -t_0)+J(t_0)\\
        &\le J(A+s_0)+J(0)\\
        &\le J(A+s_0)+J(t_0-A).
    \end{align*}
    This shows $G$ is supermodular. If $p>1$, $G(s,t)=-|s-t|^p$ is strictly supermodular follows by the strict convexity of $|\cdot|^p$.
\end{proof}

Next, we start to prove the main result. We begin with the one-dimensional case as shown below, which is proved in Appendix A by the discrete polarization method introduced by the first author\cite{HH10}.
\begin{lemma}
    \label{lemma:one-dimensional_discrete_Riesz_inequality}
    \label{lemma:one-dimensional_discrete_Riesz_inequality_for_finitely_supported_functions}
    Let $H:\R_+ \to \R_+$ be a decreasing function, $G:\R_+ \times \R_+ \to \R_+$ be a supermodular function.
    \begin{enumerate}[(i)]
        \item If $u,v \in C_c(\Z)$ are two nonnegative functions, then:
        \begin{equation*}
            \sum_{x\in \Z} \sum_{y\in \Z}G(u(x),v(y))H(|x-y|) \le \sum_{x\in \Z} \sum_{y\in \Z}G(R_{\Z}u(x),R_{\Z}v(y))H(|x-y|).
        \end{equation*}
        \item If $u,v \in C_c(\Z+1/2)$ are two nonnegative functions, then:
        \begin{align*}
            &\sum_{x\in \Z+1/2} \sum_{y\in \Z+1/2}G(u(x),v(y))H(|x-y|) \\
            \le &\sum_{x\in \Z+1/2} \sum_{y\in \Z+1/2}G(R_{\Z+1/2}u(x),R_{\Z+1/2}v(y))H(|x-y|).
        \end{align*}
        \item If $u \in C_c(\Z)$, $v\in C_c(\Z+1/2)$ are two nonnegative functions, then:
        \begin{align*}
            &\sum_{x\in \Z} \sum_{y\in \Z+1/2}G(u(x),v(y))H(|x-y|) \\
            \le &\sum_{x\in \Z} \sum_{y\in \Z+1/2}G(R_{\Z}u(x),R_{\Z+1/2}v(y))H(|x-y|).
        \end{align*}
    \end{enumerate}
\end{lemma}

\begin{lemma}[Discrete Riesz inequality for finitely supported functions]
 \label{lemma:discrete_Riesz_inequality_for_finitely_supported_functions}

 $~$

    Let $u,v \in C_c(\Z^d)$ be two nonnegative functions, $H:\R_+ \to \R_+$ be a decreasing function, then
    \begin{equation}
        \sum_{x,y \in \Z^d}G(u(x),v(y))H(d(x,y)) \le \sum_{x,y \in \Z^d}G(u^*(x),v^*(y))H(d(x,y)).
    \end{equation}
\end{lemma}
\begin{proof}
    Given $e\in \Vec{E}$, let $\{V_e^{\alpha}\}_{\alpha \in I_e}$ be the partition of $\Z^d$ with respect to $\overset{e}{\sim}$. Then
    \begin{equation}
        \sum_{x,y \in \Z^d}G(u(x),v(y))H(d(x,y))=\sum_{\alpha,\beta \in I_e} \sum_{x\in V_e^{\alpha},y\in V_e^{\beta}}G(u(x),v(y))H(d(x,y)).
    \end{equation}
    We only need to prove that for any $\alpha,\beta \in I_e$, the following holds:
    \begin{equation}
        \sum_{x\in V_e^{\alpha},y\in V_e^{\beta}}\!\!G(u(x),v(y))H(d(x,y)) \le \sum_{x\in V_e^{\alpha},y\in V_e^{\beta}}\!\!G(R_e u(x),R_e v(y))H(d(x,y)).
    \end{equation}

    \underline{Case I:} $e=e_i$ be an unit vector. For $l,m\in \Z$, we write $x_l\in V_e^{\alpha}$ for the vertex in $V_e^{\alpha}$ with $<\!e,x_l\!>\:=l$, and $y_m\in V_e^{\beta}$ for the vertex in $V_e^{\beta}$ with $<\!e,y_m\!>\:=m$. Then $d(x_l,y_m)=|l-m|+d(x_0,y_0)$. Let $\Tilde{H}(|l-m|)=H(|l-m|+d(x_0,y_0))$, which is a nonincreasing function. Regard $u|_{V_e^{\alpha}}$ and $v|_{V_e^{\beta}}$ as functions on $\Z$, by Lemma~\ref{lemma:one-dimensional_discrete_Riesz_inequality}, we have
    \begin{align*}
        \sum_{x\in V_e^{\alpha},y\in V_e^{\beta}}G(u(x),v(y))H(d(x,y))
        &=\sum_{l,m\in \Z}G(u(x_l),v(y_m))\Tilde{H}(|l-m|)\\
        &\le \sum_{l,m\in \Z}G(R_e u(x_l),R_e v(y_m))\Tilde{H}(|l-m|)\\
        &=\!\!\sum_{x\in V_e^{\alpha},y\in V_e^{\beta}}\!\!G(R_e u(x),R_e v(y))H(d(x,y)).
    \end{align*}

    \underline{Case II:} $e=\frac{e_i+e_j}{2}$ or $\frac{e_i-e_j}{2}$. Let $x_l=(x_l^1,x_l^2,\cdots, x_l^d) \in V_e^{\alpha}$ with \\$<e,x_l>=l$, $l\in \Z$ or $\Z+1/2$; $y_m=(y_m^1,y_m^2,\cdots,y_m^d) \in V_e^{\beta}$ with $<e,y_m>=m$, $m\in \Z$ or $\Z+1/2$. Note that: $x_l^k$ only depends on $\alpha$ for $k\ne i,j$; $y_m^k$ only depends on $\beta$ for $k\ne i,j$; $x_l^i+x_l^j=2l$; $y_m^i+y_m^j=2m$; $x_l^i-l$ only depends on $\alpha$; $y_m^i-m$ only depends on $\beta$. The combinatorial distance
    \begin{align*}
        d(x_l,y_m)
        =&\sum_{k\ne i,j} |x_l^k-y_m^k|+|x_l^i-y_m^i|+|x_l^j-y_m^j|\\
        =&\sum_{k\ne i,j} |x_l^k-y_m^k|+|(l-m)+(x_l^i-l-y_m^i+m)|\\
        &+|(l-m)-(x_l^i-l-y_m^i+m)|\\
        =&\sum_{k\ne i,j} |x_l^k-y_m^k|+\max\{2|l-m|,2|x_l^i-l-y_m^i+m|\}\\
        =& A+\max\{2|l-m|,2B\},
    \end{align*}
    where $A:=\sum_{k\ne i,j} |x_l^k-y_m^k|$, $B:=2|x_l^i-l-y_m^i+m|$, only depend on $\alpha$ and $\beta$. Let $\Tilde{H}(|l-m|)=H(A+\max\{2|l-m|,2B\})=H(d(x_l,y_m))$, the result follows by the same argument of Case I.

    This proves the result.
\end{proof}

Now we are ready for the proof of the main result: The generalized discrete Riesz inequality. For the reader's convenience, we restate the theorem here.
\begin{theorem}[Generalized discrete Riesz inequality]
    \label{thm:discrete_Riesz_inequality}
    Let $u,v\in C_0^+(\Z^d)$ be admissible, $u^*=R_{\Z^d} u$, $v^*=R_{\Z^d} v$, $H:\R_+ \to \R_+$ be a decreasing function, and $G:\R_+ \times \R_+ \to \R_+$ be a supermodular function with $G(0,0)=0$. Then
    \begin{equation}
        \label{equ:discrete_Riesz_inequality}
        \sum_{x,y \in \Z^d}G(u(x),v(y))H(d(x,y)) \le \sum_{x,y \in \Z^d}G(u^*(x),v^*(y))H(d(x,y)).
    \end{equation}
    Moreover, if $u=u^*$ with $\ran u=[a_1>a_2>\cdots]$, $G$ is strictly supermodular, $H(n)>H(n+1)$ for some $n\in \N$, $|\sum\limits_{x,y \in \Z^d}G(u^*(x),v^*(y))H(d(x,y))|<\infty$, then the equality holds if and only if $v=v^*$.
\end{theorem}
\begin{proof}
    Suppose that $\sum_{x,y \in \Z^d}G(u^*(x),v^*(y))H(d(x,y))<\infty$, otherwise the result is trivial. Let
    $$\widetilde G(u,v) = G(u,v)-G(u,0)-G(0,v).$$
    Then
    \begin{enumerate}[(i)]
        \item $\widetilde G:\R_+ \times \R_+ \to \R_+$ is a supermodular function with $\widetilde G(\cdot,0) \equiv 0$ and $\widetilde G(0,\cdot) \equiv 0$;
        \item $\widetilde G(u,v)$ is nondecreasing on $u$ and $v$;
        \item \eqref{equ:discrete_Riesz_inequality} holds if and only if
        $$\sum_{x,y \in \Z^d}\widetilde G(u(x),v(y))H(d(x,y)) \le \sum_{x,y \in \Z^d}\widetilde G(u^*(x),v^*(y))H(d(x,y)).$$
    \end{enumerate}
    (i) is obtained directly by the definition of $\widetilde G$. Given $u_0>0$, $\widetilde G(u+u_0,v)=\widetilde G(u+u_0,v)+\widetilde G(u,0) \ge \widetilde G(u,v)+\widetilde G(u+u_0,0)=\widetilde G(u,v)$, so (ii) holds.
    Note that
    \begin{align*}
        &\sum_{x,y \in \Z^d}G(u(x),0)H(d(x,y))+\sum_{x,y \in \Z^d}G(0,v(y))H(d(x,y))\\
        =&\sum_{x,y \in \Z^d}G(u^*(x),0)H(d(x,y))+\sum_{x,y \in \Z^d}G(0,v^*(y))H(d(x,y))\\
        \le & \sum_{x,y \in \Z^d}G(u^*(x),v^*(y))H(d(x,y)) < \infty.
    \end{align*}
    by Proposition~\ref{prop:discrete_Cavalieri's_principle}, then (iii) holds.

    Let $B_i=\{x\in \Z^d: ||x||_1\le i \}$ be the $i$-ball, $u_i=\mathbbm{1}_{B_i} u$, $u_i^*=R_{\Z^d}u_i$, $v_i=\mathbbm{1}_{B_i} v$, and $v_i^*=R_{\Z^d}v_i$. Then $u_{i+1}(x) \ge u_{i}(x)$, $u_{i+1}^*(x) \ge u_{i}^*(x)$ for all $x\in \Z^d$; $v_{i+1}(y) \ge v_{i}(y)$, $v_{i+1}^*(y) \ge v_{i}^*(y)$ for all $y\in \Z^d$. Note that $\widetilde G(u,v)$ is nondecreasing on $u$ and $v$, then \eqref{equ:discrete_Riesz_inequality} holds by Lemma~\ref{lemma:discrete_Riesz_inequality_for_finitely_supported_functions} and monotone convergence.

    Next we prove the equality cases. For that, we only need to  prove that $v|_{V^{\alpha}_e} = (v|_{V^{\alpha}_e})^*$ for all $V^{\alpha}_e$. It is sufficient to show  the result for $e=e_1$, the proofs for the other cases are similar. Note that if the equality holds in \eqref{equ:discrete_Riesz_inequality}, then
    \begin{equation}
        \sum_{x,y \in V_{e_1}^{\alpha}}G(u(x),v(y))H(d(x,y)) = \sum_{x,y \in V_{e_1}^{\alpha}}G(u^*(x),R_{e_1}v(y))H(d(x,y)),
    \end{equation}
    for all $V_{e_1}^{\alpha}$, then the result follows by Lemma~\ref{lemma:one_dimensional_generlized_Risez_inequality}.
\end{proof}

Following the proof of Theorem~\ref{thm:discrete_Riesz_inequality}, we have the following result.
\begin{thmbis}{thm:discrete_Riesz_inequality}\label{thm:discrete_Riesz_inequality_b}
    Let $u,v\in C_0^+(\Z^d)$ be admissible, $u^*=R_{\Z^d} u$, $v^*=R_{\Z^d} v$, $H:\R_+ \to \R_+$ be a decreasing function, and $G:\R_+ \times \R_+ \to \R$ be a supermodular function with $G(0,0)=0$. If
    $$|\sum_{x,y \in \Z^d} G(u(x),0)H(d(x,y))|<\infty \text{, and } |\sum_{x,y \in \Z^d} G(0,v(y))H(d(x,y))|<\infty,$$
    then
    \begin{equation}
        \sum_{x,y \in \Z^d}G(u(x),v(y))H(d(x,y)) \le \sum_{x,y \in \Z^d}G(u^*(x),v^*(y))H(d(x,y)).
    \end{equation}
    Moreover, if $u=u^*$ with $\ran u=[a_1>a_2>\cdots]$, $G$ is strictly supermodular, $H(n)>H(n+1)$ for some $n\in \N$, $|\sum_{x,y \in \Z^d}G(u(x),v(y))H(d(x,y))|<\infty$, then the equality holds if and only if $v=v^*$.
\end{thmbis}

For the continuous situation, the assumption that $G$ is supermodular is necessary as shown in \cite{hajaiej2011necessity} by constructing a counter-example. In the discrete analogy, we have the following result by using a different idea.

\begin{prop}[Necessity of the supermodularity assumption in the discrete Riesz inequality]\label{prop5.8}
    Let $H:\R_+ \to \R$ be nonincreasing and not identically equal to zero, and satisfies
    \begin{equation}\label{equ:assumption_of_H}
        \lim_{r\to \infty} r^{d-1}H(r)=0.
    \end{equation}
    Suppose that $G:\R_+ \times R_+ \to \R$ is a Borel measurable function satisfying
    \begin{equation}\label{equ:assumption_of_G}
        G(s,0)=G(0,t)=0, \quad \forall\ s,t\ge 0.
    \end{equation}
    If for all nonnegative functions $u,v\in C_c(\Z^d)$ the following holds:
    \begin{equation}\label{equ:assumption_of_Riesz_inequality_hold}
        \sum_{x,y\in \Z^d} G(u(x),v(y))H(d(x,y)) \le \sum_{x,y\in \Z^d} G(u^*(x),v^*(y))H(d(x,y)),
    \end{equation}
    then $G$ is supermodular.
\end{prop}

\begin{proof}
    By the assumption of $H$, there exists a positive integer $l$ such that $H(0)>H(l)$. Fix $z=(l,0,0,\cdots)\in \Z^d$, let $s,s_0,t,t_0\in \R_+$, $u_L=s \mathbbm{1}_{V^{\Diamond}_L}+s_0 \mathbbm{1}_{\{0\}}$, $v_L=t \mathbbm{1}_{V^{\Diamond}_L}+t_0 \mathbbm{1}_{\{z\}}$, where $L>l$. Then $u=u^*$ and $v^*=t \mathbbm{1}_{V^{\Diamond}_L}+t_0 \mathbbm{1}_{\{0\}}$. By \eqref{equ:assumption_of_G} and \eqref{equ:assumption_of_Riesz_inequality_hold}, we have
    \begin{equation}
        \begin{aligned}
            0 \le & \sum_{x,y\in \Z^d} G(u_L^*(x),v_L^*(y))H(d(x,y))-\sum_{x,y\in \Z^d} G(u(x),v(y))H(d(x,y))\\
            = & G(s,t)\!\!\sum_{(V^{\Diamond}_L\setminus\{0\})\times(V^{\Diamond}_L\setminus\{0\})}\!\!H(d(x,y)) + G(s+s_0,t)\!\!\sum_{\{0\}\times(V^{\Diamond}_L\setminus\{0\})}\!\!H(d(x,y))\\
            &+ G(s,t+t_0)\!\!\sum_{(V^{\Diamond}_L\setminus\{0\})\times \{0\}}\!\!H(d(x,y)) + G(s+s_0,t+t_0)\!\!\sum_{\{0\}\times \{0\}}\!\!H(d(x,y))\\
            &-G(s,t)\!\!\sum_{(V^{\Diamond}_L\setminus\{0\})\times(V^{\Diamond}_L\setminus\{z\})}\!\!H(d(x,y)) - G(s+s_0,t)\!\!\sum_{\{0\}\times(V^{\Diamond}_L\setminus\{z\})}\!\!H(d(x,y))\\
            &-G(s,t+t_0)\!\!\sum_{(V^{\Diamond}_L\setminus\{0\})\times \{z\}}\!\!H(d(x,y)) - G(s+s_0,t+t_0)\!\!\sum_{\{0\}\times \{z\}}\!\!H(d(x,y))\\
            =& [H(0)-H(l)][G(s+s_0,t+t_0)+G(s,t)-G(s+s_0,t)-G(s,t+t_0)]\\
            &+[G(s,t+t_0)-G(s,t)]\left[\sum_{V^{\Diamond}_L\times \{0\}}H(d(x,y))-\sum_{V^{\Diamond}_L\times \{z\}}H(d(x,y))\right]
        \end{aligned}
    \end{equation}
    holds for all $L>l$. Note that by \eqref{equ:assumption_of_H} we have
    \begin{align*}
         \sum_{V^{\Diamond}_L\times \{0\}}H(d(x,y))-\sum_{V^{\Diamond}_L\times \{z\}}H(d(x,y))
        \le  \sum_{V^{\Diamond}_{L+l}\setminus V^{\Diamond}_L} H(d(0,x)) \to 0
    \end{align*}
    as $L \to \infty$, then we obtain
    $$G(s+s_0,t+t_0)+G(s,t)-G(s+s_0,t)-G(s,t+t_0) \ge 0,$$
    i.e. $G$ is supermodular.
\end{proof}

\begin{remark}
    If there exists an $r_0 \ge 0$ such that $H(r)=0$ for all $r\ge r_0$, then \eqref{equ:assumption_of_G} can be replaced by $G(0,0)=0.$
\end{remark}

Recall that a Borel measurable function $G:\R_+^m \to \R$ is supermodular if
$$G(y+s e_i+t e_j)+G(y) \ge G(y+s e_i)+G(y+t e_j),\ \forall s,t\ge 0,\ y\in \R_+^m, \ i\ne j.$$
By a similar argument as Theorem~\ref{thm:discrete_Riesz_inequality}, we obtain the following result. One may refer to \cite{burchard2006rearrangement} for the continuous case.
\begin{theorem}[Discrete extended Riesz inequality]
    Let $u_1,u_2,\cdots,u_m$ be admissible functions on $\Z^d$, $G:\R_+^m\to \R_+$ be supermodular with $G(\mathbf{0})=0$, each $H_{ij}$ is nonincreasing and nonnegative, then
    \begin{equation}
        \begin{aligned}
            & \sum_{x_1,\cdots,x_m\in \Z^d} G(u_1(x_1),\cdots,u_m(x_m))\prod_{1\le i<j \le m}H_{ij}(d(x_i,x_j))\\
            \le & \sum_{x_1,\cdots,x_m\in \Z^d} G(u_1^*(x_1),\cdots,u_m^*(x_m))\prod_{1\le i<j \le m}H_{ij}(d(x_i,x_j)).
        \end{aligned}
    \end{equation}
\end{theorem}

By setting $H(t)=\delta_0(t)$ in Theorem~\ref{thm:discrete_Riesz_inequality}, we get the discrete generalized Hardy-Littlewood inequality.
\begin{corollary}
    Let $u,v \in C_0^+(\Z^d)$ be admissible, and $G:\R_+ \times \R_+ \to \R_+$ be a a supermodular function with $G(0,0)=0$. Then
    \begin{equation}
        \sum_{x \in \Z^d}G(u(x),v(x)) \le \sum_{x \in \Z^d}G(u^*(x),v^*(x)).
    \end{equation}
    Moreover, if $u=u^*$ with $\ran(u)=[a_1>a_2>\cdots]$, $G$ is strictly supermodular, $|\sum_{x \in \Z^d}G(u^*(x),v^*(x))|<\infty$, then the equality holds if and only if $v=v^*$.
\end{corollary}
\begin{corollary}[Discrete Hardy-Littlewood inequality]
    Let $u,v \in C_0^+(\Z^d)$ be admissible, then
    \begin{equation}\label{equ:discrete_Hardy-Littlewood_inequality}
        \sum_{x \in \Z^d}u(x)v(x) \le \sum_{x \in \Z^d}u^*(x)v^*(x).
    \end{equation}
    Moreover, if $u\in \mathcal{S}(\Z^d)$ with $\ran u=[a_1>a_2>\cdots]$, $\sum_{x \in \Z^d}u^*(x)v^*(x)<+\infty$, then the equality holds if and only if $v=v^*$.
\end{corollary}

\begin{corollary}
    The rearrangement operator is a contraction in $l^p(\Z^d)$ in the following sense:
    \begin{equation}
        \sum_{x \in \Z^d}|u^*(x)-v^*(x)|^p \le \sum_{x \in \Z^d}|u(x)-v(x)|^p
    \end{equation}
    holds for all nonnegative functions $u,v \in l^p(\Z^d)$.
    Moreover, if $u=u^*$ with $\ran u=[a_1>a_2>\cdots]$, $p>1$, then the equality holds if and only if $v=v^*$.
\end{corollary}
\begin{proof}

    Apply Theorem~\ref{thm:discrete_Riesz_inequality_b} by setting $G(u,v)= -|u-v|^p$ and $H(t)=\delta_0(t)$, then the result follows.
\end{proof}
In the following example, we will show that the conditions of the equality cases in Theorem~\ref{thm:discrete_Riesz_inequality} are necessary.

\begin{example}\label{example5.14}
    We consider admissible functions $u,v$ on $\Z$.\\
   \begin{enumerate}[(i)]
       \item The condition ``$\ran(u)=[a_1>a_2>\cdots]$'' is necessary. Consider $u=\mathbbm{1}_{\{0,\pm 1,\pm 2\}}$, then
       $$\sum_{x \in \Z}u(x)v(x) = \sum_{x \in \Z}u^*(x)v^*(x)$$
       holds for all nonnegative $v$ with $\supp{v} \subset \{0,\pm 1,\pm 2\}$.
       \item The condition ``$G$ is strictly supermodular'' is necessary. Let $u\in \mathcal{S}(\Z)$ with $\ran(u)=[a_1=5>a_2=4>a_3=1>\cdots]$, let
       \begin{align*}
           v(x)=
           \begin{cases}
               2, \quad & x=0;\\
               3, & x=1;\\
               u(x), &\text{others}.
           \end{cases}
       \end{align*}
       Then $\sum_{x \in \Z}|u^*(x)-v^*(x)| = \sum_{x \in \Z}|u(x)-v(x)|$, and $v^* \notequiv v$.
       \item The condition ``$H(n)>H(n+1)$ for some $n\in \N$'' is necessary. For the trivial $H\equiv 0$, we immediately get the result. If $H\equiv 1$, $u \in \mathcal{S}^1(\Z)$, then
       $$\sum_{x \in \Z, \ y\in \Z}u(x)v(y) = \sum_{x \in \Z,\ y\in \Z}u^*(x)v^*(y)=||u||_1 ||v||_1$$
       for all $v\in l^1(\Z)$.
       \item The condition ``$|\sum\limits_{x,y \in \Z^d}G(u^*(x),v^*(y))H(d(x,y))|<\infty$'' is necessary. Let
       \begin{align*}
           u(x)=
           \begin{cases}
               \frac{1}{2|x|+1}, \quad & x\le 0;\\
               \frac{1}{2|x|}, & x>0.
           \end{cases}
       \end{align*}
       Let $H(n)=\mathbbm{1}_{\{0\}}(n)+1$. Then
       $$\sum_{x \in \Z, \ y\in \Z}u(x)v(y)H(d(x,y)) \ge ||v||_{\infty}||u||_1=+\infty$$
       for all admissible $v\notequiv\:0$.
   \end{enumerate}
\end{example}

\begin{theorem}[Discrete P\'olya-Szeg\"o inequality]
    \label{theorem:discrete_Polya-Szego_inequality}
    Let $u\in l^p(\Z^d)$ be nonnegative, $p\ge 1$, $u^*=R_{\Z^d}u$, then
    \begin{equation}
        \label{equ:discrete_Polya-Szego_inequality}
        ||\nabla u^*||_p \le ||\nabla u||_p.
    \end{equation}
    Moreover, if $p>1$, $\ran(u)=[a_1 > a_2 > \cdots]$, then the equality holds if and only if
    $$u\cong u^*.$$
\end{theorem}
\begin{proof}
    The inequality~\eqref{equ:discrete_Polya-Szego_inequality} follows by applying Theorem~\ref{thm:discrete_Riesz_inequality_b} to $u=v$, $G(u(x),u(y))=-|u(x)-u(y)|^p$, and
    \begin{align*}
        H(d(x,y))=
        \begin{cases}
            1, \quad & d(x,y) =0 \text{ or } 1,\\
            0, & \text{others}.
        \end{cases}
    \end{align*}
    Next we prove the second part of the result. We will address $\Z^2$ for the simplicity of notation, the proof for higher dimensions is similar. W.O.L.G., we always assume that $u(0,0)=a_1$, i.e. the largest function value attains at $(0,0)$. Note that if $||\nabla u^*||_p = ||\nabla u||_p$, then $||\nabla R_e u||_p = ||\nabla u||_p$ for all $e\in \Vec{E}$.\\
    \underline{Step 1}: We prove that $R_{e_1}u \cong u$.\\
    Consider $e=e_1$. Let
    $$\mathcal{E}_p^{e,i}(u)=\sum_{(l,i)\in V_{e}^i} |u(l,i)-u(l+1,i)|^p,$$
    and
    $$\mathcal{E}_p^{e',i}(u)=\sum_{(l,i)\in V_{e}^i} |u(l,i)-u(l,i+1)|^p.$$
    Then
    $$\mathcal{E}_p(u)=\sum_{i\in \Z}(\mathcal{E}_p^{e,i}(u)+\mathcal{E}_p^{e',i}(u)).$$
    By Corollary ~\ref{cor:appendix_A.5} and Lemma~\ref{lem:appendix_A.6}, we have $\mathcal{E}_p^{e,i}(R_{e}u) = \mathcal{E}_p^{e,i}(u)$, $\mathcal{E}_p^{e',i}(R_{e} u) = \mathcal{E}_p^{e',i}(u)$, and $ R_{e_1} u =u$ or $u\circ \sigma_{(0,1)}$, where $\sigma_{(0,1)}$ is the reflection with respect to line $(0,1)$.
    \begin{figure}[H]
        \centering
        \setcounter{subfigure}{0}
        \subfigure[$\mathcal{E}_p^{e_1,2}$ takes over the energy on the bold edges]{
        \includegraphics[width=0.4\textwidth]{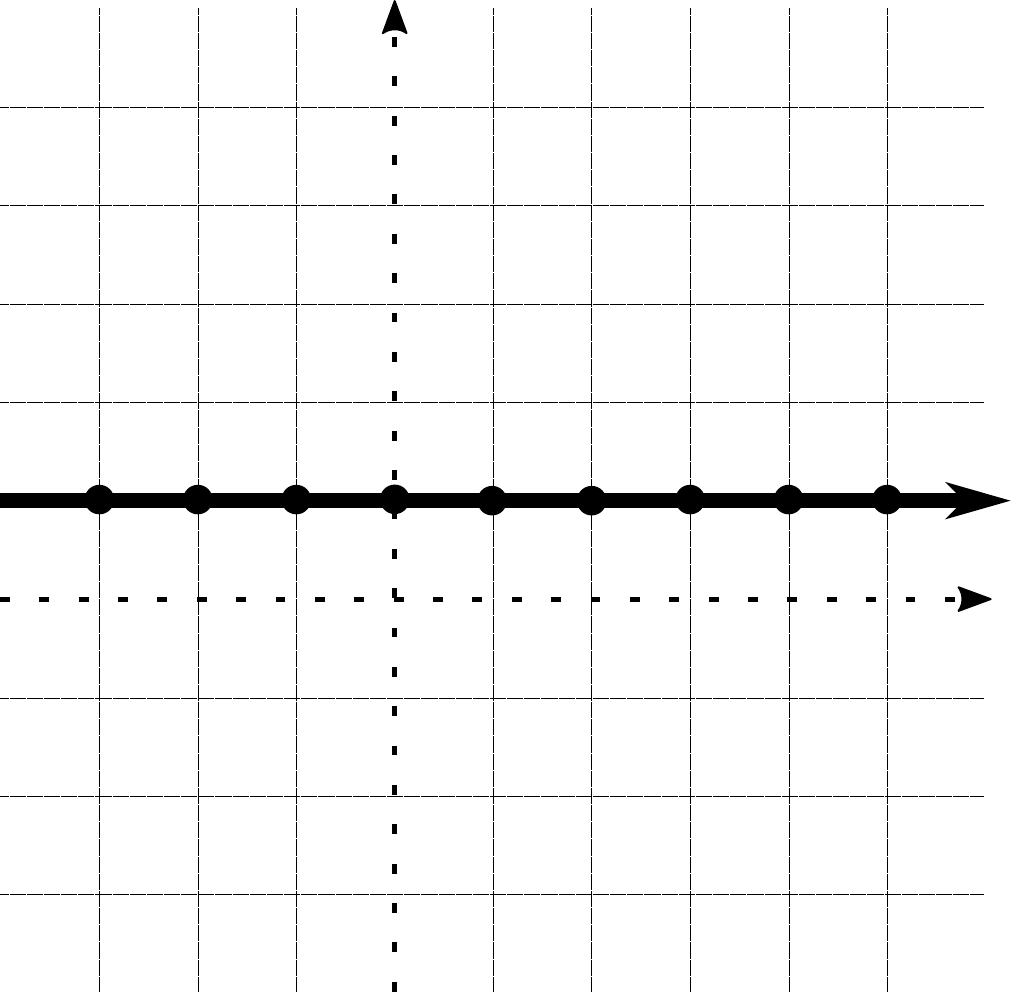}} \quad
        \subfigure[$\mathcal{E}_p^{e_1',2}$ takes over the energy on the bold edges]{
        \includegraphics[width=0.4\textwidth]{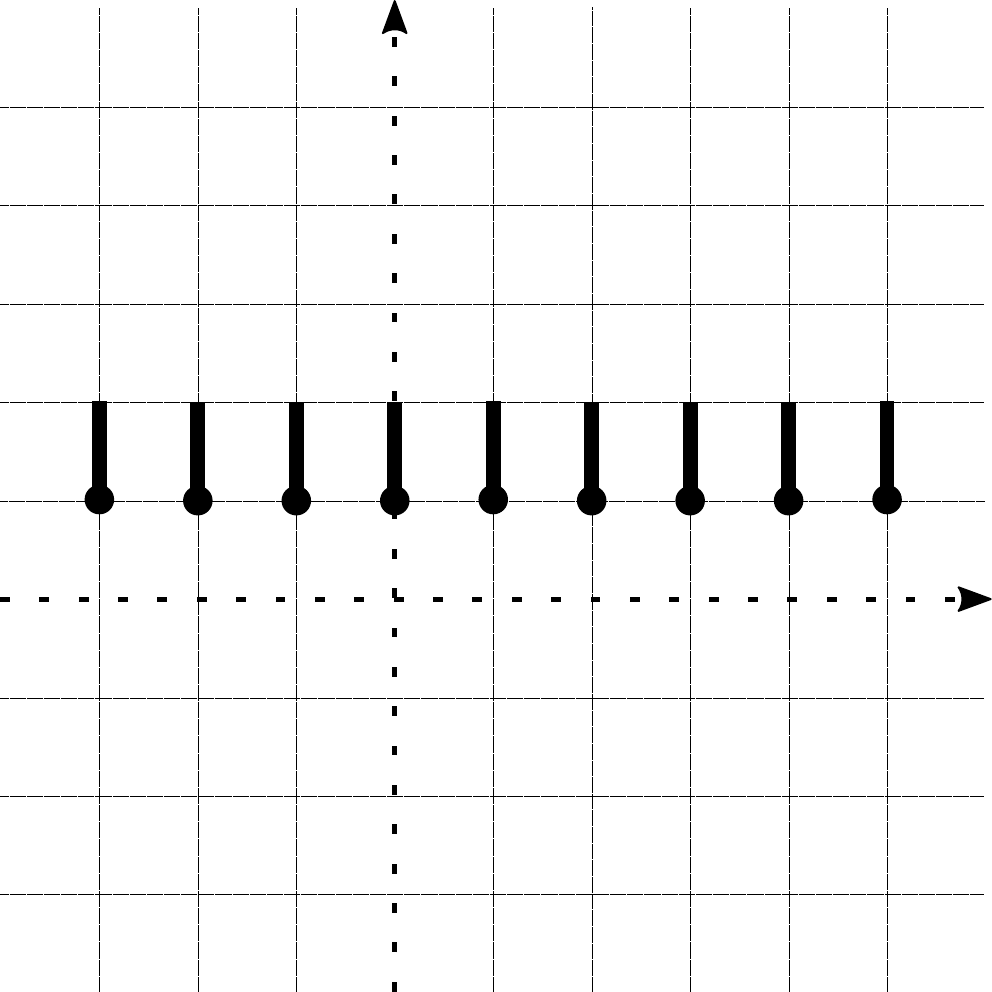}} \quad
    \end{figure}
    \underline{Step 2}: We prove that $R_{\frac{e_1+e_2}{2}}u \cong u$.\\
    Consider $e=\frac{e_1+e_2}{2}$. Let
    $$\mathcal{E}_p^{e,i}(u)=\sum_{(l,m)\in V_{e}^i} (|u(l,m)-u(l-1,m)|^p+|u(l,m)-u(l,m+1)|^p).$$
    Then
    $$\mathcal{E}_p(u)=\sum_{i\in \Z}\mathcal{E}_p^{e,i}(u).$$
    By Lemma~\ref{lem:appendix_A.6}, we have $\mathcal{E}_p^{e,i}(R_e u) = \mathcal{E}_p^{e,i}(u)$, and $R_{\frac{e_1+e_2}{2}} u =u$ or $u\circ \sigma_{(1,-1)}$, where $\sigma_{(1,-1)}$ is the reflection with respect to the line $(1,-1)$.
    \begin{figure}[H]
        \centering
        \includegraphics[width=0.4\textwidth]{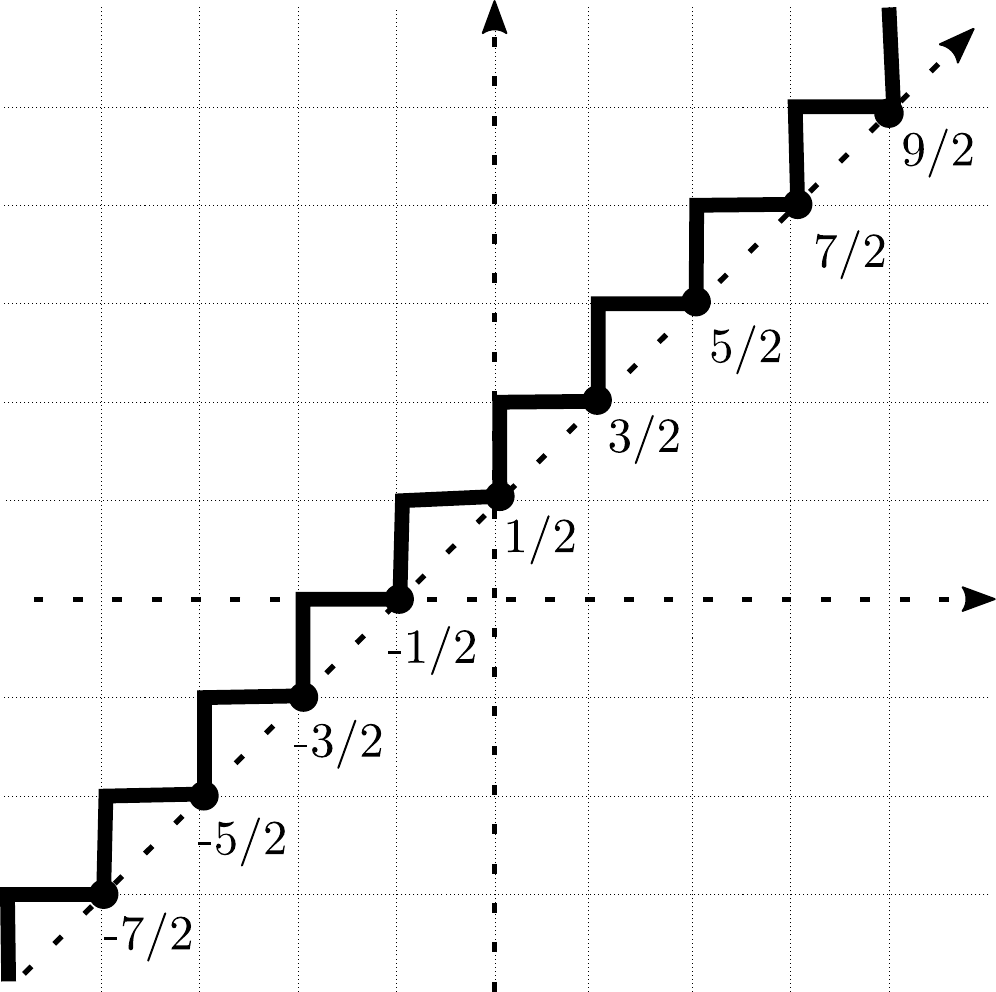}
        \caption{$\mathcal{E}_p^{\frac{e_1+e_2}{2},-1}$ takes over the energy on the bold edges}
    \end{figure}
    With a similar argument, we have $ R_{e_1} u =u$ or $u\circ \sigma_{(1,0)}$, $R_{\frac{e_1-e_2}{2}} u =u$ or $u\circ \sigma_{(1,1)}$. This proves the result.
\end{proof}

One may wonder whether the condition ``$a_{n}>a_{n+1}$ for all $n\in \N^+$'' in the cases of equality of Theorem~\ref{theorem:discrete_Polya-Szego_inequality} can be removed. Unfortunately it does not seem to be the case, and we have the following counter-example.
\begin{example}
    Given $u\in \mathcal{S}(\Z)$ with $\ran(u)=[a_1 > a_2 >a_3 >a_4=a_5>a_6>a_7 \succ \cdots]$, we switch the positions of function values $a_2$ and $a_3$ to get a new function $u'$ as shown in the following figures, then $||\nabla u||_p=||\nabla u'||_p$ and $u \not\cong u'$.
    \begin{figure}[H]
        \centering
        \subfigure[$u$]{\includegraphics[width=0.8\textwidth]{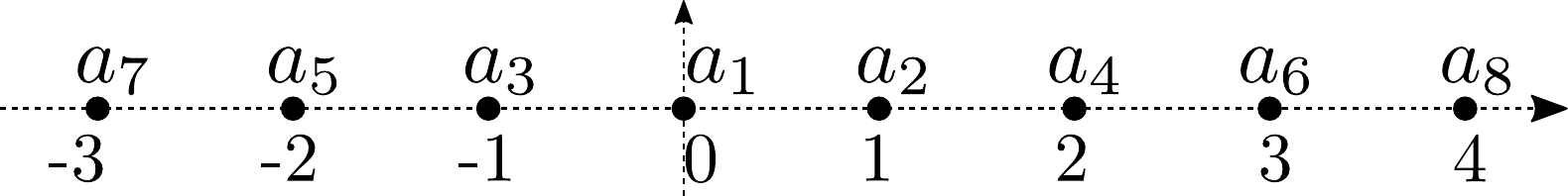}} \quad
        \subfigure[$u'$]{\includegraphics[width=0.8\textwidth]{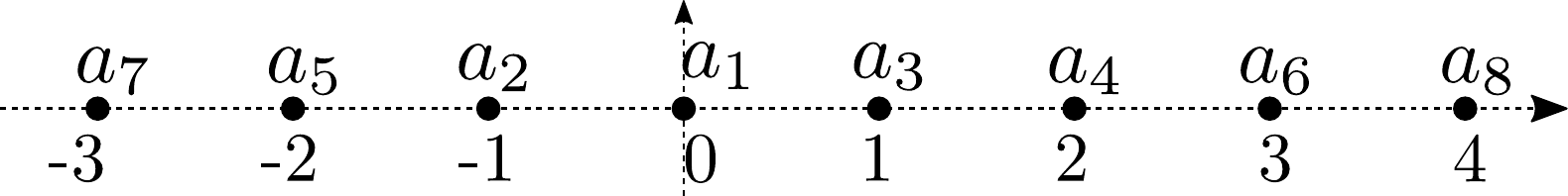}}
    \end{figure}
\end{example}

The following example shows that we can not even get $\supp u \cong \supp u^*$ in the cases of equality of Theorem~\ref{theorem:discrete_Polya-Szego_inequality} if the condition ``$a_{n}>a_{n+1}$ for all $n\in \N^+$'' is removed.
\begin{example}
   Consider $\Omega \subset \Z^2$ as shown in (a) of the following figure, let $u=\mathbbm{1}_{\Omega}$. Let $\Omega' \subset \Z^2$ as shown in (b), then $R_{\Z^2}u=\mathbbm{1}_{\Omega'}$. We have $||\nabla u||_p=||\nabla R_{\Z^2}u||_p$. However, $\Omega \not \cong \Omega'$.
    \begin{figure}[H]
        \setcounter{subfigure}{0}
        \centering
        \subfigure[$\Omega$ and $u=\mathbbm{1}_{\Omega}$]{\includegraphics[width=0.3\textwidth]{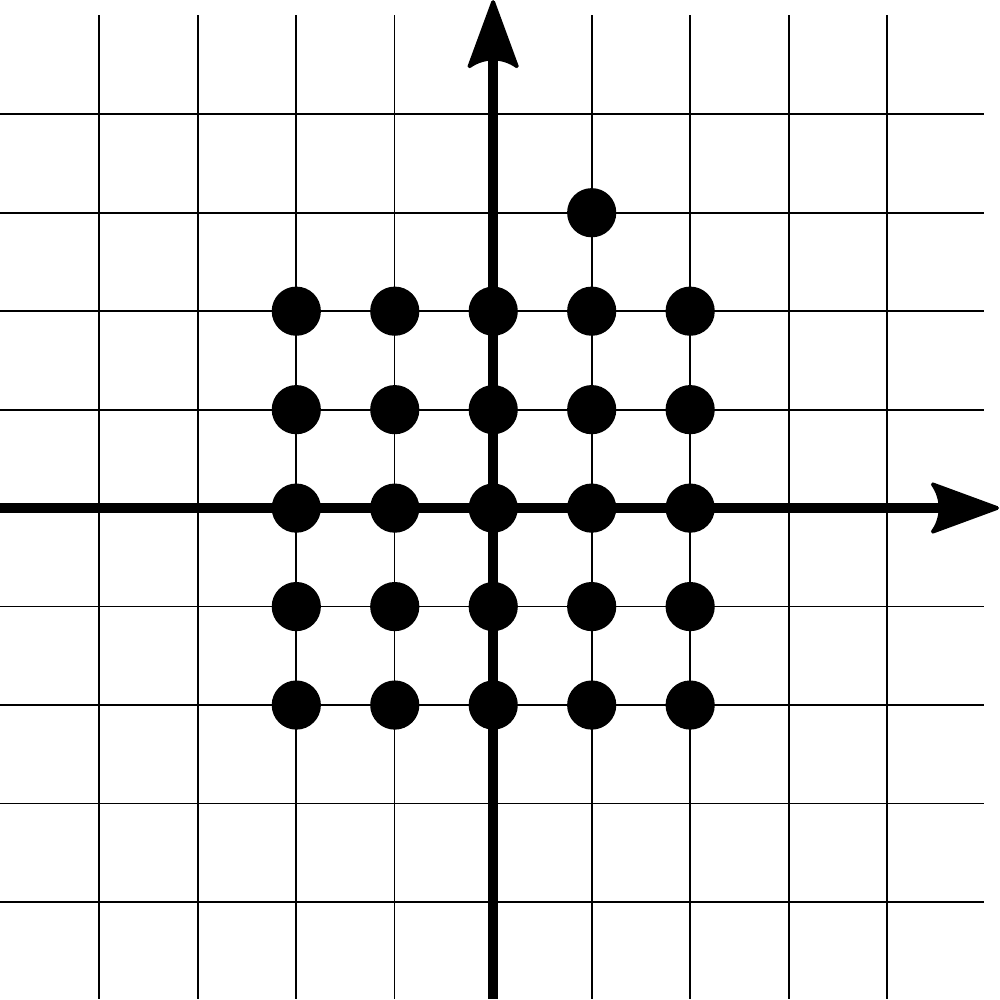}} \quad
        \subfigure[$\Omega'$ and $R_{\Z^2}u=\mathbbm{1}_{\Omega'}$]{\includegraphics[width=0.3\textwidth]{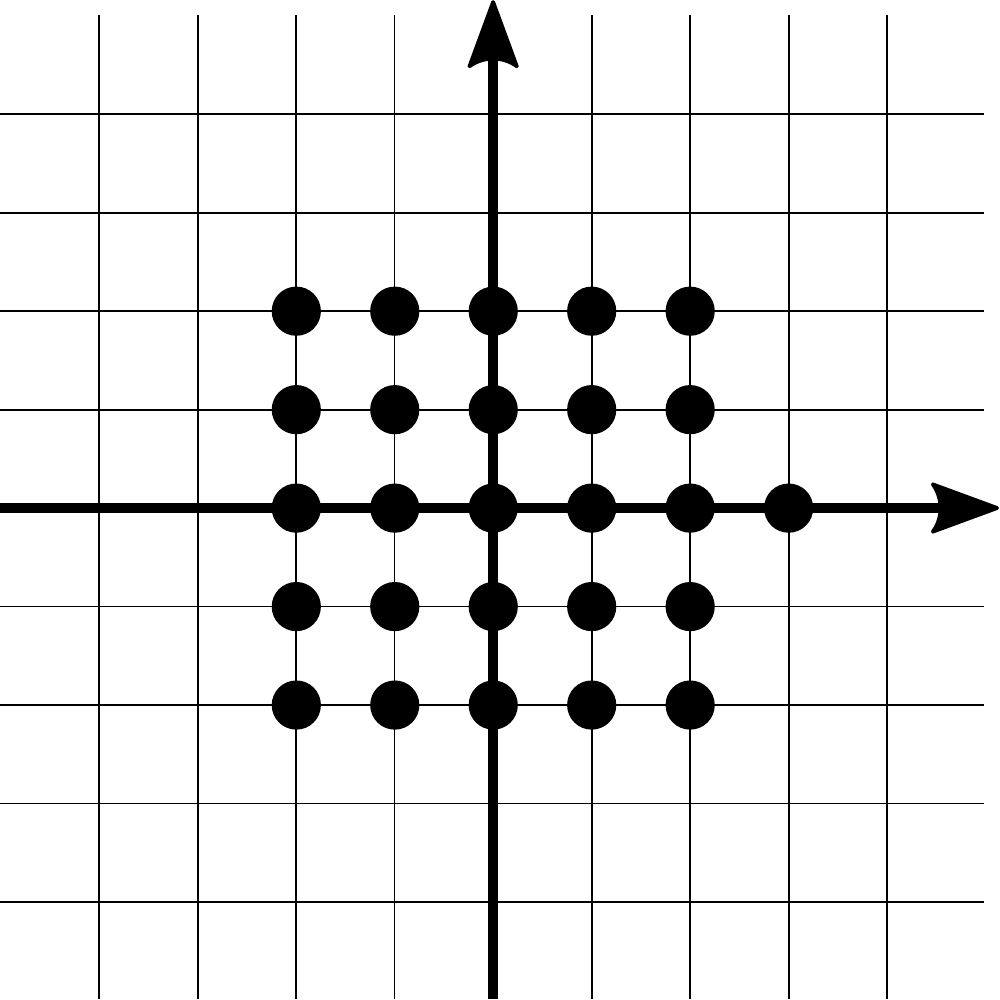}}
    \end{figure}
\end{example}

\begin{lemma}
    \label{lemma:existence_of_best_symmetric_function}
    Let $u\in l^p(\Z^d)$ be nonnegative, $p\ge 1$, then there exists $v \overset{e.m.}{\sim} u$ such that $v\in \mathcal{S}(\Z^d)$ and
    \begin{equation}
        ||\nabla v||_p = \inf\{||\nabla u'||_p:\ u'\overset{e.m.}{\sim} u\}.
    \end{equation}
\end{lemma}
\begin{proof}
    Let $\{v_i\}$ be a minimizing sequence, i.e.:
    \begin{enumerate}[(i)]
        \item $v_i \overset{e.m.}{\sim} u$ for all $i$;
        \item $||\nabla v_i||_p \ge ||\nabla v_{i+1}||_p$ and $\lim_{i \to \infty}||\nabla v_i||_p = \inf\{||\nabla u'||_p:\ u'\overset{e.m.}{\sim} u\}$.
    \end{enumerate}
    By Theorem~\ref{theorem:discrete_Polya-Szego_inequality}, we may assume that $v_i \in \mathcal{S}(\Z^d)$. By a similar argument with Theorem~\ref{thm:compact_embedding_for_rearranged_functions}, we have (by passing to a subsequence)
    \begin{equation}
        v_i \overset{l^p}{\longrightarrow} v.
    \end{equation}
    Then $||\nabla v||_p=\inf\{||\nabla u'||_p:\ u'\overset{e.m.}{\sim} u\}$ follows by the equivalence of $l^p$ norm and $W^{1,p}$ norm on $\Z^d$. This proves the result.
\end{proof}

By a similar argument as Theorem~\ref{thm:discrete_Riesz_inequality}, we have the following result.
\begin{theorem}
    Let $u\in C_0^+(\Z^d)$ be admissible, $F\colon \R_+\times\R_+\to \R_+$ be such that:
    \begin{enumerate}[(i)]
        \item $-F$ is supermodular with $F(0,0)=0$;
        \item $\sum_{x\in \Z^d}F(|x|,0)<\infty$.
    \end{enumerate}
    Then
    \begin{equation}\label{ineq4.9}
        \sum_{x}F(|x|,u(x))\leq \sum_{x}F(|x|,u^*(x)).
     \end{equation}
\end{theorem}

\section{Applications: Constrained minimization problems on lattice graphs}

\subsection{Construction of waves for discrete non-linear Schr\"odinger equations}

In this section, we will study the following minimization problem: For $u\in H^1(\mathbb{Z}^d)=W^{1,2}(\mathbb{Z}^d):$
\begin{align}\label{ineq5.1}
  &E(u)=\frac{1}{2}||\nabla u||_2^2-\sum_{x\in \Z^d}F(|x|,u(x)).
\end{align}
\begin{equation}\label{ineq5.2}
  I_c=\inf\left\{E(u):\ u\in S_c\right\},\quad S_c=\left\{u\in H^1(\mathbb{Z}^d):\ \sum_{x}u^2(x) = c^2 \right\}.
\end{equation}

\begin{theorem}\label{thm5.1}
  Let $F\colon\mathbb{N}\times\R_+\to \R_+$ be a function satisfying the following assumptions:
  \begin{itemize}
    \item [($F_0$)] $F(|x|,t)\leq F(|x|,|t|)$, $\forall\ x\in \Z^d,$ $t\in\R$;
    \item [($F_1$)] $F(|x|,\cdot)$ is continuous for all $x\in \Z^d$;
    \item [($F_2$)] There exist $K>0,$ $l>0$, such that
    $$0\leq F(|x|,t)\leq K \left(t^2+t^{l+2}\right),\ \forall\ x\in \Z^d,\ \text{and}\ t\geq0;$$
    \item [($F_3$)] $\forall\ \epsilon>0,$ $\exists\ X_0 >0,$ $t_0\in\R,$ such that
    $$F(|x|,t)\leq\epsilon t^2,\quad \forall\ |x| > X_0\ \text{and}\ |t|<t_0;$$
    \item [($F_4$)] $-F$ is supermodular with $F(0,0)=0$;
    \item [($F_5$)] $F(|x|,\theta t)\geq \theta^2 F(|x|, t),$ $\forall\ x\in \Z^d,$ $\theta \ge 1,$ $t\in\R$;
    \item [($F_6$)] There exist $\delta>0,$ $t_1>0,$ $X_1>0,$ $0<\sigma<\frac{2}{d}$, such that $$F(|x|,t)>\delta t^{2+2\sigma},\  \forall |x|> X_1\ \mathrm{and}\ |t|<t_1.$$
  \end{itemize}
  Then \eqref{ineq5.2} admits a Schwarz symmetric solution, i.e. $u=R_{\mathbb{Z}^d}u$ such that $E(u)=I_c$, $u\in S_c$.
\end{theorem}

A typical example is $F(|x|,t)=\frac{1}{2+2\sigma}|t|^{2+2\sigma}$ with $\sigma < \frac{2}{d}$, and the minimization problem is
\begin{equation}
    I_c=\inf\left\{||\nabla u||_2^2 - \frac{1}{\sigma+1}||u||_{2\sigma+2}^{2\sigma+2}:\ ||u||_2^2 = c^2\right\},
\end{equation}
whose minimizers solve the following equation on $\Z^d$:
\begin{equation}
    \label{equ:DNLS}
    -\Delta u + \omega u - |u|^{2\sigma}u=0,
\end{equation}
where $\omega>0$ is the Lagrange multiplier that depends on $c$. Weinstein \cite{weinstein1999excitation} proved the existence of solution of \eqref{equ:DNLS} by the concentration-compactness arguments. We give a different proof by the rearrangement inequalities established in the previous sections.
\begin{remark}
    Under $(F_6)$, $I_c<0$ for all $c>0$; see \cite[Lemma 3]{stefanov2021ground}.\\
    For any $K\in \N$, let 
    \begin{align*}
        u_K(x)=
        \begin{cases}
            c_K \cdot \frac{K-||x||_{1}}{K^{\frac{d}{2}+1}}, \quad & ||x||_1 \le K,\\
            0, & \text{otherwise},
        \end{cases}
    \end{align*}
    where $c_K$ is a constant such that $||u_K||_2=c$. In \cite{stefanov2021ground}, the followings are proved:
    \begin{enumerate}[(i)]
        \item $||\nabla u_K||_2^2 \le c'K^{-2}$ for some constant $c'$;
        \item $||u_K||_{2\sigma +2}^{2\sigma +2} \ge c''K^{-d\sigma}$ for some constant $c''$.
    \end{enumerate}
    The result follows by choosing $K$ large enough.
\end{remark}

\begin{proof}[Proof of Theorem \ref{thm5.1}]
$~$

\underline{Step 1}: We will first show that $ I_c>-\infty$ and that all minimizing sequences are bounded in $H^1$( The variational problem is well-posed).\\
Since $\sum_{x}u^2(x) = c^2$, we have $||u||_{\infty}\le c$. By $(F_2)$, we can write
\begin{align*}
    \sum_{x\in\Z^d}F\left(|x|,u(x)\right)
    &\le K \sum_{x}u^2+ K \sum_{x}|u|^l|u|^2\\
    &\le K \sum_{x}u^2+ K ||u||_{\infty}^l\sum_{x}|u|^2\\
    &\le K(1+c^l)\sum_{x}u^2=K(1+c^l)c^2.
\end{align*}
Therefore,
\begin{align*}
E(u)&\geq \frac{1}{2}||\nabla u||_{2}^{2}- K(1+c^l)c^2.
\end{align*}
Thus, $ I_c>-\infty$. All minimizing sequences are bounded in $H^1$ by the equivalence of $l^2$ norm and $H^1$ norm on $\Z^d$.\\
\underline{Step 2}: Existence of Schwarz symmetric sequence.\\
By the rearrangement inequalities proved in section 5, we certainly have $E\left( R_{\Z^d}|u|\right)\leq E\left( |u|\right)\leq E\left( u\right)$. Therefore without loss of generality, we can assume that $I_c$ always admits a Schwarz symmetric minimizing sequence, i.e, $u_n^c$ such that $u_n^c=R_{\Z^d}u_n^c$ and $u_n^c\in l^2(\Z^d)$.\\
\underline{Step 3}: Compactness  of Schwarz symmetric minimizing sequences.\\
Let $\{u_n^c\}$ be a
Schwarz symmetric minimizing sequence, then $\forall t >0$, we can find $X_2>0$ such that
\begin{equation}\label{ineq5.4}
  \qquad |u_n^c(x)|\leq t,\quad \forall\ |x|\geq  X_2.
\end{equation}
By the lower semi-continuity of $||\cdot||_2$, we certainly have 
$$||\nabla u^c||_2\leq \liminf_n||\nabla u_n^c||_2,$$ 
where $u$ is the weak limit of $u_n$ in $H^1$( up to subsequence).\\
For all fixed $X_3>0$, $u_n^c $ strongly converge to $u$ in $l^{l+2}\left( |x| \le X_3\right)$. This implies that,
$$\lim_{n\to\infty}\sum_{|x|\le X_3}F\left(|x|,u^c_n\right)=\sum_{|x|\le X_3}F\left(|x|,u^c\right).
$$
\eqref{ineq5.4} together with $(F_3)$ imply that:
$$\sum_{|x|>X_{\epsilon}}F\left(|x|,u^c_n\right)\text{ and } \sum_{|x|>X_{\epsilon}}F\left(|x|,u^c\right)<\varepsilon,\ \forall \varepsilon>0,\ \text{for some } X_{\epsilon}\text{ big enough.}$$
In conclusion, $\displaystyle{\lim_{n\to\infty}\sum_{x}F\left(|x|,u^c_n\right)=\sum_{x}F\left(|x|,u^c\right)}
$.\\
\underline{Step 4}: $I_c$ is achieved.\\
By the weak lower semi-continuity of the norm $l^2$, we have $\displaystyle{\sum_{x\in \Z^d}u^2\left(x\right)\leq c^2}.$\\
$u\neq 0$, since $I_c<0$ and $ F(.,0)=0$. Set $t=\frac{c}{||u||_2}$, then $t\geq 1$. On the other hand,
$$I_c\leq E(tu)\leq t^2 E(u)\leq t^2I_c\Rightarrow t\leq 1,$$
by the strict negativity of $I_c$.
\end{proof}

 Let us go back to equation~\eqref{equ:DNLS}. A different way to consider this type of equation is tackling the following constrained minimization problem:
 \begin{equation}\label{equ:nonormalized_wave}
     I_{\omega}=\inf\{||\nabla u||_2^2+\omega||u||_2^2 : ||u||_{2\sigma+2}^{2\sigma+2}=1\},
 \end{equation}
whose minimizers solve the following equation:
\begin{equation}
    -\Delta u + \omega u - c_{\omega} |u|^{2\sigma}u=0,
\end{equation}
where $\omega >0$ is a fixed constant, $c_{\omega}>0$ depends on $\omega$.

The existence of the minimizer is well known; see \cite{stefanov2021ground,hua2022class}, we give an alternative proof here.

\begin{theorem}
    The minimization problem \eqref{equ:nonormalized_wave} admits a Schwarz symmetric solution.
\end{theorem}
\begin{proof}
    Let $\{u_k\}$ be a minimizing sequence. By passing to a subsequence, we may assume that:
    \begin{enumerate}[(i)]
        \item $u_k\in \mathcal{S}(\Z^d)$;
        \item $\{u_k\}$ converge weakly in $l^2(\Z^d)$.
    \end{enumerate}
    By (i) and Theorem~\ref{thm:geometry_of_Schwarz_symmetric_functions},
    \begin{equation}
        u_k \overset{l^{2\sigma+2}}{\longrightarrow} u
    \end{equation}
    for some $u\in l^{2\sigma+2}(\Z^d)$, which implies $||u||_{2\sigma+2}=1$.
    By (ii), $||\nabla u||_2^2+\omega||u||_2^2=I_{\omega}$.
    This proves the result.
\end{proof}

\subsection{Existence of extremal functions for Sobolev inequality on $\Z^d$}
Let $D^{1,p}(\Z^d)$ be the completion of $C_c(\Z^d)$ in the norm $||u||_{D^{1,p}}=||\nabla u||_p$.
{\bf Discrete Sobolev Inequality} (see \cite{hua2015time}):
there exists a constant $C>0$ such that
\begin{equation}\label{equ:discrete_Soblev_inequality}
    ||u||_{p^*} \le C ||u||_{D^{1,p}}, \ \forall u \in D^{1,p}(\Z^d),
\end{equation}
where $d \ge 3$, $1 \le p < d$, $p^*=\frac{d p}{d-p}$.

Consider the following minimization problem:
\begin{equation}
    \label{equ:minimiaztion_problem_for_Soblev_inequality}
    I=\inf \{||\nabla u||_p:\ u\in S\}, \ S=\{u\in D^{1,p}(\Z^d):\ ||u||_q=1\},
\end{equation}
where $d\ge 3$ and $q > p^*=\frac{dp}{d-p}$. The minimizers solve the following equation on $\Z^d$:
\begin{equation}
    \Delta_p u + u^{q-1}=0,
\end{equation}
where
\begin{equation}
    \Delta_p u(x):=\sum_{y\sim x}|u(y)-u(x)|^{p-2}(u(y)-u(x))
\end{equation}
is called the $p$-Laplacian on graphs. Hua-Li~\cite{hua2021existence} proved the existence of minimizer of \eqref{equ:minimiaztion_problem_for_Soblev_inequality} by a discrete concentration-compactness principle, we provide an alternative proof here.
\begin{theorem}
    The minimization problem~\eqref{equ:minimiaztion_problem_for_Soblev_inequality} admits a Schwarz symmetric solution.
\end{theorem}
\begin{proof}
    Let $\{u_k\}$ be a minimizing sequence of \eqref{equ:minimiaztion_problem_for_Soblev_inequality}. By passing to a subsequence, we may assume that:
    \begin{enumerate}[(i)]
        \item $u_k \in \mathcal{S}(\Z^d)$;
        \item $u_k$ converge weakly in $l^p(\Z^d)$;
        \item $||\nabla u_k||_p < E$ for some constant $E$.
    \end{enumerate}
    By the discrete Sobolev inequality \eqref{equ:discrete_Soblev_inequality} and (iii), we have
    $$||u_k||_{p^*} < CE.$$
    By passing to a subsequence and Theorem~\ref{thm:compact_embedding_for_rearranged_functions}, we have
    $$u_k \overset{l^q(\Z^d)}{\longrightarrow} u,$$
    for some $u \in l^q(\Z^d)$. Then $||u||_q=1$ and $||u||_{D^{1,p}}=I$. This proves the result.
\end{proof}

\begin{remark}
  As in the continuous case, rearrangement inequalities play a crucial rule in the establishment of the optimizers of many important functional inequalities. This is the subject of another paper under preparation for the discrete setting.
\end{remark}

\section*{Acknowledgments}
H. Hajaiej would like to thank A. Puss for many helpful discussions in 2009-2010 and Jean Van Schaftingen for insightful and illuminating conversations about the topic at the University of Giessen, and for his scientific generosity. B. Hua is supported by NSFC, grants no.11831004, and by Shanghai Science and Technology Program [Project No. 22JC1400100].

\appendix
\section{Polarizations method on $\Z$}
For the convenience of the reader, we list some basic methods and results on the discrete polarizations, one can find these in \cite{HH10}. Given a semi-finite  open interval $H=(a/2,\infty)$ or $(-\infty, a/2)$, where $a\in \Z$, the reflection with respect to $\partial H$ is denoted by $\sigma_H$.
\begin{definition}[Polarization]
    Let $V=\Z$ or $\Z+1/2$, the polarization of $u:V \to \R$ is defined via:
    \begin{align*}
        u^H(x)=
        \begin{cases}
            \max(u(x),u(\sigma_H x)) \quad & x\in V\cap H;\\
            \min(u(x),u(\sigma_H x)) \quad & x\in V \setminus H.
        \end{cases}
    \end{align*}
\end{definition}
The following lemma plays a key role in the proof of one-dimensional discrete Riesz inequality.
\begin{lemma}
    \label{lemma:A.2}
    Let $W:\R_+ \to \R$ be nonincreasing, and $G:\R_+ \times \R_+ \to \R$ be supermodular.
    \begin{enumerate}[(i)]
        \item If $u,v$ be admissible functions on $\Z$, then
        \begin{equation}
            \sum_{x,y \in \Z} G(u(x),v(y))W(|x-y|) \le \sum_{x,y \in \Z} G(u^{H}(x),v^{H}(y))W(|x-y|).
        \end{equation}
        \item If $u,v$ are admissible functions on $\Z+1/2$, then
        \begin{equation}
            \sum_{x,y \in \Z+1/2} G(u(x),v(y))W(|x-y|) \le \sum_{x,y \in \Z+1/2} G(u^{H}(x),v^{H}(y))W(|x-y|).
        \end{equation}
        \item If $u:\Z \to \R_+$ and $v: \Z+1/2 \to \R_+$ are admissible functions, then
        \begin{equation}
            \sum_{x\in \Z,\ y\in {\Z+1/2}}\!\!\! G(u(x),v(y))W(|x-y|) \le \!\!\!\sum_{x\in \Z,\ y \in \Z+1/2}\!\!\!G(u^{H}(x),v^{H}(y))W(|x-y|).
        \end{equation}
    \end{enumerate}
\end{lemma}
\begin{proof}
    We only prove (i). Only need to prove
    \begin{align*}
        &G(u(x),v(y))W(|x-y|)+G(u(x'),v(y'))W(|x'-y'|)\\
            &+G(u(x'),v(y))W(|x'-y|)+G(u(x),v(y'))W(|x-y'|)\\
        \le &G(u^H(x),v^H(y))W(|x-y|)+G(u^H(x'),v^H(y'))W(|x'-y'|)\\
        &+G(u^H(x'),v^H(y))W(|x'-y|)+G(u^H(x),v^H(y'))W(|x-y'|).
    \end{align*}
    holds for all $x,y \in \Z$, where $x'=\sigma_H(x)$ and $y'=\sigma_H(y)$.

    Assume that $x<y$ W.L.O.G..

    \underline{Case I}: $\partial H \notin (x,y)$. Note that $|x-y|=|x'-y'|\le|x'-y|=|x-y'|$. Then
    \begin{align*}
            &G(u(x),v(y))W(|x-y|)\!+\!G(u(x'),v(y'))W(|x'-y'|)\\
            &+G(u(x'),v(y))W(|x'-y|)\!+\!G(u(x),v(y'))W(|x-y'|)\\
            =& W(|x'-y|)[G(u(x)\!,\!v(y))\!+\!G(u(x')\!,\!v(y'))\!+\!G(u(x')\!,\!v(y))\!+\!G(u(x)\!,\!v(y'))]\\
            &+[W(|x-y|)\!-\!W(|x'-y|)][G(u(x),v(y))\!+\!G(u(x'),v(y'))],
    \end{align*}
    and
    \begin{align*}
        &G(u(x),v(y))\!+\!G(u(x'),v(y'))\!+\!G(u(x'),v(y))\!+\!G(u(x),v(y'))\\
        =&G(u^H\!(x),v^H\!(y))\!+\!G(u^H\!(x'),v^H\!(y'))\!+\!G(u^H\!(x'),v^H\!(y))\!+\!G(u^H\!(x),v^H\!(y')).
    \end{align*}
    Since $G$ is supermodular and $W$ is decreasing, we get:
    \begin{equation}\label{a.1}
        \begin{aligned}
            &[W(|x-y|)-W(|x'-y|)][G(u(x),v(y))+G(u(x'),v(y'))]\\
            \le&[W(|x-y|)-W(|x'-y|)][G(u^H(x),v^H(y))+G(u^H(x'),v^H(y'))],
        \end{aligned}
    \end{equation}
    and this implies the result.

    \underline{Case II}: $\partial H \in (x,y)$. Note that $|x-y|=|x'-y'|\ge|x'-y|=|x-y'|$. The proof is similar.

    This completes the proof.
\end{proof}

We use two particular polarizations. Let $H_+=(0,+\infty)$ and $H_-=(-\infty,1/2)$, so
\begin{align*}
    u^{H_+}(x)=
    \begin{cases}
        \max(u(x),u(-x)) \quad & x \ge 0;\\
        \min(u(x),u(-x)) & x\le 0;
    \end{cases}
\end{align*}
and
\begin{align*}
    u^{H_-}(x)=
    \begin{cases}
        \max(u(x),u(1-x)) \quad & x \le 1/2;\\
        \min(u(x),u(1-x)) & x\ge 1/2.
    \end{cases}
\end{align*}

\begin{lemma}
    \label{lemma:A.3}
    Let $Tu:=(u^{H_-})^{H_+}$.
    \begin{enumerate}[(i)]
        \item If $u\in C_c(\Z)$ be nonnegative, then there exists an integer $K_1$ such that for all $k\ge K_1$
        \begin{equation}
            R_{\Z}u=T^k u.
        \end{equation}
        \item If $u\in C_c(\Z+1/2)$ be nonnegative, then there exists an integer $K_2$ such that for all $k\ge K_2$
        \begin{equation}
            R_{\Z+1/2}u=T^k u.
        \end{equation}
    \end{enumerate}
\end{lemma}
\begin{proof}
    Only prove (i). We prove by induction on $|\supp{u}|$. Suppose that (i) holds for all $|\supp{u}| \le N-1$. Consider function $u$ with $\ran(u)=[a_1 \succ \cdots \succ a_{N-1} \succ a_{N} \succ 0 \succ \cdots]$, where $a_N>0$. Let $\Tilde{u}=\hat{u}_{N-1}$ be the $(N-1)$-th cut-off of $u$. Then:
    \begin{enumerate}[(i)]
        \item Positions of function value $a_n$ are the same for $T^k u$ and $T^k \Tilde{u}$, for all $n \le N-1$ and $k\in \N$;
        \item There exists $K_0$ such that $T^k \Tilde{u}=R_{\Z} \Tilde{u}$ for all $k\ge K_0$.
    \end{enumerate}
    This implies $T^k u(x) =R_{\Z} \Tilde{u}(x)$ for all $x\in \supp{R_{\Z} \Tilde{u}}$ and $k\ge K_0$. Let $X_k$ be the position of $x_N$ in $T^k u$, $k \ge K_0$. Then:
    \begin{enumerate}[(i)]
        \item $|X_{k+1}| \le |X_{k}|$, which implies there exists $K_1$ such that $|X_{k}|=|X_{K_1}|$ for all $k\ge K_1$;
        \item $X_{K_1} \le 0$ if and only if $|X_{K_1}|$ is captured by some $a_n$, $n\le N-1$.
    \end{enumerate}
    This implies $T^{K_1+1}u=R_{\Z}u$, which proves the result.
\end{proof}

\begin{proof}[Proof of Lemma~\ref{lemma:one-dimensional_discrete_Riesz_inequality_for_finitely_supported_functions}]The result follows by Lemma~\ref{lemma:A.2} and Lemma~\ref{lemma:A.3}.
\end{proof}

\begin{lemma}[One-dimensional discrete Riesz inequality]
    \label{lemma:one_dimensional_generlized_Risez_inequality}
    Let $W:\R_+ \to \R_+$ be nonincreasing, and $G:\R_+ \times \R_+ \to \R_+$ be supermodular with $G(0,0)=0$, $u,v$ be admissible functions on $\Z$, then
    \begin{equation}
        \label{inequlity:main_inequlity_of_lemma:one_dimensional_generlized_Risez_inequality}
        \sum_{x,y \in \Z} G(u(x),v(y))W(|x-y|) \le \sum_{x,y \in \Z} G(u^*(x),v^*(y))W(|x-y|).
    \end{equation}
    Moreover, if $G$ is strict supermodular, $u\in \mathcal{S}(\Z)$ with $\ran(u)=[a_1>a_2>\cdots]$, $\sum_{x,y \in \Z} G(u^*(x),v^*(y))W(|x-y|)<\infty$, $W$ satisfies $W(n)>W(n+1)$ for some $n\in \N$, then the equality holds if and only if
    $$v=v^*.$$
\end{lemma}
\begin{proof}
    We only prove the equality case. Firstly, we prove $v(y_0) \ge v(-y_0)$ for all $y_0\in \Z_+$. Note that the following holds if ``$=$'' holds in \eqref{inequlity:main_inequlity_of_lemma:one_dimensional_generlized_Risez_inequality}.
    $$\sum_{x,y \in \Z} G(u(x),v(y))W(|x-y|) = \sum_{x,y \in \Z} G(u^{H_+}(x),v^{H_+}(y))W(|x-y|).$$
    Then for all positive integers $x\ge y>0$, the following holds by \eqref{a.1}:
    \begin{equation}
        \begin{aligned}
            &[W(x-y)-W(x+y)][G(u(x),v(y))+G(u(-x),v(-y))\\
            &-G\left(u(x),\max\{v(y),v(-y)\}\right)-G\left(u(-x),\min\{v(y),v(-y)\}\right)]=0.
        \end{aligned}
    \end{equation}
    Choose $y=y_0$, $x=y_0+n$, then $v(y_0) \ge v(-y_0)$ by the assumptions.\\
    With a similar argument, one can prove $v(1-y_0) \ge v(y_0)$ for any positive integer $y_0$, then the result follows.
\end{proof}

\begin{corollary}\label{cor:appendix_A.5}
    Let $u,v\in C_0^+(\Z)$ be admissible, then
    \begin{equation}
        \sum_{x\in \Z}|u^*(x)-v^*(x)|^p \le \sum_{x\in \Z}|u(x)-v(x)|^p.
    \end{equation}
    Moreover, If $p>1$, $u\in \mathcal S(\Z)$ with $\ran(u)=[a_1 > a_2 > \cdots]$, then the equality holds if and only if $v=v^*$.
\end{corollary}

Let $\frac{\Z}{2}=\Z \cup (\Z+1/2)=\{0,\pm \frac{1}{2}, \pm 1, \cdots \}$. Let $u:\frac{\Z}{2} \to \R$ be admissible, $u_1=u|_{\Z}$, $u_2=u|_{\Z+1/2}$. By defining graph structure on $\Z+1/2$ via $x\sim y$ if and only if $|x-y|=1$, on $\frac{\Z}{2}$ via $x\sim y$ if and only if $|x-y|=\frac{1}{2}$, $||\nabla u_2||_p$ and $||\nabla u||_p$ are well defined:
\begin{align*}
    ||\nabla u_2||_p^p=\sum_{x\in \Z+1/2} |u_2(x+1)-u_2(x)|^p; \ ||\nabla u||_p^p=\sum_{x\in \frac{\Z}{2}} |u(x+\frac{1}{2})-u(x)|^p.
\end{align*}
Let $u_1^*=R_{\Z} u_1$, $u_2^*=R_{\Z+1/2}u_2$, and let $u^*$ defined via $u^*|_{\Z}=u_1^*$, $u^*|_{\Z+1/2}=u_2^*$.
The discrete P\'olya-Szeg\"o inequality tells us that: $||\nabla u_1^*||_p \le ||\nabla u_1||_p$. By a similar argument, we also have
$||\nabla u_2^*||_p \le ||\nabla u_2||_p$ and $||\nabla u^*||_p \le ||\nabla u||_p$.
We focus on the case that the equality holds.

\begin{lemma}\label{lem:appendix_A.6}
    Let $u:\frac{\Z}{2} \to \R$ be admissible, with $\ran(u)=[a_1 > a_2 > \cdots]$. For any $p > 1$, if $||\nabla v^*||_p = ||\nabla v||_p<+\infty$, then
    \begin{equation}
        v \cong v^*,
    \end{equation}
    where $v=u_1$, $u_2$ or $u$, as defined above.
\end{lemma}
\begin{proof}
    Only prove the case of $v=u_1$ with $\ran(v)=[b_1>b_2>\cdots]$. W.L.O.G., assume that $v(0)=b_1$. Note that $||\nabla v||_p=||\nabla v^{H_-}||_p=||\nabla v^{H_+}||_p$. 
    By the proof of Lemma~\ref{lemma:A.2}, we know the equality in \eqref{a.1} holds for $G=-|v(x)-v(y)|^p$, $H=H_+$, $W=\mathbbm{1}_{[0,1]}$, $x>0$, and $y=x+1$, i.e. 
    \begin{equation}
        \begin{aligned}
            &|v(x)-v(x+1)|^p+|v(-x)-v(-x-1)|^p\\
            =&|\max\{v(x),v(-x)\}-\max\{v(x+1),v(-x-1)\}|^p\\
            &+|\min\{v(x),v(-x)\}-\min\{v(x+1),v(-x-1)\}|^p.
        \end{aligned}
    \end{equation}
    By the strict supermodular of $G(s,t)=-|s-t|^p$ for $p>1$, we have $v(x+1) > v(-x-1)$ if and only if $v(x)>v(-x)$ for all $x\in \N$.\\
    By a similar argument for $H=H_-$, we have $v(|x|)>v(|x|+1)$ for all $x\in \Z$. Then $v=v^*$, or $v=v^* \circ \sigma_{H_+}$, where $v^* \circ \sigma_{H_+}$ is the reflection of $v^*$ with respect to $0$.
\end{proof}

\normalem
\bibliographystyle{alpha}
\bibliography{reference}
\end{document}